\documentclass[aop,MSNbibl,seceqn,dvips]{arximspdf}
\usepackage{mathrsfs}
\usepackage{stfloats}
\usepackage{graphicx}

\doi{10.1214/13-AOP836} 
\volume{42}
\issue{4}
\pubyear{2014}
\firstpage{1516}
\lastpage{1589}

\makeatletter
\fnbelowfloat
\renewcommand{\mid}{|}
\newcommand{\mix}{\mathrm{MIX}}
\newcommand{\all}{\mathrm{all}}
\newcommand{\rrvert}{\vert}
\newcommand{\llvert}{\vert}
\newproclaim{quatation}[theorem]{Reader's Guide}
\renewcommand{\restriction}{\mathord{\upharpoonright}}
\newcommand{\tmix}{T_{\mathrm{MIX}}}
\newcommand{\trel}{T_{\mathrm{REL}}}
\newcommand{\gap}{\mathtt{gap}}
\newtheorem{theorem}{Theorem}[section]
\newtheorem{maintheorem}{Theorem}
\newtheorem{lemma}[theorem]{Lemma}
\newtheorem{proposition}[theorem]{Proposition}
\newproclaim{remark}[theorem]{Remark}
\newproclaim{remarks}[theorem]{Remarks}
\newtheorem{claim}[theorem]{Claim}
\newproclaim{definition}[theorem]{Definition}
\newproclaim{notation}[theorem]{Notation}
\newcommand{\Dim}{\mathrm{D}}
\newcommand{\betaR}{\beta_{\mathrm{R}}}
\makeatother

\begin{document}
\begin{frontmatter}

\title{Dynamics of $(2+1)$-dimensional SOS surfaces above
a~wall: Slow mixing induced by entropic repulsion\thanksref{T1}}
\runtitle{Dynamics of SOS surfaces above a wall}

\begin{aug}
\author[A]{\fnms{Pietro} \snm{Caputo}\ead[label=e1]{caputo@mat.uniroma3.it}},
\author[B]{\fnms{Eyal} \snm{Lubetzky}\corref{}\ead[label=e2]{eyal@microsoft.com}},
\author[A]{\fnms{Fabio} \snm{Martinelli}\ead[label=e3]{martin@mat.uniroma3.it}},\break
\author[C]{\fnms{Allan} \snm{Sly}\ead[label=e4]{sly@stat.berkeley.edu}}
\and
\author[D]{\fnms{Fabio Lucio} \snm{Toninelli}\ead[label=e5]{toninelli@math.univ-lyon1.fr}\thanksref{t2}}
\runauthor{P. Caputo et al.}
\affiliation{Universit\`a Roma Tre, Microsoft Research, Universit\`a
Roma Tre,\break University of California and CNRS and Universit\'e Lyon 1}
\address[A]{P. Caputo\\
F. Martinelli\\
Dipartimento di Matematica e Fisica\\
Universit\`a Roma Tre\\
Largo S. Murialdo 1\\
00146 Roma\\
Italia\\
\printead{e1}\\
\phantom{E-mail: }\printead*{e3}} 
\address[B]{E. Lubetzky\\
Microsoft Research\\
One Microsoft Way\\
Redmond, Washington 98052-6399\\
USA\\
\printead{e2}}
\address[C]{A. Sly\\
Department of Statistics\\
University of California, Berkeley\\
Berkeley, California 94720\\
USA\\
\printead{e4}}
\address[D]{F. L. Toninelli\\
CNRS and Universit\'e Lyon 1\\
Institut Camille Jordan\\
43 bd du 11 novembre 1918\\
69622 Villeurbanne\\
France\\
\printead{e5}\hspace*{8pt}}
\end{aug}
\thankstext{T1}{Supported by the European Research Council through the
``Advanced Grant'' PTRELSS 228032.}
\thankstext{t2}{Supported in part by ANR grant SHEPI.}

\received{\smonth{6} \syear{2012}}
\revised{\smonth{1} \syear{2013}}

%
\begin{abstract}
We study the Glauber dynamics for the $(2+1)\mathrm{D}$
Solid-On-Solid model above a hard wall and below a far away ceiling, on
an $L\times L$ box of $\mathbb{Z}^2$ with zero boundary conditions, at
large inverse-temperature $\beta$. It was shown by Bricmont,
El~Mellouki and Fr{\"o}hlich [\textit{J. Stat. Phys.} \textbf{42}
(1986) 743--798] that the floor constraint induces an entropic
repulsion effect which lifts the surface to an average height $H
\asymp(1/\beta)\log L$. As an essential step in understanding the
effect of entropic repulsion on the Glauber dynamics we determine the
equilibrium height $H$ to within an additive constant:
$H=(1/4\beta)\log L+O(1)$. We then show that starting from zero initial
conditions the surface rises to its final height $H$ through a sequence
of metastable transitions between consecutive levels. The time for a
transition from height $h=aH$, $a\in(0,1)$, to height $h+1$
is roughly $\exp(cL^a)$ for some constant $c>0$. 
In particular, the mixing time of the dynamics is exponentially large
in $L$, that is, $T_{\mathrm{MIX}}\geq e^{cL}$. We also provide the
matching upper bound $T_{\mathrm{MIX}}\leq e^{c'L}$, requiring a
challenging analysis of the statistics of height contours at low
temperature and new coupling ideas and techniques. Finally, to
emphasize the role of entropic repulsion we show that without a floor
constraint at height zero the mixing time is no longer exponentially
large in $L$.
\end{abstract}

%
\begin{keyword}[class=AMS]
\kwd{60K35}
\kwd{82C20}
\end{keyword}
\begin{keyword}
\kwd{SOS model}
\kwd{Glauber dynamics}
\kwd{random surface models}
\kwd{mixing times}
\end{keyword}

\end{frontmatter}

\setcounter{footnote}{2}
\section{Introduction}\label{sec-intro}

The $(d+1)$-dimensional \textit{Solid-On-Solid} model is a crystal
surface model whose definition goes back to Temperley \cite{Temperley}
in 1952 (also known as the Onsager-Temperley sheet). Its configuration
space on a finite box $\Lambda\subset\mathbb{Z}^d$ with a floor (wall)
at $0$, a ceiling at some $n^+$ and zero boundary conditions is the set
$\Omega_{\Lambda,n^+}$ of all height functions $\eta$ on $\mathbb
{Z}^d$ such that $\Lambda\ni x \mapsto\eta_x\in\{0,1,\ldots,n^+\}$
whereas $\eta_x=0$ for all $x\notin\Lambda$. The probability of
$\eta\in\Omega_{\Lambda,n^+}$ is given by the Gibbs distribution
%
%
\begin{equation}
\label{eq-ASOS} \pi_\Lambda(\eta) = \frac{1}{Z_\Lambda} \exp \biggl(-\beta
\sum_{x\sim y}|\eta_x-\eta_y|
\biggr),
\end{equation}
where $\beta>0$ is the inverse-temperature, $x\sim y$ denotes a
nearest-neighbor bond in the lattice $\mathbb{Z}^d$ and the normalizing
constant $Z_\Lambda$ is the partition function.

Numerous works have studied the rich random surface phenomena, for
example, roughening, localization/delocalization, layering and wetting
to name but a few, exhibited by the SOS model and some of its many
variants. These include the \textit{discrete Gaussian} (replacing
$|\eta_x-\eta_y|$ by $|\eta_x-\eta_y|^2$ for the integer analogue of
the Gaussian free field), \textit{restricted} SOS (nearest neighbor
gradients restricted to $\{0,\pm1\}$), \textit{body centered} SOS
\cite{vanBeijeren2}, etc. (for more on these flavors see, e.g.,
\mbox{\cite{Abraham,Baxter,Bolt2}}).

Of special importance is SOS with $d=2$, the only dimension featuring a
\textit{roughening transition}. Consider the SOS model without
constraining walls (the height function $\eta$ takes values in
$\mathbb{Z}$). For $d=1$, it is well known
\cite{Temperley,Temperley56,Fisher} that the SOS surface is
\textit{rough} (delocalized) for any $\beta>0$, that is, the expected
height at the origin (in absolute value) diverges in the thermodynamic
limit $|\Lambda|\to\infty$. However, for $d\geq3$ a Peierls argument
shows that the surface is \textit{rigid} (localized) for any $\beta>0$
(see \cite{BFL}), that is, $|\eta_0|$ is uniformly bounded in
expectation. This is also the case for $d=2$ and large enough $\beta$
\cite{BW,GMM}. That the surface is rough for $d=2$ at high
temperatures was established in seminal works of Fr{\"o}hlich and
Spencer \cite{FS1,FS2,FS3}. Numerical estimates for the critical
inverse-temperature $\betaR$ where the roughening transition takes
place suggest that $\betaR\approx0.806$.

One of the main motivations
for studying an SOS surface constrained between two walls, both its
statics and its dynamics, stems from its correspondence with the Ising
model in the phase coexistence region. For concreteness, take a box of
side-length $L$ in $\mathbb{Z}^3$ with minus boundary conditions on the
bottom face and plus elsewhere. One can view the $(2+1)\Dim$ SOS
surface taking values in $\{0,\ldots,L\}$ as the interface of the minus
component incident to the bottom face, in which case the Hamiltonian
in~(\ref{eq-ASOS}) agrees with that of Ising up to bubbles in the bulk.
At low enough temperatures bubbles and interface overhangs are
microscopic, thus SOS should give a qualitatively correct approximation
of Ising (see~\cite{ABSZ,Fisher,PV}). Indeed, in line with the
$(2+1)\Dim$ SOS picture, it is known \cite{vanBeijeren1} that the
$3\Dim$~Ising model undergoes a roughening transition at some
$\betaR^{\mathrm{IS}}$ satisfying $\beta_c(3) \leq\betaR^{\mathrm{IS}}
\leq\beta_c(2)$ [where $\beta_c(d)$ is the critical point for Ising on
$\mathbb{Z}^d$], yet there is still no rigorous proof that
$\betaR^{\mathrm{IS}} > \beta_c(3)$ (see \cite{Abraham} for more
details).

When the $(2+1)\Dim$ SOS surface is constrained to stay above a hard
wall (or floor), Bricmont, El~Mellouki and Fr{\"o}hlich \cite{BEF}
showed in 1986 the appearance of the \textit{entropic repulsion}: for
large enough $\beta$, the floor pushes the SOS surface to diverge even
though $\beta> \betaR$. More precisely, using Pirogov--Sina\"{\i}
theory (see the review \cite{Sinai}), the authors of \cite{BEF} showed
that the SOS surface on an $L\times L$ box rises, amid the penalizing
zero boundary, to an average height $H(L)$ satisfying $(1/C\beta)\log
L\leq H(L) \leq(C/\beta)\log L$ for some absolute constant $C>0$, in
favor of freedom to create spikes downwards.

Entropic repulsion is one of the key features of the physics of random
surfaces. This phenomenon has been rigorously analyzed mainly for some
continuous-height variants of the SOS model in which the interaction
potential $|\eta_x-\eta_y|$ is replaced by a \textit{convex} potential
$V(\eta_x-\eta_y)$; see, for example,
\cite{BDZ,BDG,DG,Bolt,Vel1,Vel2}, see also \cite{ADM} for a recent
analysis of the wetting transition in the SOS model. As we will see
below, entropic repulsion has a profound impact not only on the
\textit{equilibrium shape} of the surface but also on its \textit{time
evolution} under natural Markovian dynamics for the interface. The
rigorous analysis of these dynamical effects of entropic repulsion will
be the central focus of this work.

The dynamics we consider is the heat bath dynamics, or Gibbs sampler,
for the equilibrium measure $\pi_\Lambda$, that is, the discrete time
Markov chain where at each step a site $x\in\Lambda$ is picked at
random and the height $\eta_x$ of the surface at $x$ is replaced by a
random variable $\eta'_x\in\{0,\ldots,n^+\}$ distributed according to
the conditional probability $\pi_\Lambda(\cdot|\eta_y, y\neq x)$.
This defines a Markov chain with state space $\Omega_{\Lambda,n^+}$,
reversible with respect to $\pi_\Lambda$, commonly referred to as the
Glauber dynamics. As explained below, our results apply equally well to
other standard choices of reversible Markov chains, such as, for
example, the Metropolis chain where only moves of the type
$\eta'_x=\eta_x\pm1$ are allowed.

The mixing time $\tmix$ is defined as the number of steps needed to
reach approximate stationarity with respect to total variation
distance, see Section~\ref{sec-defs} for definitions.

The main result of this paper is that
the mixing of Glauber dynamics for the $(2+1)\Dim$ SOS is exponentially
slow, due to the nature of the entropic repulsion effect.

%
%
\begin{maintheorem}
\label{th-principale} For any sufficiently large inverse-temperature
$\beta$ there is some $c(\beta)>0$ such that the following holds for
all $L\in\mathbb{N}$. The mixing time $\tmix$ of the Glauber dynamics
of the $(2+1)\Dim$ SOS model on $\Lambda=\{1,\ldots,L\}^2$ with zero
boundary conditions, floor at zero and ceiling at $n^+$ with $\log
L\leq n^+\leq L$ satisfies
%
%
\begin{equation}
\label{eq-tmix-principale} e^{c L}\leq\tmix\leq e^{(1/c) L}.
\end{equation}
\end{maintheorem}

The exponentially large mixing time in (\ref{eq-tmix-principale}) is in
striking contrast with the rapid mixing displayed by Glauber dynamics
of the $(1+1)\Dim$ SOS model \mbox{\cite{MS,CMT}}. When $d=1$ it is
known that the main driving effect is a mean-curvature motion which
induces a diffusive relaxation to equilibrium, with $\tmix$ of order
$L^2$ up to $\operatorname{poly}(\log L$) corrections. As we will see,
in $(2+1)\Dim$ instead the main mechanism behind equilibration is a
series of metastable transitions through an increasing series of
effective energy barriers caused by the entropic repulsion. This is
also in contrast with the behavior of related interface models with
continuous heights as, for example, in \cite{DN,FFNV}.
%
%

\subsection{Metastability and entropic repulsion}\label{sec-metastability}
Consider the evolution of an initially flat surface at height zero. We
shall give a rough description of how it rises to the final height
$H(L)$ through a series of metastable states indexed by $h \geq0$.
Roughly speaking the surface in state with label $h$ is approximately
flat at height $h$ with rare up or downward spikes. Of course downward
spikes cannot be longer than $h$ because of the hard wall. If $h<H(L)$
then the surface has an advantage to rise to the next level $h+1$. This
is due to the gain in entropy, measured by the possibility of having
\textit{downward} spikes of length $h+1$, beating the energy loss from
the zero boundary conditions.

The mechanism for jumping to the next level should then be very similar
to that occurring in the $2\Dim$ Ising model at low temperature with a
small external field opposite to the boundary conditions (see
\cite{SC,SS}). Specifically, via a large deviation the surface at
height $h$ creates a large enough droplet of sites at height $h+1$
which afterwards expands to cover most of the available area. The
energy/entropy balance of any such droplet is roughly\footnote{Here we
are neglecting finer results taking into account the surface tension
and the associated Wulff theory; the basic conclusions of this
reasoning are nevertheless still valid.} of order $\beta|\gamma|-
e^{-4\beta(h+1)}A(\gamma)$ where $|\gamma|$ and $A(\gamma)$ are the
boundary length and area, respectively, and the effective field
$e^{-4\beta(h+1)}$ represents the probability of a $1\times1\times
(h+1)$ isolated downward spike. Simple considerations suggest then that
the critical length of a droplet should be proportional to $ e^{4\beta
(h+1)}$. Finally, the well-established metastability theory for the
$2\Dim$ Ising model indicates that the activation time $T_h$ for such a
critical droplet should be exponential in the critical
length\footnote{At the early stages of the process when $h$ is quite
small the activation time has important corrections to this guess due
to the many locations in the $L\times L$ box where the droplet can
appear. However, as soon as $h$ becomes of order $\log\log L$ these
entropic corrections become negligible.} (i.e., a double exponential in
$h$) as seen in Figure~\ref{fig:meta}.

Of course, in order to establish, even partially, the above picture and
to prove the asymptotic of $\log(\tmix)$ as per
(\ref{eq-tmix-principale}) it is imperative to estimate the final
equilibrium height of the surface $H(L)$ to within an \textit{additive}
$O(1)$. In Section~\ref{sec-eq} (Theorem~\ref{thm-equil-shape}), we
improve the estimates of \cite{BEF} to show that in fact the typical
height of the surface at equilibrium is $H(L)+O(1)$, where
%
%
\begin{equation}
\label{eq-hL} H(L)= \biggl\lfloor\frac1{4\beta}\log L \biggr\rfloor.
\end{equation}
%

The aforementioned picture of the evolution of the SOS surface through
a series of metastable states is quantified by the following result.

\begin{figure}

\includegraphics{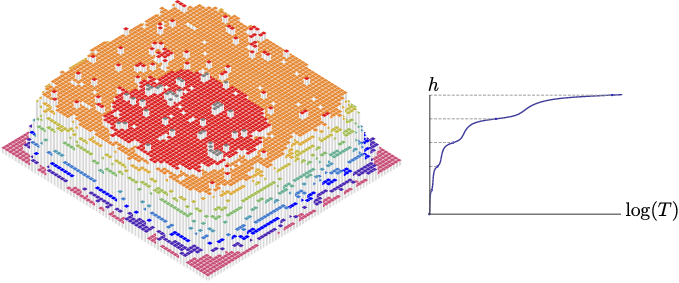}

\caption{Illustration of the series of metastable states in the
surface evolution. The dynamics waits time $e^{c \exp(4\beta h)}$
until the formation of a macroscopic droplet (marked in red) which
eventually raises the average height from $h-1$ to $h$.}\label{fig:meta}
\end{figure}

%
%
\begin{maintheorem}
\label{th-cascade} For any sufficiently large inverse-temperature
$\beta$ there is some $c(\beta)>0$ such that the following holds. Let
$(\eta(t))_{t\geq0}$ be the Glauber dynamics for the SOS model on
$\Lambda=\{1,\ldots,L\}^2$ with zero boundary conditions, floor at zero
and ceiling at $n^+$ with $\log L\leq n^+\leq L$, started from the
all-zero initial state. Fix $a\in(0,1)$ and let $\tau_a = \min\{ t\dvtx
\eta(t) \in\Omega_a\}$ where
%
%
\begin{equation}
\label{eq-Omega-a-def} \Omega_a= \bigl\{\eta\in
\Omega_{\Lambda,n^+}\dvtx\# \bigl\{x\dvtx\eta_x\geq a H(L) \bigr\} >
\tfrac9{10}|\Lambda| \bigr\}.
\end{equation}
Then $\lim_{L\to\infty}\pi_\Lambda(\Omega_a) = 1$ and yet
%
%
\begin{equation}
\label{eq-cascade} \lim_{L\to\infty} \mathbb{P}
\bigl(e^{c L^a} \leq\tau_a \leq e^{(1/c)L^a} \bigr)=1.
\end{equation}
\end{maintheorem}

In fact, we prove this with the constant $\frac9{10}$
in~(\ref{eq-Omega-a-def}) replaced by $1-\varepsilon(\beta)$ where
$\lim_{\beta\to\infty}\varepsilon(\beta)= 0$. Moreover, the statement
of the above theorem remains valid when $a=a(L)\to1$ as long as the
target level $h=a H(L)$ satisfies $h \leq H(L) - c$ for some
sufficiently large $c(\beta)>0$.

\begin{remark*}
A natural conjecture in light of Theorem~\ref{th-cascade} is that there
exists a~constant $\lambda$ such that the distribution of $\tau
_a\times e^{-\lambda L^a}$ converges as $L\to\infty$ to an exponential
random variable.
\end{remark*}

We wish to emphasize that, as will emerge from the proof, the
exponential slowdown of equilibration is a coupled effect of entropic
repulsion \textit{and} of the rigidity of the interface. In particular,
the following rough upper bound shows that the situation is very much
different when the floor constraint is absent (yet the ceiling
constraint remains unchanged).

%
\begin{maintheorem}\label{th-nomuro} Consider the $(2+1)\Dim$ SOS setting as in
Theorem~\ref{th-principale} with the exception that the surface heights
belong to the \textit{symmetric} interval\break
$[-n^+,n^+]$. 
Then $\tmix\leq\exp(o(L))$.
\end{maintheorem}

Specifically, our proof gives the estimate $\tmix\leq
\exp(L^{(1/2)+o(1)})$. No effort was made to improve the exponent
$\frac12$ as we would expect the true mixing behavior to be polynomial
in $L$. We further expect that in the presence of a floor yet for
$\beta< \betaR$ the mixing time will have a different scaling with the
side-length $L$.

It is useful to compare our results with those of \cite{CM}, where the
Glauber dynamics for the $(2+1)\Dim$ SOS above a hard wall, at low
temperature and in the presence of a weak attracting (towards the wall)
external field was analyzed in details. There it was proved that
certain critical values of the external field induce exponentially slow
mixing while for all other values the dynamics is rapidly mixing.
Although the slow mixing proved in \cite{CM} is similar to the one
appearing in (\ref{eq-tmix-principale}), the physical phenomenon behind
it is very different. When an external field is present, a~critical
value of it results in two possible and roughly equally likely heights
for the surface. In this case, slow mixing arises because of the
presence of a typical bottleneck in the phase space related to the
bi-modal structure of the equilibrium distribution. In the setting of
Theorems~\ref{th-principale} and~\ref{th-cascade} instead, there is in
general no bi-modal structure of the Gibbs measure and the slow mixing
takes place because of a \textit{multi-valley} structure of the
effective energy landscape induced by the entropic repulsion which
produces a whole family of bottlenecks.

\subsection{Methods}\label{sec-methods}
We turn to a description of the main techniques involved in the proof
of the main theorems. Our results can be naturally divided into three
families: equilibrium estimates, lower bounds on equilibration times,
and upper bounds on equilibration times.

\subsubsection*{Equilibrium estimates}
Our proof begins by deriving estimates for the equilibrium distribution
which are crucial to the understanding of the dynamics (as discussed in
Section~\ref{sec-metastability}) and of independent interest. Over most
of the surface, the height is concentrated around $H(L)$ as defined
in~(\ref{eq-hL}) with typical fluctuations of constant size. Achieving
estimates with a precision level of an additive $O(1)$ turns out to be
essential for establishing the order of the mixing time exponent:
indeed, analogous estimates up to some additive $g(L)$ tending to
$\infty$ with $L$ would set off this exponent by a factor of
$e^{O(g)}$.

The main techniques deployed for this part are a range of Peierls-type
estimates for what we refer to as $h$-contours, defined as closed dual
circuits with values at least $h$ on the sites along their interior
boundary and at most $h-1$ along their exterior boundary. In the
simpler setting of no floor or ceiling (i.e., the sites are free to
take all values in $\mathbb{Z}$ as their heights), the map $S_\gamma$ which
decreases all sites inside an $h$-contour $\gamma$ by 1 is bijective
and increases the Hamiltonian by $|\gamma|$, the length of the contour.
Hence, the probability of a given $h$-contour in this setting is
bounded by $\exp(-\beta|\gamma|)$. Iterating estimates of this form
allows us to bound the deviations of the sites with the correct
asymptotic in the setup of having no walls.

The presence of a floor renders this basic Peierls argument invalid
since the map $S_\gamma$ may leave sites in the interior with negative
values. Rather than a technicality, this in fact lies at the heart of
the entropic repulsion effect. We resort to estimating the probability
that a given $h$-contour has a strictly positive interior, a quantity
directly involving its area. By analyzing an isoperimetric tradeoff
between the contour's area and perimeter, we show that large contours
above height $H(L)$ are unlikely, which in turn implies $O(1)$ typical
fluctuations above this level. For a lower bound on the typical height
of the surface we show that if too many sites are below $H(L) -k$ then
the loss in energy due to raising the entire surface by 1 is more than
compensated by the increased entropy from the freedom to create
downward spikes reaching 0. Put together, these estimates guarantee
that the height of most sites is within a constant of $H(L)$.

\subsubsection*{Equilibration times: Lower bounds}
Fix $h=aH(L)-1$ with $a\in(0,1)$ and consider the restricted ensemble
obtained by conditioning the equilibrium measure~on the event $A$ that
all $h$-contours have area smaller than $\delta L^{2a}$, for some small
$\delta>0$. Our equilibrium estimates imply that in this restricted
ensemble:
\begin{longlist}[(iii)]
\item[(i)] each $h$-contour is actually very small [e.g., with area
less than $\log(L)^2$], with very high probability;

\item[(ii)] the probability of the boundary of $A$ is $O[\exp(-c
L^a)]$;

\item[(iii)] the probability of having a large density of heights at
least $h+1=aH(L)$ is $O[\exp(-c L^a)]$.
\end{longlist}
In some sense (i), (ii) and (iii) above establish a bottleneck the
Markov chain must pass through and thus provide the sought lower bound
of $\exp(c(\beta)L^a)$ on the typical value of the hitting time
$\tau_a$ in Theorem~\ref{th-cascade} when the initial state is the all
zero configuration. In fact, the initially flat configuration can be
replaced by monotonicity by the restricted ensemble described above.
Then, in order for $\tau_a$ to be smaller than $T$, either the dynamics
has gone through the boundary of $A$ before $T$ or the event described
in (iii) occurred without leaving $A$. Either way an event with
$O(\exp(-c L^a))$ probability occurred and the minimal time to see it
must be proportional to the inverse of its probability.

\subsubsection*{Equilibration times: Upper bounds}
By the monotonicity of the system, it is enough to consider the chain
starting from the maximum and minimum configurations. The natural
approach is to apply the well-known canonical paths method
(see~\cite{DSa,DSt,JS,Sinclair} for various flavors of the method). As
the cut-width of the cube is $L^2$, the most na\"{\i}ve application of
this approach would give a bound of $\exp(O(L^2))$. A better bound can
be shown by considering the problem with maximum height $n^+= \log L$.
In this case, the cut-width is of order $L \log L$ yielding a mixing
time upper bound of $\exp(O(L \log L))$. Since the height fluctuations
are logarithmic, we can iterate this analysis using monotonicity and
censoring to get a bound of $\exp(O(L \log L))$ for the original model
with $n^+=L$, vs. our lower bound of $\exp(c L)$. However, removing the
$\log L$ factor that separates these exponents entails a significant
amount of extra work.

The basic structure of the proof is to first establish a burn-in phase
where we show that, starting from the maximal and minimal
configurations, the process reaches a ``good'' set 
featuring small deviations from the equilibrium level $H(L)$. From
there, we establish a modified canonical paths estimate
(Theorem~\ref{th-pathflowsgeneral}), showing that it is enough to
establish a reasonable probability of hitting the good set from any
starting location together with a good canonical paths estimate
restricted to this set. This new tool, which we believe is of interest
on its own right, is described in detail in
Section~\ref{sec-improved-canonical} and proved in a general context in
Section~\ref{sec-canonical-proof}.


Showing that the surface falls down from the ceiling (the maximum
height) to $H(L)$, as depicted in Figure~\ref{fig:dyn}, ought to have
been the easier part of the burn-in argument since high above the floor
there is no entropic repulsion effect. Unfortunately a number of major
technical challenges must be overcome.
%
%
\begin{figure}

\includegraphics{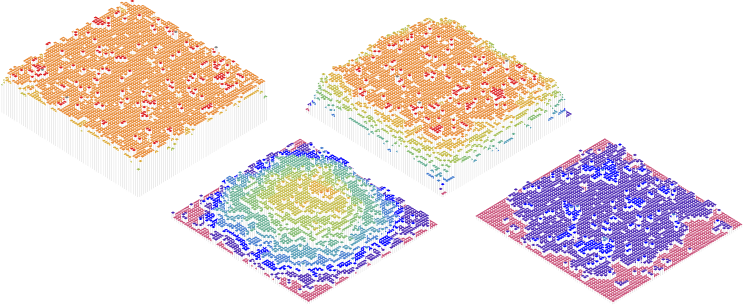}

\caption{Glauber dynamics for SOS on a $64 \times64$ square lattice
at $\beta=0.9$ from an initial state $\eta\equiv10$.
Surface gradually falls towards level $H=1$. Snapshots at $t=10$ (top
left), $t=100$, $t=1000$, $t={}$10,000 (bottom right) in cont.
time.}\label{fig:dyn}
\end{figure}

First, the effect of the entropic repulsion is still apparent for the
estimates we require when the surface is fairly close to $H(L)$.
To overcome this, we add a small external field to the model, thereby
modifying the mixing time by a factor of at most $\exp(O(L))$ (which is
large but still of the same order as our designated upper bound) and
tilting the measure to remove these entropic repulsion effects. Second,
while our main equilibrium estimates were proved using Peierls-type
estimates, for the burn-in we require some of the cluster expansion
machinery of \cite{DKS} which we extend to the SOS framework. This
involves a number of challenges including showing that the contours we
consider do not interact significantly with the boundary conditions, a
highly nontrivial fact. Implementing this scheme is the biggest
challenge of the paper and we provide extensive notes for the reader in
these sections to explain the rather technical proofs.

Finally, the fact that the surface rises from the floor (the all zero
initial condition) to the vicinity of the equilibrium height $H(L)$ in
time $\exp(O(L))$ is proved via an unusual inductive scheme. Unlike
other multi-scale inductive schemes, somewhat surprisingly the one used
here does not incur any penalizing factor on the upper bound. We first
prove weaker bounds on the mixing time and use these estimates to show
that a smaller box of side-length $L/\log L$ mixes by time
$\exp(O(L))$. By monotonicity, we can use this to bound the distance
from the equilibrium height of the surface in the original box by
$H(L)-H(L/\log L)$.
By using this height estimate along with our canonical paths result, we
get improved bounds on the mixing time. This in turn allows us to take
larger sub-boxes and iteratively achieve better and better estimates on
the distance to $H(L)$. After sufficiently many iterations, we show
that the surface reaches height $H(L)-O(1)$ in time $\exp(O(L))$ and
thereafter the canonical paths estimate completes the proof.

\subsection{Related open problems}

\subsubsection*{Tilted walls}
An interesting and to our knowledge widely open problem concerns the
SOS model with a nonhorizontal hard wall, that is, when the constraint
$\eta_x\geq0$ is replaced by $\eta_x\geq\phi_x^{\mathbf n}$, where
$\phi_x^{\mathbf n}$ denotes the discrete approximation of the plane
orthogonal to the unit vector ${\mathbf n}$, and ${\mathbf n}$ is
assumed to have all components different from zero.
The equilibrium fluctuations for $\beta=+\infty$ can be analyzed via
their representation through dimer coverings \cite{KOS} and the
variance of the surface height in the middle of the box can be shown to
be $O(\log L)$; see \cite{CMST}, Section~5, for a proof. Moreover, at
$\beta=+\infty$, as far as the dynamics is concerned, it has been
proved \cite{CMT} that the mixing time is of order $L^2$ up to
$\mathrm{polylog}(L)$ corrections and that the relaxation process is
driven by mean curvature motion. The case $\beta<+\infty$, however,
remains open both for equilibrium fluctuations and for mixing time
bounds.

\subsubsection*{Mixing time for Ising model}

In view of the natural connection with the Ising model, the study of
Glauber dynamics for the SOS can also shed some light on a, still open,
central problem in the theory of stochastic Ising models: its mixing
time under an \textit{all-plus} boundary in the phase coexistence
region. 
The long-standing conjecture is that the mixing time of Glauber
dynamics for the Ising model
on a box of side-length $L$ with all-plus boundary should be at most
polynomial in $L$ at \textit{any temperature}. More precisely, the
convergence to equilibrium should be driven by a mean-curvature motion
of the interface of the minus droplet in accordance with Lifshitz's law
\cite{Lifshitz}. For instance, the mixing time of Glauber dynamics for
Ising on an $L\times L$ square lattice is conjectured \cite{FH} to be
of order $L^2$ in continuous time. This was confirmed at zero
temperature~\cite{CSS,FSS,LST} and near-zero temperatures~\cite{CMST},
yet the best-known upper bound for finite $\beta>\beta_c$ remains quite
far, a~quasi-polynomial bound of $L^{O(\log L)}$ due to \cite{LMST}.
The understanding of $3\Dim$ Ising is far more limited: while at zero
temperature bounds of $L^{2+o(1)}$ were recently proven in \cite{CMST},
no sub-exponential mixing bounds are known at any finite
$\beta>\beta_c$.

\section{Definitions and tools}\label{sec-defs}

\subsection{Glauber dynamics for solid-on-solid}\label{sec-model}
Let $\sqcup$ and $\sqcap$ denote the minimal and maximal configurations
in $\Omega_{\Lambda,n^+}$, that is, $\sqcup_x=0$ and $\sqcap_x=n^+$ for
every \mbox{$x\in\Lambda$}.
Given a finite connected subset $\Lambda\subset\mathbb{Z}^2$, let
$\partial\Lambda$ denote its external boundary, that is, the set of
sites in $\Lambda^c$ which are at distance $1$ from $\Lambda$. To
extend the SOS definition to arbitrary boundary conditions (b.c.) given
by $\xi\dvtx \mathbb{Z}^2 \to\mathbb{Z}$, define the SOS Hamiltonian
with b.c. $\xi$ to be
%
%
\begin{equation}
\label{eq-Hamiltonian} \mathcal{H}_\Lambda^\xi(\eta):=
\frac12 \mathop{\sum_{x,y\in\Lambda}}_{|x-y|=1}|
\eta_x-\eta_y| +\mathop{\sum
_{x\in\Lambda, y\in
\partial\Lambda}}_{|x-y|=1}|\eta_x-
\xi_y|.
\end{equation}
Given $\beta>0$ and $n^+$, the Gibbs measure $\pi^\xi_{\Lambda}$ on
$\Omega_{\Lambda,n^+}$ with b.c. $\xi$ is defined as
%
%
\begin{equation}
\label{eq-Gibbs} \pi^\xi_{\Lambda}(\eta)=\frac1{Z^\xi_\Lambda}
\exp \bigl[-\beta\mathcal{H}^\xi_\Lambda(\eta) \bigr].
\end{equation}
%

%
%
\begin{notation}\label{notation}
In the sequel when the b.c. $\xi\equiv n\in\mathbb{Z}$ we will use the
abbreviated form $\pi_\Lambda^n$. We will occasionally drop the
subscript $\Lambda$ and superscript $\xi$ from the notation of
$\pi_\Lambda^\xi$ when there is no risk of confusion. Moreover, we will
need to address the following variants of $\pi_\Lambda^\xi$:
\begin{longlist}[(iii)]
\item[(i)] the measure $\hat\pi{}^n_\Lambda$ of SOS without walls (no
floor and no ceiling) and with b.c. at height $n$;

\item[(ii)] the measure $\Pi^\xi_\Lambda$ corresponding to $\pi_\Lambda
^\xi$ with $n^+=+\infty$ (no ceiling);

\item[(iii)] starting from Section~\ref{sec-mtub} the measures
$\pi_\Lambda^{\xi,f}$ (and its analog $\Pi_\Lambda^{\xi,f}$ with no
ceiling) corresponding to the SOS Hamiltonian with an additional
external field of the form $\frac1L \sum_{y\in \Lambda}f(\eta_y)$
with $|f|_{\infty}=O(e^{-c\beta})$ for some fixed constant $c$
[see, e.g.,~(\ref{eq-10})].
\end{longlist}
\end{notation}

The dynamics under consideration is a discrete-time Markov chain\break
$(\eta(t))_{t=0,1,\ldots}$, defined as follows.
%
To construct $\eta(t+1)$ given $\eta(t)$,
\begin{itemize}
\item pick a site $x\in\Lambda$ uniformly at random;

\item sample a new value for $\eta_x(t+1)$ from the equilibrium
measure $\pi_\Lambda^\xi$ conditioned on the current heights at
the neighboring sites, that is, $\eta(t+1)\sim
\pi^\xi_\Lambda(\eta\in\cdot\mid\eta_y=\eta_y(t)\ \forall y\neq
x)$.
\end{itemize}
The law of the process with initial condition $\zeta$ is denoted by
$\mathbb{P}^\zeta$, the configuration at time $t$ is $\eta^\zeta (t)$
and its law is $\mu^{\zeta}_t$. When there is no need to emphasize the
initial condition, we simply write $\eta(t)$ for the configuration at
time $t$. It is well known that this Markov chain 
is reversible w.r.t. the invariant measure $\pi^\xi_\Lambda$.

The mixing time $\tmix$ is defined to be the time the process takes to
converge to equilibrium in total variation distance, that is,
%
%
\begin{eqnarray}
\label{eq-tmix} \tmix&=&\inf \biggl\{t>0\dvtx\max_{\eta\in\Omega
_{\Lambda,n^+}} \bigl\|
\mu_t^\eta-\pi^\xi_\Lambda \bigr\|\leq
\frac1{2e} \biggr\},
\end{eqnarray}
where $\|\mu-\nu\|$ denotes the total variation distance between two
measures $\mu,\nu$. It is well known (e.g., \cite{LPW}, Section~4.5)
that the total variation distance from equilibrium decays exponentially
with rate $\tmix$, namely
%
%
\begin{equation}
\label{eq-sottomol} \max_{\eta\in\Omega_{\Lambda,n^+}} \bigl\|\mu _t^\eta-
\pi^\xi_\Lambda \bigr\|\leq e^{-\lfloor t/\tmix\rfloor}.
\end{equation}
The relaxation time $\trel$ is the inverse of the spectral gap of the
transition kernel of the chain. The spectral gap, denoted by $\gap$,
has the following variational characterization:
%
%
\begin{eqnarray}
\label{eq-gap} \gap&=&\inf\frac{\pi^\xi_\Lambda
(f(I-P)f)}{\operatorname{Var}_{\pi^\xi_\Lambda}(f)},
\end{eqnarray}
where $P$ is the transition kernel of the chain, $I$ is the identity
matrix and the infimum is over all nonconstant functions $f$. The
following standard inequality (see, e.g., \cite{LPW}, Section~12.2, and
\cite{SaloffCoste}) relates the mixing time and the relaxation time:
%
%
\begin{equation}
\label{eq-20} \trel-1\leq\tmix\leq\trel\log(2e/\pi_{\min})
\end{equation}
with $\pi_{\min}:=\min_{\eta\in\Omega_{\Lambda,n^+}}\pi^\xi
_{\Lambda}(\eta)$. By definition, in the SOS model
$|\Omega_{\Lambda,n^+}|=(n^+ + 1)^{|\Lambda|}$ and
$\pi_{\min}\geq\exp(-4\beta|\Lambda|n^+)/|\Omega_{\Lambda,n^+}|$, thus
for large enough $n^+$
%
%
\begin{equation}
\label{eq-20bis} \trel-1\leq\tmix\leq5\beta|\Lambda|n^+ \trel.
\end{equation}

From now on we refer to the Markov chain defined above as the Glauber
dynamics. One can use standard comparison estimates to obtain
equivalent versions of our main results for other standard choices of
Markov chains that are reversible w.r.t. the SOS Gibbs measures, such
as, for example, the Metropolis chain with $\pm$1 updates.
Indeed, since the heights are confined within an interval of size
$O(L)$ it is not hard to see that the ratio between the different
mixing times is at most polynomial in $L$. We refer to, for example,
\cite{CMST}, Section~6, for a detailed argument in this direction.

\subsection{Monotonicity}\label{sec-mon}
Our dynamics is monotone (or attractive) in the following sense. One
equips the configuration space with the natural partial order such that
$\sigma\leq\eta$ if $\sigma_x\leq\eta_x$ for every $x\in\Lambda
$. It is
possible to couple on the same probability space the evolutions
corresponding to every possible initial condition $\zeta$ and boundary
condition $\xi$ in such a way that if $\xi\leq\xi'$ and
$\zeta\leq\zeta'$ then $\eta^{\zeta}(t,\xi)\leq\eta^{\zeta
'}(t,\xi')$
for every $t$. Here, we indicated explicitly the dependence on the
boundary conditions but we will not do so in the following. The law of
the \textit{global monotone coupling} is denoted $\mathbb{P}$.

A first consequence of monotonicity is that the FKG inequalities
\cite{FKG} hold: if $f$ and $g$ are two increasing (w.r.t. the above
partial ordering) functions, then
$\pi^\xi_\Lambda(fg)\geq\pi^\xi_\Lambda(f) \pi^\xi_\Lambda (g)$ and the
same holds for the measure $\hat\pi{}^\xi_\Lambda$ without the
floor/ceiling.

Monotonicity also implies the following standard fact [cf., e.g., the
proof of~\cite{MT}, equation~(2.10)]: for every initial condition
$\eta$ and boundary condition~$\xi$,
%
%
\begin{equation}
\label{eq-22} \bigl\|\mu_t^\eta-\pi_\Lambda^\xi
\bigr\|\leq2 n^+|\Lambda|\max \bigl( \bigl\|\mu_t^\sqcup-
\pi_\Lambda^\xi \bigr\|, \bigl\|\mu_t^\sqcap-
\pi_\Lambda^\xi \bigr\| \bigr).
\end{equation}

Another consequence of monotonicity is the so-called Peres--Winkler
censoring inequality. Take integers $0=t_0<t_1<\cdots<t_k=T$, a
sequence of $V_i\subset\Lambda$ and $0\leq a_i\leq b_i\leq n^+, i\leq
k$. Consider the following modified dynamics $(\tilde\eta(t))_{0\leq
t\leq T}$. To construct $\tilde\eta(t+1)$ given $\tilde\eta(t)$,
\begin{itemize}
\item pick a site $x\in\Lambda$ uniformly at random;

\item at time $t$ with $t_{i-1}<t\leq t_{i}$ do as follows:
\begin{itemize}
\item if $x\notin V_i$ or if $x\in V_i$ and
$\tilde\eta_x(t)\notin\{a_i,\ldots,b_i\}$ then do nothing;

\item if $x\in V_i$ and $a_i\leq\tilde\eta_x(t)\leq b_i$ then
replace its value with a new value $\tilde\eta_x(t+1)$ in
$\{a_i,\ldots,b_i\}$ with probability proportional to the
stationary measure conditioned on the value of the
neighboring columns,
\[
\tilde\eta(t+1)\sim\pi^\xi_\Lambda \bigl(\eta\in\cdot\mid
\eta_x\in\{a_i,\ldots,b_i\},
\eta_y=\tilde\eta_y(t)\ \forall y\neq x \bigr).
\]
\end{itemize}
\end{itemize}
Call $\tilde\mu{}^\nu_t$ the law at time $t$ when the initial
distribution is $\nu$. The following then holds:

%
\begin{theorem}[(Special case of \cite{PW}, Theorem 1.1)]\label{th-PW}
If the initial distribution $\nu$ is such that $\nu(\eta)/\pi _\Lambda
^\xi(\eta)$ is an increasing (resp., decreasing) function, then
$\tilde\mu{}_t^\nu(\eta)/\pi_\Lambda^\xi(\eta)$ is also increasing
(resp., decreasing) for $t\leq T$ and $\mu_t^\nu\preceq\tilde\mu
_t^\nu$ (resp., $\tilde\mu{}_t^\nu\preceq\mu_t^\nu$). In addition,
%
%
\begin{equation}
\label{eq-21} \bigl\|\mu_t^\nu-\pi_\Lambda^\xi
\bigr\|\leq \bigl\|\tilde\mu{}_t^\nu-\pi_\Lambda
^\xi \bigr\|.
\end{equation}
\end{theorem}

\subsection{An improved path argument}\label{sec-improved-canonical}

Geometric techniques can prove very effective in getting upper bounds
on the relaxation time and therefore on the mixing time of a Markov
chain~\cite{DSa,DSt,JS,Sinclair} (see also~\cite{LPW}, Section~13.5).
Let us recall the basic principle. 

Let $(X(t))_{t=0,1,\ldots}$ be a discrete-time reversible Markov chain
on a finite state space $\Omega$, with invariant measure $\pi$. For
$a,b\in\Omega$ such that the one-step transition probability $p(a,b)$
from $a$ to $b$ is nonzero, set $Q(a,b)=\pi(a)p(a,b)=Q(b,a)$. For each
couple $(c,d)\in\Omega^2$, fix a path $\gamma(c,d)=(x_1,\ldots,\break x_n)$ in
$\Omega$ with $x_1=c$, $x_n=d$ and $p(x_i,x_{i+1})\ne0$ and let
$|\gamma(c,d)|:=n$. Then the relaxation time of the Markov chain is
bounded as
%
%
\begin{equation}
\label{eq-18bis} \trel\leq\max_{(a,b)\dvtx Q(a,b)\ne0} \frac 1{Q(a,b)}\mathop{
\sum_{\eta,\eta'\in\Omega\dvtx}}_{ (a,b)\in
\gamma(\eta,\eta')} \bigl|\gamma \bigl(\eta,
\eta' \bigr) \bigr|\pi(\eta)\pi \bigl(\eta' \bigr).
\end{equation}
Here, $(a,b)\in\gamma(\eta,\eta')$ means that if $\gamma(\eta,\eta')=
(x_1,\ldots,x_n)$ then there exists~$i$ such that $a=x_i$, $b=x_{i+1}$.
The proof is simply an application of the Cauchy--Schwarz inequality;
see, for example, \cite{SaloffCoste}.

An 
application of this principle gives the following proposition.

%
%
\begin{proposition}
\label{prop-canpaths} For the SOS dynamics in the
$\Lambda=\{1,\ldots,L\}\times\break \{1,\ldots,m\}$, $m\leq L$, with floor at
height zero, ceiling at $n^+$ and b.c. $\xi$, one has for some
\mbox{$c=c(\beta)$}
%
%
\begin{equation}
\label{eq-19} \trel\leq c L^2 m^2n^+ \exp \bigl(7\beta m
n^+ \bigr)
\end{equation}
and, thanks to (\ref{eq-20bis}), $\tmix=\exp(O(\beta L n^+))$ if $L=m$.
\end{proposition}

That (\ref{eq-19}) easily follows from (\ref{eq-18bis}) was observed in
\cite{Martinelli94} in the case of the Glauber dynamics of the Ising
model (in this case one refers to the paths $\gamma(\eta,\eta')$ as
``canonical paths''). For SOS the proof is very similar and is given
for completeness in Section~\ref{sec-canpaths}.

However, this upper bound is too rough for our purposes since we have
$n^+\geq\log L$ while we wish to get a mixing time upper bound which is
exponential in $L$. Therefore, a significant part of the present work
is devoted to getting rid of the nonphysical factor $n^+$ in the
argument of the exponential in the r.h.s. of (\ref{eq-19}). Although
this task may appear to be mainly of technical nature it actually
requires a much deeper understanding of the actual behavior of the
dynamics compared to that provided by canonical paths, and the support
of new ideas.

One of the key ingredients we use is the following improved version of
(\ref{eq-18bis}), which we believe can be interesting in a more general
context.

%
%
\begin{theorem}
\label{th-pathflowsgeneral} Let $G\subset\Omega$ and assume that, for
some $T>0$ and for every initial condition $x$, $\mathbb{P}^x(X(T)\in
G)\geq\alpha$ with $\mathbb{P}^x$ denoting the law of the chain
starting at $x$. Assume further that for every $\eta,\eta'$ in $G$
there exists a path $\tilde\gamma(\eta,\eta')$ as above which stays in
$G$ and let
%
%
\begin{equation}
\label{eq-W} W(G):=\mathop{\max_{a,b\in G}}_{Q(a,b)\ne0}
\frac1{Q(a,b)}\mathop{\sum_{\eta,\eta'\in G\dvtx}}_{(a,b)\in
\tilde
\gamma(\eta,\eta')} \bigl|
\tilde\gamma \bigl(\eta,\eta' \bigr) \bigr|\pi(\eta)\pi \bigl(
\eta' \bigr).
\end{equation}
Then,
%
%
\begin{equation}
\label{eq-pathsgen} \gap^{-1}\leq\frac6{\alpha} \biggl(
\frac{T^2}{ p_{\min}}+\frac{W(G)}\alpha \biggr)
\end{equation}
with 
$ p_{\min}:=\min\{p(\sigma,\sigma')>0\dvtx \sigma,\sigma'\in
\Omega
\}$.
\end{theorem}

This is clearly an improvement provided that $\alpha$ is bounded away
from zero, that $W(G)\ll W(\Omega)$ and that $T$ is not too large (in
simple words, we need that with nonzero probability the chain enters
``quickly'' the good set $G$ where canonical paths work well).

In our SOS application, roughly speaking, we will choose $G$ to be the
set of configurations such that
$|\Lambda_L|^{-1}\sum_{x\in\Lambda_L}|\eta_x-H(L)|$ is upper bounded by
a constant. We will see that, irrespective of the starting
configuration, at time $T=\exp(O(L))$ the dynamics is in $G$ with
probability at least $\frac12$. On the other hand, a minor modification
of Proposition~\ref{prop-canpaths} will give $W(G)=\exp(O(\beta L))$.
Then, Theorem~\ref{th-pathflowsgeneral} allows us to improve the mixing
time upper bound to $\tmix=\exp(O(\beta L))$.

\section{Equilibrium results}\label{sec-eq}

%
%
\begin{theorem}\label{thm-equil-shape} Let $\Lambda\subset\mathbb
{Z}^2$ be a box of side-length $L$ and let $\beta\geq1$. Set $H =
\lfloor\frac1{4\beta} \log L\rfloor$. There exist some absolute
constants $C,K>0$ (with $K$ integer) such that for any integer $k\geq
K$,
%
%
\begin{equation}
\label{eq-equilib-i} \pi^0_\Lambda \bigl(\# \{v\dvtx
\eta_v\leq H - k \} > e^{-2\beta k}L^2 \bigr) \leq
\exp \bigl(-e^{\beta k}L \bigr),
\end{equation}
\begin{eqnarray}
\label{eq-equilib-ii} && \pi^0_\Lambda \bigl(\# \{v\dvtx
\eta_v\geq H + k \} > e^{-2\beta k}L^2 \bigr)
\nonumber
\\[-8pt]
\\[-8pt]
&&\qquad \leq\exp \bigl(-C e^{-2\beta k}L \bigl(1 \wedge e^{-2\beta k} L
\log^{-8} L \bigr) \bigr).
\nonumber
\end{eqnarray}
\end{theorem}

(Notice that the bound on downward fluctuations improves with the size
of the deviation whereas the bound on upward fluctuations deteriorates
with the distance.)

Recall that $\pi_\Lambda^0$ has a floor at $0$ and a ceiling at height
$\log L\leq n^+\leq L$ (together with zero boundary conditions). It
will be convenient throughout this section to work in the setting of a
floor at 0 but no ceiling, where the corresponding measure
$\Pi_\Lambda^0$ is asymptotically equal to $\pi_\Lambda^0$.

%
%
\begin{lemma}
\label{lem-no-ceil} There is an absolute constant $c>0$ such that for
any $\beta\geq1$ and any subset of configurations $A \subseteq
\{0,\ldots,n^+\}^\Lambda$,
\[
\Pi^0_\Lambda(A) \leq\pi^0_\Lambda(A)
\leq \bigl(1 + c L^2 e^{-2\beta n^+} \bigr) \Pi^0_\Lambda(A).
\]
\end{lemma}

The above lemma, which will be proved further on in this section,
entitles us to derive results on $\pi_\Lambda^0$ from $\Pi_\Lambda ^0$
at an asymptotically negligible cost.

The following notion of a \textit{contour} and that of an
\textit{$h$-contour}, a level line at height~$h$, play a crucial role
in our proofs.

%
%
\begin{definition}\label{contourdef}
We let ${\mathbb{Z}^2}^*$ be the dual lattice of $\mathbb{Z}^2$ and we
call a \textit{bond} any segment joining two neighboring sites in
${\mathbb{Z}^2}^*$. Two sites $x,y$ in $\mathbb{Z}^2$ are said to be
\textit{separated by a bond $e$} if their distance\vspace*{-1pt} (in
$\mathbb{R}^2$) from $e$ is $\frac12$. A pair of orthogonal bonds which
meet in a site $x^*\in{\mathbb{Z}^2}^*$ is said to be a \textit{linked
pair of bonds} if both bonds are on the same side of the forty-five
degrees line across $x^*$. A \textit{geometric contour} (for short a
contour in the sequel) is a sequence $e_0,\ldots,e_n$ of bonds such
that:
\begin{enumerate}[(3)]
\item[(1)] $e_i\ne e_j$ for $i\ne j$, except for $i=0$ and $j=n$
where $e_0=e_n$;

\item[(2)] for every $i$, $e_i$ and $e_{i+1}$ have a common vertex
in ${\mathbb{Z}^2}^*$;

\item[(3)] if $e_i$, $e_{i+1}$, $e_j$, $e_{j+1}$ intersect at some
$x^*\in {\mathbb{Z}^2}^*$, then $e_i,e_{i+1}$ and $e_j,e_{j+1}$
are linked pairs of bonds.
\end{enumerate}
We denote the length of a contour $\gamma$ by $|\gamma|$, its interior
(the sites in $\mathbb{Z}^2$ it surrounds) by $\Lambda_\gamma$ and its
interior area (the\vadjust{\goodbreak} number of such sites) by $|\Lambda_\gamma|$.
Moreover, we let $\Delta_{\gamma}$ be the set of sites in $\mathbb
{Z}^2$ such that either their distance (in $\mathbb{R}^2$) from
$\gamma$ is $\frac12$, or their distance from the set of vertices in
${\mathbb{Z}^2}^*$ where two nonlinked bonds of $\gamma$ meet equals
$1/\sqrt2$. Finally, we let
$\Delta^+_\gamma=\Delta_\gamma\cap\Lambda_\gamma$ and $\Delta
^-_\gamma= \Delta_\gamma\setminus\Delta^+_\gamma$.
\end{definition}

%
%
\begin{definition}
Given a contour $\gamma$ we say that $\gamma$ is an
\textit{$h$-contour} for the configuration $\eta$ if
\[
\eta\restriction_{\Delta^-_\gamma}\leq h-1,\qquad\eta\restriction _{\Delta^+_\gamma}
\geq h.
\]
We will say that $\gamma$ is a contour for the configuration $\eta$ if
there exists $h$ such that $\gamma$~is a $h$-contour for $\eta$.
Finally, $\mathscr{C}_{\gamma,h}$ will denote the event that $\gamma$
is an $h$-contour.
\end{definition}

To illustrate the above definitions with a simple example, consider the
elementary contour given by the square of side $1$ surrounding a site
$x\in\mathbb{Z}^2$. In this case, $\gamma$ is an $h$-contour iff $\eta
_{x}\geq h$ and $\eta_y\leq h-1$ for all $y\in\{x\pm e_1, x\pm e_2, x+
e_1+e_2,x-e_1-e_2\}$. In general, $\Delta^+_\gamma$ (resp.,
$\Delta^-_\gamma$) is the set of $x\in\Lambda_\gamma$ (resp.,
$x\in\Lambda_\gamma^c$) either at distance 1 from $\Lambda_\gamma^c$
(resp., $\Lambda_\gamma$) or at distance $\sqrt2$ from a vertex
$y\in\Lambda_\gamma^c$ (resp., $y\in\Lambda_\gamma$) in the south--west
or north--east direction.

%
%
\begin{remark}
\label{rem:splitting-rules}
As the reader may have noticed the definition of an $h$-contour is
asymmetric in the sense that we require the minimal height of the
surface at the inner boundary of $\gamma$, $\Delta^+_\gamma$, to be
larger than the maximum height at the external boundary. In a sense,
this definition covers \textit{upward fluctuations} of the surface. Of
course one could provide the reverse definition covering
\textit{downward fluctuations}. In the sequel, the latter is not really
needed thanks to monotonicity and symmetry arguments. We also observe
that, contrary to what happens in, for example, Ising models, a
geometric contour $\gamma$ could be at the same time a $h$-contour and
a $h'$-contour with $h\neq h'$. More generally two geometric contours
$\gamma,\gamma'$ could be contours for the same surface with different
height parameters even if $\gamma\cap\gamma'\ne\varnothing$ (but one of
them must be contained in the other).
\end{remark}

The following estimates play a key role in the proof of
Theorem~\ref{thm-equil-shape}.

%
%
\begin{proposition}\label{p-starBound}
There exists an absolute constant $C_0>0$ such that for all $\beta\geq
1$ and $h\geq1$,
%
%
\begin{equation}
\label{e-contourFloorBound} \pi^0_\Lambda(
\mathscr{C}_{\gamma,h} ) \leq\exp \bigl(-\beta|\gamma|+C_0 |
\Lambda_{\gamma}|e^{-4\beta h} \bigr).
\end{equation}
Moreover, for any family of $h$-contours $\{(\gamma_s,h_s)\}_{s\in
\mathscr{S}}$ such that for all $i\geq1$
\[
\mathop{\bigcup_{s\in\mathscr{S}}}_{h_s=i+1}
\Lambda_{\gamma_s} \subseteq\mathop{\bigcup_{s\in\mathscr{S}}}_{h_s=i}
\Lambda_{\gamma_s}
\]
and $\Lambda_{\gamma_s}\cap\Lambda_{\gamma_{s'}} =\varnothing$ when
$h_s=h_{s'}$, $s\neq s'$,we have
%
%
\begin{equation}
\label{e-contourFloorConditional} \pi^0_\Lambda \biggl(
\bigcap_{s\in\mathscr{S}}\mathscr{C}_{\gamma
_s,h_s} \biggr) \leq
\exp \biggl( \sum_{s\in\mathscr{S}} \bigl(-\beta|
\gamma_s|+C_0 |\Lambda_{\gamma_s}|e^{-4\beta h_s}
\bigr) \biggr).
\end{equation}
\end{proposition}

As a step towards the proof of the above proposition, we consider the
setting of no floor and no ceiling, where the picture is simpler as
there is no entropic repulsion.

%
%
\begin{lemma}\label{l:contourBound}
For any $h$-contour $\gamma$ in any domain $\Lambda$ with any boundary
condition $\xi$ we have
\[
\hat{\pi} {}^\xi_\Lambda( \mathscr{C}_{\gamma,h} ) \leq
\exp \bigl(-\beta|\gamma| \bigr).
\]
Moreover, if $h'<h$ and $\gamma,\gamma'$ are contours with
$\Lambda_\gamma\subseteq\Lambda_{\gamma'}$ then
%
%
\begin{equation}
\label{e-contourConditional} \hat{\pi} {}^\xi_\Lambda (
\mathscr{C}_{\gamma,h} \mid\mathscr{C}_{\gamma
',h'} ) \leq\exp \bigl(-\beta|
\gamma| \bigr).
\end{equation}
\end{lemma}

\begin{pf}
Define the map $T=T_\gamma\dvtx\mathbb{Z}^\Lambda\to\mathbb
{Z}^\Lambda
$ by
%
%
\begin{equation}
\label{e-TDefn} (T \eta)_v= \cases{ \eta_v-1, &\quad
$v \in\Lambda_\gamma$, \vspace*{2pt}
\cr
\eta_v, &\quad
otherwise.}
\end{equation}
If $\eta$ has an $h$-contour at $\gamma$, then the difference along
every edge in $\mathbb{Z}^2$ crossing $\gamma$ decreases by 1 so
$\hat{\pi}{}^\xi_\Lambda(T\eta)=e^{\beta|\gamma|}\hat{\pi}{}^\xi
_\Lambda(\eta)$. Since $T$ is a bijection it follows that
\[
\sum_{\mathscr{C}_{\gamma,h}} \hat{\pi} {}^\xi_\Lambda(
\eta) = e^{-\beta
|\gamma|}\sum_{ T^{-1}(\mathscr{C}_{\gamma,h})} \hat{ \pi}
{}^\xi_\Lambda(T \eta)\leq e^{-\beta|\gamma|}.
\]
Equation~(\ref{e-contourConditional}) follows from the same argument by
noting that if $\eta\in\mathscr{C}_{\gamma,h} \cap\mathscr
{C}_{\gamma',h'}$ then $T_\gamma\eta$ remains in
$\mathscr{C}_{\gamma',h'}$. This completes the proof.
\end{pf}

%
%
\begin{remark}
\label{rem:nested-contours} In the context of considering the interior
of an $h$-contour $\gamma$ for possibly nested contours [such as the
ones featured in equation~(\ref{e-contourConditional})], a useful
observation is that
\[
\pi_\Lambda^0 (\eta\restriction_{\Lambda_\gamma}\in\cdot\mid
\mathscr{C}_{\gamma,h} ) = \pi_{\Lambda_\gamma}^\xi ( \cdot\mid
\eta\restriction_{\Delta^+_\gamma} \geq h )
\]
for any boundary condition $\xi$, that is, at most $h$ all along
$\Delta^-_\gamma$. This follows from the fact that conditioning on any
fixed $\xi\leq h$ would contribute an equal pre-factor to all
configurations thanks to having $\eta\restriction_{\Delta^+_{\gamma}}
\geq h$, and as this includes all $\xi$'s with
$\eta\restriction_{\Delta^-_\gamma} \leq h-1$ this further includes
$\mathscr{C}_{\gamma,h}$. Moreover, the same holds when conditioning on
$\mathscr{C}_{\gamma,h} \cap E$ (instead of just $\mathscr
{C}_{\gamma,h}$) for an arbitrary event $E$ which is only a function of
the configuration on $(\Lambda_\gamma)^c$.
\end{remark}

(Note that the above remark similarly applies to $\Pi$ and $\hat{\pi}$
by the same argument.)

A Peierls-type argument will transform the above lemma into the
following bound on upward (downward) fluctuations in the no floor, no
ceiling setting.

%
%
\begin{proposition}\label{p-heightBounds}
There exists an absolute constant $c>0$ such that for any $\beta\geq
1$, domain $\Lambda$, site $v\in\Lambda$ and height $h\geq0$,
\[
\tfrac12 e^{-4\beta h} \leq\hat{\pi} {}^0_\Lambda(\eta
_v \geq h ) \leq ce^{-4\beta h}.
\]
\end{proposition}

\begin{pf}
Define the map $S\dvtx\mathbb{Z}^\Lambda\to\mathbb{Z}^\Lambda$ by $(S
\eta)_u=\eta_u$ for $u\neq v$ and
\[
(S \eta)_v= %
\cases{ \eta_v+h, & \quad$
\eta_v \geq0 $, \vspace*{2pt}
\cr
\eta_v-h, & \quad$
\eta_v < 0 $.} %
\]
Observe that $|(S \eta)_v| \geq h$ and that since $S$ changes the
Hamiltonian by at most $4h$,
\[
\hat{\pi} {}^0_\Lambda(S \eta) \geq e^{-4\beta h} \hat{
\pi} {}^0_\Lambda(\eta).
\]
Moreover, as $S$ is injective, summing over $\eta$ we have that
\[
\hat{\pi} {}^0_\Lambda \bigl(|\eta_v| \geq h \bigr) =
\sum_{\eta\in\mathbb{Z}^\Lambda} \hat{\pi} {}^0_\Lambda(S
\eta) \geq e^{-4\beta h} \sum_{\eta\in\mathbb{Z}^\Lambda} \hat{\pi }
{}^0_\Lambda( \eta) = e^{-4\beta h}.
\]
Since by symmetry $\hat{\pi}{}^0_\Lambda(\eta_v \geq h ) =
\hat{\pi}{}^0_\Lambda(\eta_v \leq-h )$ the lower bound follows.

To get the upper bound, define a set of nested contours surrounding $v$
as
\[
\mathcal{A}(h,v)= \bigl\{(\gamma_1,\ldots,\gamma_h)
\dvtx v \in\Lambda_{\gamma
_h}\mbox{ and }\Lambda_{\gamma_{i+1}} \subseteq
\Lambda_{\gamma_i}\mbox{ for all }1\leq i \leq h-1 \bigr\}
\]
and observe that, if $\eta$ is such that $\eta_v \geq h$, then
necessarily there exists
$(\gamma_1,\ldots,\gamma_h)\in\mathcal{A}(h,v)$ such that
$\eta\in\bigcap_{1\leq i \leq h}\mathscr{C}_{\gamma_i,i}$.

Applying Lemma~\ref{l:contourBound} iteratively (while bearing
Remark~\ref{rem:nested-contours} in mind), we now obtain that for every
$(\gamma_1,\ldots,\gamma_h)\in\mathcal{A}(h,v)$,
%
%
\begin{equation}
\label{e-nestedCountourBound} \hat{\pi} {}^\xi_\Lambda
\biggl( \bigcap_{1\leq i \leq h} \mathscr{C}_{\gamma_i,i}
\biggr) \leq e^{-\beta\sum_{i=1}^h
|\gamma_i|}.
\end{equation}
Simple counting gives that the number of contours of length $n$
starting from a vertex is at most $R_n$, the number of self avoiding
walks of length $n$. If such a path surrounds $v$, then it must cross
the horizontal line containing $v$ to its right within distance $n$ so
the number of $\gamma$ with $|\gamma|=n$ and $v\in\Lambda_{\gamma }$ is
at most $n R_n$ (with room to spare). Hence
\[
\mathop{\sum_{\gamma\dvtx v\in\Lambda_{\gamma}}}_{|\Lambda
_{\gamma}|>2}
e^{-\beta|\gamma|+6\beta} \leq\sum_{n=8}^\infty n
R_n e^{-\beta n +6\beta},
\]
which is uniformly bounded in $\beta$ for any $\beta\geq1$ since the
connective constant $\mu_2 = \lim_{n\to\infty} R_n^{1/n}$ is known to
satisfy $\mu_2 <2.68 < e$. Hence, for some large enough $M$,
independent of $\beta$,
%
%
\begin{equation}
\label{e-MBound} \mathop{\sum_{\gamma\dvtx v\in\Lambda_{\gamma
}}}_{|\Lambda
_{\gamma}|> M}
e^{-\beta|\gamma|} \leq e^{-6\beta}.
\end{equation}
Now define a collection of nested contours of area at least 2 and at
most $M$ as
\[
\mathcal{A}_M(h,v)= \bigl\{(\gamma_1,\ldots,
\gamma_h)\in\mathcal{A}(h,v)\dvtx2\leq|\Lambda_{\gamma_i}| \leq M
\mbox{ for all }1\leq i \leq h \bigr\}.
\]
We note that
%
%
\begin{eqnarray}
\label{e-AMSizeBound} \bigl|\mathcal{A}_M(h,v) \bigr| &\leq& \bigl|
\mathcal{A}_M(1,v) \bigr|^{M-1} \pmatrix{h + M - 2
\cr
\vspace*{2pt} M-2}
\nonumber
\\[-8pt]
\\[-8pt]
&\leq& \bigl|\mathcal{A}_M(1,v) \bigr|^{M-1}(h+M)^{M-2}
\nonumber
\end{eqnarray}
since, examining the way $|\Lambda_{\gamma_i}|$ decreases, there are at
most $M-2$ transitions of $|\Lambda_{\gamma_i}|<
|\Lambda_{\gamma_{i-1}}|$ and in each case the number of possible
$\gamma_i$ is at most $|\mathcal{A}_M(1,v)|$ with much room to spare.

For any $(\gamma_1,\ldots,\gamma_h)\in\mathcal{A}(h,v)$, we can find
$0\leq k \leq l \leq h$ such that $|\Lambda_{\gamma_i}|>M$ for $1\leq i
\leq k$, that $(\gamma_{k+1},\ldots,\gamma_l)\in\mathcal{A}_M(l-k,v)$
and $|\Lambda_{\gamma_i}|=1$ for $l<i\leq h$. Then
\begin{eqnarray*}
&& \sum_{(\gamma_1,\ldots,\gamma_h)\in\mathcal{A}(h,v)} e^{-\beta
\sum_{i=1}^h |\gamma_i|}
\\
&&\qquad = \sum_{0\leq k \leq l\leq h} \mathop{\sum
_{(\gamma_1,\ldots,\gamma
_k)}}_{|\Lambda_{\gamma_k}|>M} \sum_{(\gamma_{k+1},\ldots,\gamma_l)\in\mathcal{A}_M(l-k,v)}
e^{-\beta\sum_{i=1}^{l} |\gamma_i| - 4\beta(h-l)}
\\
&&\qquad \leq\sum_{0\leq k \leq l\leq h} \bigl|\mathcal{A}_M(l-k,v)\bigr| e^{-6\beta l -
4\beta(h-l)}
\\
&&\qquad \leq e^{-4\beta h}\sum_{0\leq k \leq l\leq h} \bigl|
\mathcal{A}_M(1,v) \bigr|^{M-1} (l-k+M)^{M-2}
e^{-2\beta l}
\\
&&\qquad \leq c e^{-4\beta h},
\end{eqnarray*}
where the first equality holds since $|\gamma_i|=4$ when
$|\Lambda_{\gamma_i}|=1$, the inequality in the second line is by
equation~(\ref{e-MBound}) and the fact that every contour with
$|\Lambda_{\gamma}|\geq2$ has $|\gamma|\geq6$, and where the transition
in the third line is by~(\ref{e-AMSizeBound}). Combining with
equation~(\ref{e-nestedCountourBound}) completes the proof.
\end{pf}

Proposition~\ref{p-heightBounds} allows us to readily infer
Lemma~\ref{lem-no-ceil}.

\begin{pf*}{Proof of Lemma~\ref{lem-no-ceil}} One has
\[
\pi^0_\Lambda(A) = \Pi^0_\Lambda(A)
\frac{\Xi_\Lambda
^0}{Z^0_\Lambda},
\]
where $\Xi_\Lambda^0$ denotes the partition function corresponding to
$\Pi_\Lambda^0$. The fact that $ \Pi^0_\Lambda(A) \leq\pi
^0_\Lambda
(A) $ follows immediately from $ \Xi_\Lambda^0\geq Z^0_\Lambda$. To
show that $\pi^0_\Lambda(A) \leq(1 + c L^2 e^{-2\beta n^+} )
\Pi^0_\Lambda(A)$, observe that
\[
\frac{Z^0_\Lambda}{\Xi_\Lambda^0}=\Pi^0_\Lambda \bigl(\eta\leq n^+ \bigr),
\]
so that
\[
\pi^0_\Lambda(A) = \frac{\Pi^0_\Lambda(A)}{1-
\Pi^0_\Lambda(\bigcup_{v\in\Lambda} \{\eta_v > n^+\} )} \leq \frac{\Pi^0_\Lambda(A)}{1-\sum_{v\in\Lambda}
\Pi^0_\Lambda( \{\eta_v > n^+\} )}.
\]
Thanks to monotonicity and Proposition~\ref{p-heightBounds}, for any $v
\in\Lambda$ we have
\begin{eqnarray*}
\Pi^0_\Lambda \bigl(\eta_v > n^+ \bigr) &\leq&
\Pi^{n^+/2}_\Lambda \bigl(\eta_v > n^+ \bigr) \leq
\frac{\hat{\pi}{}^0_\Lambda(\eta_v > n^+/2)}{
\hat{\pi}{}^0_\Lambda(\eta\geq-n^+/2) } \leq\frac{c e^{-2\beta
n^+}}{1-c |\Lambda| e^{-2\beta n^+}},
\end{eqnarray*}
(where we took $n^+/2$ to be an integer to simplify the exposition) as
required.
\end{pf*}

Having bounded the probability of exceeding a certain height in the no
floor setting, we can now quantify the entropic repulsion effect and
derive an estimate on $\pi^0_\Lambda(\mathscr{C}_{\gamma,h})$.

\begin{pf*}{Proof of Proposition~\ref{p-starBound}}
Thanks to Lemma~\ref{lem-no-ceil}, it suffices to prove the analogous
estimates for the measure $\Pi$ with no ceiling.

By\vspace*{-2pt} Remark~\ref{rem:nested-contours}, the conditional
distribution of $\eta\restriction_{\Lambda_{\gamma}}$ given
$\mathscr{C}_{\gamma,h}$ is equal to
$\Pi^h_{\Lambda_{\gamma}}(\cdot\mid
\eta\restriction_{\Delta^+_\gamma}\geq h)$ which stochastically
dominates $\Pi^h_{\Lambda_{\gamma}}$. Hence,
\begin{eqnarray*}
\Pi^0_\Lambda (\eta\restriction_{\Lambda_{\gamma}} > 0\mid
\mathscr{C}_{\gamma,h} )&\geq&\Pi^h_{\Lambda_{\gamma}}(\eta
\restriction_{\Lambda_{\gamma}} > 0) \geq\prod_{v\in\Lambda
_{\gamma}}
\Pi^h_{\Lambda_{\gamma}}(\eta_v > 0)
\\
&\geq&\prod_{v\in\Lambda_{\gamma}} \hat{\pi} {}^h_{\Lambda_{\gamma
}}(
\eta_v > 0) \geq \biggl(\frac12 \vee \bigl(1-ce^{-4\beta h}
\bigr) \biggr)^{|\Lambda_{\gamma}|},
\end{eqnarray*}
where the second inequality follows by the FKG inequality, the third
follows by monotonicity of removing the floor and the final inequality
by symmetry and Proposition~\ref{p-heightBounds}. Therefore,
%
%
\begin{equation}
\label{e-FKGLowerBound} \Pi^0_\Lambda(\eta
\restriction_{\Lambda_{\gamma}} > 0\mid\mathscr{C} _{\gamma,h}) \geq\exp
\bigl(-2c |\Lambda_{\gamma}| e^{-4\beta h} \bigr),
\end{equation}
since $\frac12 \vee(1-x) \geq\exp(-2x)$ for $x\geq0$. With
$T_\gamma$
defined as in~(\ref{e-TDefn}), on the event that $\gamma$ is an
$h$-contour and $\eta(\Lambda_{\gamma}) > 0$ we have $T \eta\geq0$ and
$\Pi^0_\Lambda(T \eta) = e^{\beta|\gamma|}\Pi^0_\Lambda(\eta)$. It
follows from this bijection that
%
%
\begin{eqnarray}
\label{eq-FKGLowerBoundbis} 1 &\geq&\mathop{\sum_{\eta\dvtx \eta
\restriction_{\Lambda_{\gamma
}} > 0,}}_{\mathscr{C}_{\gamma,h}}
\Pi^0_\Lambda(T \eta) = e^{\beta
|\gamma|}\Pi^0_\Lambda(
\eta\restriction_{\Lambda_{\gamma
}} > 0, \mathscr{C}_{\gamma,h} )
\nonumber
\\[-8pt]
\\[-8pt]
&\geq&\exp \bigl(\beta|\gamma| -2c |\Lambda_{\gamma}| e^{-4\beta
h}
\bigr) \Pi^0_\Lambda(\mathscr{C}_{\gamma,h})
\nonumber
\end{eqnarray}
with the second inequality by~(\ref{e-FKGLowerBound}). Rearranging this
establishes~(\ref{e-contourFloorBound}). To
obtain~(\ref{e-contourFloorConditional}) note first that the proof
applies unchanged if $h_s=h$ for all $s$, that is, when the family of
disjoint contours is of the form $\{(\gamma_s,h)\}_{s\in\mathscr{S}}$,
in this case yielding
\[
\Pi^0_\Lambda \biggl(\bigcap_{s\in\mathscr{S}}
\mathscr{C}_{\gamma
_s,h} \biggr) \leq\exp \biggl( \sum
_{s\in\mathscr{S}} \bigl(-\beta|\gamma_s|+2c |
\Lambda_{\gamma_s}|e^{-4\beta h} \bigr) \biggr).
\]
Now take a general family $\{(\gamma_s,h_s)\}_{s\in\mathscr{S}}$
satisfying the hypothesis of the lemma. We proceed by induction over
the levels of the contours from top to bottom. If $h_+=\max_s h_s$,
then conditioning on $\bigcap_{s\in\mathscr{S}\dvtx h_s<h_+}
\mathscr
{C}_{\gamma_s,h_s}$ does not affect the conditional distribution of
$\eta(\bigcup_{s\in\mathscr{S}\dvtx h_s=h_+} \Lambda_\gamma)$
given\break
$\bigcap_{s\in\mathscr{S}\dvtx h_s=h_+} \mathscr{C}_{\gamma
_s,h_+}$ (as
explained in Remark~\ref{rem:nested-contours}). Moreover, given
that\break
$\bigcap_{s\in\mathscr{S}} \mathscr{C}_{\gamma_s,h_s}$ holds then
$T_{h_+} \eta\in\bigcap_{s\in\mathscr{S}\dvtx h_s<h_+}
\mathscr{C}_{\gamma_s,h}$, where $T_{h_+}$ denotes the composition of
the $T_{\gamma_s}$'s for all $s$ such that $h_s=h_+$, that is, reducing
the height of every site in $\bigcup_{s\in\mathscr{S}\dvtx h_s=h_+}
\Lambda_{\gamma_s}$ by 1. This implies that
\begin{eqnarray*}
&& \Pi^0_\Lambda \biggl(\bigcap
_{s\in\mathscr{S}\dvtx h_s=h_+} \mathscr{C}_{\gamma _s,h_s} \Big| \bigcap
_{s\in\mathscr{S}\dvtx
h_s<h_+} \mathscr{C}_{\gamma _s,h_s} \biggr)
\\
&&\qquad \leq\exp \biggl( \sum_{s\in\mathscr{S}\dvtx h_s=h_+} \bigl(-\beta|
\gamma_s|+2c | \Lambda_{\gamma_s}|e^{-4\beta h_+} \bigr)
\biggr).
\end{eqnarray*}
The proof is completed by induction.
\end{pf*}

\begin{pf*}{Proof of Theorem~\ref{thm-equil-shape}, equation~(\ref{eq-equilib-i})}
It suffices to prove the corresponding bounds for $\Pi$. Set $h=H -k$
and $\mathcal{S}_h(\eta)=\{v\in\Lambda\dvtx\eta_v=h\}$. For each
$A\subseteq\mathcal{S}_h(\eta)$, we can define $U_A\dvtx
\Omega\to\Omega$ given by
\[
(U_A \eta)_v = %
\cases{
\eta_v+1, &\quad$v \notin A$, \vspace*{2pt}
\cr
0, &\quad$ v\in A$.}
\]
To measure the effect of $U_A$ on the Hamiltonian, observe that $U_A$
is equivalent to incrementing each height by $1$ followed by decreasing
the sites in $A$ by $h+1$. As such, this operation increases the
Hamiltonian by at most $|\partial\Lambda| + 4(h+1)|A|$ and so
altogether
\[
\Pi^0_\Lambda(U_A \eta) \geq\exp \bigl(-4\beta
L - 4\beta(h+1)|A| \bigr)\Pi^0_\Lambda(\eta).
\]
Hence,
\begin{eqnarray*}
&& \sum_{A\subseteq\mathcal{S}_h(\eta)} \Pi^0_\Lambda(U_A
\eta)
\\
&&\qquad \geq\exp(-4\beta L) \bigl(1+e^{-4\beta(h+1)} \bigr)^{|\mathcal{S}_h(\eta
)|}
\Pi^0_\Lambda(\eta)
\\
&&\qquad \geq\exp \biggl(-4\beta L+\frac12 e^{-4\beta(h+1)} \bigl|
\mathcal{S}_h( \eta) \bigr| \biggr)\Pi^0_\Lambda(\eta),
\end{eqnarray*}
since $e^{-4\beta(h+1)}\leq1$ and $(1+x)\geq e^{x/2}$ for $0\leq x\leq
1$. By construction, we have $U_A \eta\neq U_{A'} \eta$ for any $A \neq
A'$ with $A,A'\subseteq\mathcal{S}_h(\eta)$. In addition, if
$A\subseteq\mathcal{S}_h(\eta)$ and $A'\subseteq\mathcal{S}_h(\eta')$
for some $\eta\neq\eta'$ then $U_A \eta\neq U_{A'} \eta'$ (thanks to
the fact that one can recover $A$ from $U_A \eta$---the sites at level
0---then proceed to recover $\eta$). We can therefore conclude that
\begin{eqnarray*}
1 &\geq&\sum_{\eta\dvtx |\mathcal{S}_h(\eta)|\geq e^{-2\beta k}
L^2} \sum
_{A\subseteq\mathcal{S}_h(\eta)} \Pi^0_\Lambda(U_A
\eta)
\\
&\geq&\exp \biggl(-4\beta L+\frac12 e^{-4\beta(h+1)} e^{-2\beta k}
L^2 \biggr)\Pi^0_\Lambda \bigl( \bigl|
\mathcal{S}_h(\eta) \bigr|\geq e^{-2\beta k} L^2 \bigr)
\end{eqnarray*}
and so, for $k\geq1$
\begin{eqnarray*}
\Pi^0_\Lambda \bigl( \bigl|\mathcal{S}_h(\eta) \bigr|\geq
e^{-2\beta k} L^2 \bigr) &\leq& \exp \biggl(4\beta L - \frac12
e^{2\beta k-8\beta} L \biggr)
\\
&\leq&\frac12 \exp \bigl( - e^{\beta k} L \bigr),
\end{eqnarray*}
where the last inequality holds for any $k\geq12$. A union bound over
all $k \geq12$ now holds at the cost of increasing the pre-factor of
$1/2$ to $1$, as desired.
\end{pf*}

\begin{pf*}{Proof of Theorem~\ref{thm-equil-shape}, equation~(\ref{eq-equilib-ii})}
As above, we prove the corresponding
bounds for $\Pi$ and the result for $\pi$ will follow from
Lemma~\ref{lem-no-ceil}. Let $\mu_2 < 2.68$ be the connective constant
in $\mathbb{Z}^2$ and set
\[
h = H + \biggl\lceil\frac1{4\beta}\log \biggl(\frac{C_0}{1-\log
\mu_2} \biggr) \biggr
\rceil,
\]
where $C_0>0$ is the absolute constant from
Proposition~\ref{p-starBound}. By the isoperimetric inequality in
$\mathbb{Z}^2$, we have $|\Lambda_\gamma|\leq(L/4)|\gamma|$ for any
contour $\gamma$ in an $L\times L$ box $\Lambda$. Plugging these
in~(\ref{e-contourFloorBound}) gives
%
%
\begin{equation}
\label{eq-h-cont} \Pi^0_\Lambda(\mathscr{C}_{\gamma,h}
) \leq\exp \bigl(-\beta|\gamma| + C_0 (L/4)|\gamma| e^{-4\beta h}
\bigr) 
\leq\exp \bigl(-\theta|\gamma| \bigr),
\end{equation}
where
\[
\theta= \beta- \tfrac14(1-\log\mu_2) \geq1 - \tfrac14(1-\log\mu
_2) > \log\mu_2
\]
by our hypothesis that $\beta\geq1$.

Now define the random set $\mathscr{A}$ by
\[
\mathscr{A}= \mathscr{A}(\eta) = \bigl\{ \gamma\dvtx \gamma\mbox{ is an }h
\mbox{-contour of }\eta\mbox{ of length }|\gamma|\leq\log^2 L \bigr\}
\]
and let $\mathscr{A}_0$ be the result of omitting nested contours from
$\mathscr{A}$:
\[
\mathscr{A}_0 = \mathscr{A}_0(\eta) = \mathscr{A}
\setminus\{ \psi\in\mathscr{A}\dvtx\Lambda_\psi\subsetneq
\Lambda_\gamma\mbox{ for some }\gamma\in\mathscr{A} \}.
\]
For any collection of contours $A$, let also
\[
E_A = \biggl\{ \biggl|\bigcup_{\gamma\in A} \{ v
\in\Lambda_\gamma\dvtx\eta_v\geq h+k \} \biggr| >
\frac12e^{-2\beta
k}L^2 \biggr\}
\]
and observe that $E_\mathscr{A}=E_{\mathscr{A}_0}$ since $\bigcup\{
\Lambda_\gamma\dvtx\gamma\in\mathscr{A}\} =
\bigcup\{\Lambda_\gamma\dvtx\gamma\in\mathscr{A}_0\}$.
We thus have
%
%
\begin{eqnarray}
\label{eq-high-short-cont} \qquad\quad\Pi^0_\Lambda(E_{\mathscr{A}}
) &=& \sum_{A_0} \Pi^0_\Lambda
(E_{A_0} \mid\mathscr{A}_0 =A_0 )
\Pi^0_\Lambda( \mathscr{A} _0=A_0)
\nonumber
\\[-8pt]
\\[-8pt]
&=& \sum_{A_0 } \Pi^0_\Lambda
\biggl(\sum_{\gamma\in A_0} \mathscr{X}_\gamma>
\frac12e^{-2\beta k}L^2 | \mathscr{A}_0
=A_0 \biggr)\Pi^0_\Lambda(\mathscr{A}_0=A_0),
\nonumber
\end{eqnarray}
where
\[
\mathscr{X}_\gamma= \sum_{v \in\Lambda_\gamma}
\mathbf{1}_{\eta
_v\geq h+k}.
\]
Conditioned on $\mathscr{A}_0=A_0$, monotonicity enables us to increase
the values along $\Delta^- _\gamma$ for every $\gamma\in A_0$ to $h-1$
while possibly only increasing the probability of the event $E_{A_0}$,
and by doing so the variables $\{\mathscr{X}_\gamma\dvtx \gamma\in
\mathscr{A}_0\}$ become mutually independent.

Fix $\gamma\in A_0$. If $\eta_v \geq h + k$ for some $v\in
\Lambda_\gamma$ this gives rise to a sequence of nested $j$-contours
for $j=h+1,\ldots,h+k$ surrounding $v$, and by
Proposition~\ref{p-starBound} the probability for a given fixed such
sequence $\psi_1,\ldots,\psi_k$ is at most
\[
\exp \biggl(-\beta\sum_j \bigl(|
\psi_j| + C_0 |\Lambda_{\psi_j}| e^{-4\beta(h+j)}
\bigr) \biggr).
\]
However, the fact that $\sum_j|\Lambda_{\psi_j}| e^{-4\beta(h+j)} =
O(L^{-1}\log^4 L)$ shows that the area term in this estimate is
negligible, hence the same argument used for proving the upper bound of
Proposition~\ref{p-heightBounds} (in the no floor setting) yields that,
for some absolute $c>0$ and every $v\in\Lambda_\gamma$,
\[
\Pi^0_{\Lambda} (\eta_v \geq h + k\mid
\mathscr{C}_{\gamma,h} ) \leq c \exp(-4\beta k).
\]
In particular,
\[
\mathbb{E}_{\Pi^0_{\Lambda}( \cdot\mid\mathscr{C}_{\gamma,h})} [ \mathscr{X} _\gamma] \leq|
\Lambda_\gamma| c \exp(-4\beta k)
\]
and so
\[
\mathbb{E}_{\Pi^0_{\Lambda}( \cdot\mid\bigcap_{\gamma\in
A_0}\mathscr{C}
_{\gamma,h})} \biggl[\sum_{\gamma\in A_0}
\mathscr{X}_\gamma \biggr] \leq c \exp(-4\beta k) \sum
_{\gamma\in A_0} |\Lambda_\gamma| \leq c \exp(-4\beta k)
L^2.
\]
The variable $\mathscr{Y}= \sum_{\gamma\in A_0} \mathscr{X}_\gamma $ is
therefore a sum of $|A_0| \leq L^2 $ independent variables, each of
which respects the bound $|\mathscr{X}_\gamma| \leq|\Lambda_\gamma|
\leq\log^4 L$ with probability~1. Since $\beta\geq1$, for any $k
\geq\frac12\log(4c)$ we have $\mathbb{E}_{\Pi^0_{\Lambda}( \cdot
\mid\bigcap_{\gamma\in A_0}\mathscr{C}_{\gamma,h})} [ \mathscr{Y}
]\leq\frac14 e^{-2\beta k} L^2$, and
applying Hoeffding--Azuma now gives
\[
\Pi^0_\Lambda \bigl(\mathscr{Y}\geq\tfrac12
e^{-2\beta k}L^2 | \mathscr{A}_0 = A_0
\bigr) \leq\exp \bigl(-\tfrac1{32} e^{-4\beta k} L^2
\log^{-8} L \bigr).
\]
Together with~(\ref{eq-high-short-cont}) we finally get
%
%
\begin{equation}
\label{eq-30} \Pi^0_\Lambda(E_{\mathscr{A}} ) \leq\exp
\bigl(-e^{-4\beta k} L^{2-o(1)} \bigr).
\end{equation}
Having accounted for this probability in the
inequality~(\ref{eq-equilib-ii}), we are left with the problem of
handling the contribution of long contours, namely those whose length
exceeds $\log^2 L$.

Set
\[
\mathscr{B}= \mathscr{B}(\eta) = \bigl\{ \gamma\dvtx \gamma \mbox{ is an }h
\mbox{-contour of }\eta \mbox{ of length }|\gamma|>\log^2 L \bigr\}.
\]
We have shown in~(\ref{eq-h-cont}) that, for some $\theta\geq
\theta_0$ with a fixed $\theta_0 >\log\mu_2$ and any given contour
$\gamma$,
\[
\Pi^0_\Lambda(\mathscr{C}_{\gamma,h} ) \leq\exp\bigl(-
\theta|\gamma|\bigr).
\]
By the same argument [appealing to Proposition~\ref{p-starBound}, this
time to the more general bound~(\ref{e-contourConditional})], if, for
some $m=m(L)$, one considers contours $\gamma_1,\ldots,\gamma_m$ with
disjoint interiors $\{\Lambda_{\gamma_i}\}_{i=1}^m$ and individual
lengths all exceeding $\log^2 L$, then
\[
\Pi^0_\Lambda \Biggl(\bigcap_{i=1}^m
\mathscr{C}_{\gamma_i,h} \Biggr) \leq\exp \Biggl(-\theta\sum
_{i=1}^m|\gamma_i| \Biggr).
\]
By enumerating over the length of each contour $\gamma_i$, then
selecting its origin and a self-avoiding path for it (the number of
options for the latter being counted by~$R_{|\gamma_i|}$), we see that
\[
\Pi^0_\Lambda \Biggl(\bigcup_m
\bigcup_{\{\gamma_i\}
_{i=1}^m}\bigcap_{i=1}^m
\mathscr{C}_{\gamma_i,h} \Biggr) \leq\sum_m
\prod_{i=1}^m \sum
_{\log^2 L < |\gamma_i|\leq L^2} L^2 R_{|\gamma_i|} e^{-\theta
|\gamma_i|}.
\]
The relation between $\theta$ and $\log\mu_2$ suffices to eliminate
$R_{|\gamma_i|}$ while still retaining a factor of $\exp(-c \sum_i
|\gamma_i|)$ for some absolute $c>0$. The fact that $\sum_i|\gamma_i|$
is super-logarithmic now eliminates the $L^2$ pre-factor, as well as
the additional enumeration over $m$ itself (another polynomial factor).
Altogether,
\[
\Pi^0_\Lambda \biggl(\sum_{\gamma\in\mathscr{B}}|
\gamma| \geq\frac12 e^{-2\beta k}L \biggr) \leq\exp \bigl(-c e^{-2\beta k}L
\bigr)
\]
for some absolute $c>0$, and in particular [via the isoperimetric
inequality $|\Lambda_\gamma|\leq(L/4)|\gamma|$]
%
%
\begin{eqnarray}
\label{eq-33}
&& \Pi^0_\Lambda \biggl( \biggl|\bigcup
_{\gamma\in\mathscr{B}} \{ v \in\Lambda_\gamma\dvtx
\eta_v \geq h \} \biggr| > \frac18e^{-2\beta
k}L^2 \biggr)
\nonumber\\[-8pt]\\[-8pt]
&&\qquad \leq\exp \bigl(-c e^{-2\beta k}L \bigr).\nonumber
\end{eqnarray}
Together with the aforementioned bound on $\Pi^0_\Lambda(E_\mathscr
{A})$, this completes the proof.
\end{pf*}

\section{Lower bounds on equilibration times}
\subsection{Proof of Theorem~\texorpdfstring{\protect\ref{th-principale}}{1}: Lower bound
on the mixing time}
Set
\[
h = H - K,
\]
where $K$ is the constant from Theorem~\ref{thm-equil-shape}, and
define
\[
\mathcal{B}= \bigl\{\eta\dvtx\#\{x\in\Lambda_L\dvtx
\eta_x \geq h +1\}\geq\tfrac12 L^2 \bigr\}.
\]
Note that, since $\exp(-2\beta K ) \leq\frac12$,
equation~(\ref{eq-equilib-i}) of Theorem~\ref{thm-equil-shape} implies
that
%
%
\begin{equation}
\label{eq-31} \pi^0_\Lambda(\mathcal{B} ) = 1-o(1).
\end{equation}
%
Hence, if $\tau_\mathcal{B}$ denotes the hitting time of the set
$\mathcal{B}$, it will suffice to show that for a~sufficiently small
constant $c>0$
%
%
\begin{equation}
\label{eq-3100} \min_\eta\mathbb{P}^\eta \bigl(
\tau_\mathcal{B}< e^{cL} \bigr)=o(1).
\end{equation}
For this purpose, we observe that $\mathcal{B}$ is an increasing event
so that,
\[
\min_\eta\mathbb{P}^\eta \bigl(
\tau_\mathcal{B}< e^{cL} \bigr)=\mathbb{P}^\sqcup \bigl(
\tau_\mathcal{B}< e^{cL} \bigr)\leq\mathbb{P}^\nu
\bigl(\tau_\mathcal{B}< e^{cL} \bigr)
\]
for any initial law $\nu$.

We now choose $\nu$ as follows. Take $\delta\in(0,\frac14)$ to be a
sufficiently small constant so that in terms of the constant $C_0$
from~(\ref{e-contourFloorBound})
\[
\delta< \biggl[(\beta-\log\mu_2) \frac{4}{C_0} \exp \bigl(-4
\beta(H-h + 1) \bigr) \biggr]^2,
\]
where $\mu_2$ is the connective constant in $\mathbb{Z}^2$.
Rearranging the
above condition gives
%
%
\begin{equation}
\label{eq-25} \lambda:=\sqrt{\delta}(C_0/4) \exp \bigl(4\beta(H-h+1)
\bigr) < \beta- \log\mu_2.
\end{equation}
Then we take as starting law $\nu$ the conditional measure
$\pi_\Lambda^0(\cdot|A)$ where $A$ is the event that there exists no
$h$-contour $\gamma$ with area exceeding $\delta L^2$, that is,
\[
A = \bigcap_{\gamma\dvtx |\Lambda_\gamma|> \delta L^2} (\mathscr {C}_{\gamma,h})^c.
\]
In the sequel, $\partial A$ will denote the internal boundary of $A$
defined by
\[
\partial A:= \bigl\{\eta\in A\dvtx p \bigl(\eta,\eta' \bigr)>0
\mbox{ for some } \eta' \notin A \bigr\},
\]
where $p(\cdot,\cdot)$ is the transition probability of the dynamics.
Let $\tau_{\partial
A}$ be the hitting time of $\partial A$.

Notice that, up to time $\tau_{\partial A}$, the Glauber dynamics
started in $A\setminus\partial A$ coincides with the reflected Glauber
dynamics in $A$ whose reversible measure is precisely $\nu\equiv
\pi_\Lambda^0(\cdot|A)$. Therefore, a simple union bound over times
$t\in[0,e^{cL}]$ gives that
%
%
\begin{eqnarray}
\label{fabio1} \mathbb{P}^\nu \bigl(\tau_\mathcal{B}<
e^{cL} \bigr) &\leq&\mathbb{P}^\nu \bigl(\tau
_{\partial A}< e^{cL} \bigr)+ \mathbb{P}^\nu \bigl(
\tau_\mathcal{B}< e^{cL}\leq\tau_{\partial A} \bigr)
\nonumber
\\[-8pt]
\\[-8pt]
&\leq& e^{cL} \bigl(\nu(\partial A) +\nu(\mathcal{B}) \bigr).
\nonumber
\end{eqnarray}
Define now
\[
\tilde{A} = \bigcap_{\gamma\dvtx |\Lambda_\gamma|>1/5 \delta L^2} (\mathscr{C}_{\gamma,h})^c.
\]
Notice that $\partial A\subset A\setminus\tilde A$ since at most four
distinct $h$-contours can be combined by the modification of a single
site. Therefore,
\[
\nu(\partial A)= \frac{\pi^0_\Lambda(\partial A )}{\pi
^0_\Lambda(A)} \leq\frac{\pi^0_\Lambda(A\setminus\tilde A )}{\pi
^0_\Lambda(A)}.
\]
We next claim that
%
%
\begin{equation}
\label{eq-pi-A-minus-tildeA} \frac{\pi^0_\Lambda(A\setminus\tilde
A )}{\pi^0_\Lambda(A)}\leq e^{-c_1L}
\end{equation}
for some constant $c_1=c_1(\beta)$. Indeed, suppose that $\gamma$ is a
contour such that $|\Lambda_\gamma|/L^2 \in(\frac15 \delta,\delta
)$. As
in the proof of Proposition~\ref{p-starBound} [see formula
(\ref{eq-FKGLowerBoundbis})] and with $T_\gamma$ defined as in
(\ref{e-TDefn}),
\[
\pi^0_\Lambda(A)\geq\mathop{\sum
_{\eta\in A,
\eta\restriction_{\Lambda_\gamma}>0}}_{\mathscr{C}_{\gamma,h}}\pi _\Lambda^0(T_\gamma
\eta)=e^{\beta
|\gamma|} \pi_\Lambda^0( \eta
\restriction_{\Lambda_\gamma}>0\mid A, \mathscr{C}_{\gamma,h}) \pi
^0_\Lambda(A\cap\mathscr{C}_{\gamma,h}),
\]
where we used the fact that $T_\gamma\eta\in A$ if $\eta\in
A\cap\mathscr{C}_{\gamma,h} $. Next, we observe that, thanks to
(\ref{e-FKGLowerBound}) (which holds with identical proof also for
$\pi_\Lambda^0$),
\[
\pi_\Lambda^0(\eta\restriction_{\Lambda_\gamma}>0\mid A,
\mathscr{C} _{\gamma,h})= \pi_\Lambda^0(\eta
\restriction_{\Lambda_\gamma}>0\mid\mathscr{C} _{\gamma,h})\geq \exp
\bigl(-2c| \Lambda_\gamma|e^{-4\beta h} \bigr)
\]
to yield
%
%
\begin{equation}
\label{eq-29} \pi^0_\Lambda(\mathscr{C}_{\gamma,h} \mid
A ) \leq\exp \bigl(-\beta|\gamma| + C_0 |\Lambda_\gamma|
\exp(-4\beta h) \bigr).
\end{equation}
The isoperimetric inequality in $\mathbb{Z}^2$ gives that $|\Lambda
_\gamma|
\leq|\gamma|^2/16$ for any $\gamma$, so that, by the above choice of
parameters, any contour $\gamma$ with area less than $\delta L^2$
satisfies
\begin{eqnarray*}
C_0 |\Lambda_\gamma| e^{-4\beta h} &\leq&
C_0 \bigl(\sqrt{\delta L^2} \sqrt{|\gamma|^2/16}
\bigr) \biggl(\frac{e^{4\beta
H}}{e^{-4\beta} L} \biggr) e^{-4\beta h}
\\
&\leq&\sqrt{\delta} (C_0/4) |\gamma|e^{4\beta(H-h+1)}
\\
&=& \lambda|
\gamma|,
\end{eqnarray*}
where $\lambda$ is given by (\ref{eq-25}). Hence, the r.h.s. of
(\ref{eq-29}) is smaller than $e^{-(\beta-\lambda)|\gamma|}$. A~union
bound over $\gamma$'s with $|\Lambda_\gamma|>(\delta/5)L^2$ then
proves~(\ref{eq-pi-A-minus-tildeA}).

In conclusion, the first term in the r.h.s. of (\ref{fabio1}) is $o(1)$
if $c<c_1$. We now examine the second term $\nu(\mathcal{B})$ and we
proceed as in the proof of Theorem~\ref{thm-equil-shape}.
First, we claim that for any short $h$-contour $\gamma$ and
$v\in\Lambda_\gamma$, where ``short'' means of length smaller than
$\log^2(L)$, we have
%
%
\begin{equation}
\label{eq-h+1-given-C(gamma,h)} \Pi^0_{\Lambda} (
\eta_v \geq h+1\mid\mathscr{C}_{\gamma,h} ) \leq\tfrac14.
\end{equation}
Indeed, if $\mathscr{A}= \{ \gamma'\dvtx v \in\Lambda_{\gamma'}
\subseteq\Lambda_{\gamma}\}$ then an application
of~(\ref{e-contourConditional}) from Lemma~\ref{l:contourBound} shows
that
\[
\hat{\pi} {}^0_{\Lambda} (\eta_v \geq h+1\mid
\mathscr{C}_{\gamma,h} ) \leq\sum_{\gamma'\in\mathscr{A}} \hat{
\pi} {}^0_{\Lambda} (\mathscr{C} _{\gamma',h+1} \mid
\mathscr{C}_{\gamma,h} ) \leq\sum_{\gamma'\in\mathscr{A}}
e^{-\beta|\gamma'|} \leq\frac18
\]
for $\beta$ large since, as usual, the number of contours $\gamma'\in
\mathscr{A}$ of length $k$ is at most $k \mu_2^k$ (using the fact
that each of
these crosses the horizontal line to the right of $v$ within distance
at most $k$). To transfer this estimate to the setting of a floor,
observe that by Remark~\ref{rem:nested-contours},
%
%
\begin{equation}
\label{eq-h+1-given-C(gamma,h)-intermediate} \Pi^0_{\Lambda} (
\eta_v \geq h+1\mid\mathscr{C}_{\gamma,h} ) =
\frac{\hat{\pi}{}^0_{\Lambda} (\eta_v \geq h+1, \eta
\restriction_{\Lambda_{\gamma}}\geq0 \mid\mathscr{C}_{\gamma,h}
)} {
\hat{\pi}{}^h_{\Lambda_\gamma} (\eta\restriction_{\Lambda
_{\gamma}}\geq0 \mid\eta\restriction_{\Delta^+_\gamma} \geq h
)}.
\end{equation}
We have just established that the numerator is at most $1/8$, whereas
by monotonicity the denominator is at least
\[
\hat{\pi} {}^h_{\Lambda_\gamma} (\eta\restriction_{\Lambda
_{\gamma}}
\geq0 ) = \hat{\pi} {}^0_{\Lambda_\gamma} (\eta\restriction_{\Lambda
_{\gamma}}
\geq-h ) \geq1 - c e^{-4\beta(h+1)}|\Lambda_\gamma|
\]
thanks to Proposition~\ref{p-heightBounds} (with the same constant
$c>0$ appearing there) and a union bound over the sites of
$\Lambda_\gamma$. The fact that $|\Lambda_\gamma| \leq|\gamma|^2 =
O(\log^4 L)$ shows this last term is $1-L^{-1+o(1)}$, hence the effect
of the denominator in~(\ref{eq-h+1-given-C(gamma,h)-intermediate}) can
easily be countered by a factor of $2$, thus
establishing~(\ref{eq-h+1-given-C(gamma,h)}).

With inequality~(\ref{eq-h+1-given-C(gamma,h)}) available to us, the
very same concentration argument leading to (\ref{eq-30}) applies again
here to imply that
%
%
\begin{equation}
\label{eq-32} \Pi_\Lambda^0 \biggl({\sum
}'_\gamma\#\{x\in\Lambda_\gamma\dvtx
\eta_x\geq h+1\}\geq\frac12 L^2 \biggr)\leq
e^{-c_2L^{2-o(1)}}
\end{equation}
for some constant $c_2>0$, where the summation $\sum'_\gamma$ is over
every \textit{short} \mbox{$h$-}con\-tour~$\gamma$. Similarly,
following the same steps leading to (\ref{eq-33}), we get that
%
%
\begin{equation}
\label{eq-34} \Pi_\Lambda^0 \biggl({\sum
}''_{\gamma}\#\{x\in\Lambda_\gamma
\dvtx\eta_x\geq h+1\}\geq\frac12 L^2 \biggr)\leq
e^{-c_3 L}
\end{equation}
for a suitable $c_3>0$, where $\sum''_\gamma$ sums over\vspace*{-2pt}
every \textit{long} $h$-contour $\gamma$, that is, such that
$|\gamma|\geq (\log L)^2$ [in this case, the analog of
(\ref{eq-h-cont}) for $h$-contours of area smaller than $\delta L^2$
holds if $\delta$ chosen small]. Finally, Lemma~\ref{lem-no-ceil}
translates the statements on $\Pi_\Lambda^0$ into the analogous bounds
for $\pi_\Lambda^0$. In conclusion the second term in the r.h.s.
of~(\ref{fabio1}) is $o(1)$ if $c<\min(c_2, c_3)$, as required.



\subsection{Proof of Theorem~\texorpdfstring{\protect\ref{th-cascade}}{2}: Lower bound on
\texorpdfstring{$\tau_a$}{tau a}}
Here we prove that $\mathbb{P}^\sqcup(\tau_a\geq e^{c L^a})\to1$ as
$L\to\infty$ where, we recall,
\[
\Omega_a= \bigl\{\eta\mbox{ such that }\# \bigl\{x\in
\Lambda_L\dvtx\eta_x\geq a H(L) \bigr\} >\tfrac9{10}|
\Lambda_L| \bigr\}
\]
and $\tau_a$ is the hitting time of $\Omega_a$. We proceed as in the
proof of (\ref{eq-3100}) but now the height $h$ is chosen equal to
$aH(L)-1$, so that $e^{-4\beta h}\leq\exp(8\beta)L^{-a}$, and the set
$A$ is defined by
\[
A = \bigcap_{\gamma\dvtx |\Lambda_\gamma|> \delta L^{2a}} (\mathscr{C}_{\gamma,h})^c.
\]
Here $\delta$ is a small constant such that, for $|\Lambda_\gamma
|\leq
\delta L^{2a}$:
%
\begin{eqnarray*}
C_0 |\Lambda_\gamma| e^{-4\beta h} &\leq&
C_0 \bigl(\sqrt{\delta L^{2a}} \sqrt{|\gamma|^2/16}
\bigr) e^{-4\beta h} \leq\lambda|\gamma|,
\end{eqnarray*}
where $\lambda$ is analogous to~(\ref{eq-25}). As in (\ref{fabio1}), we
get
%
%
\begin{eqnarray}
\label{fabio2} \mathbb{P}^\nu \bigl(\tau_a <
e^{cL^a} \bigr)&\leq&\mathbb{P}^\nu \bigl(\tau
_{\partial A} < e^{cL^a} \bigr)+ \mathbb{P}^\nu \bigl(
\tau_a < e^{cL^a}\leq\tau_{\partial A} \bigr)
\nonumber
\\[-8pt]
\\[-8pt]
&\leq& e^{cL^a} \bigl(\nu(\partial A) +\nu(\Omega_a) \bigr).
\nonumber
\end{eqnarray}
Exactly the same arguments behind~(\ref{eq-pi-A-minus-tildeA}),
(\ref{eq-32}) and (\ref{eq-34}) now show that the r.h.s.
of~(\ref{fabio2}) is $o(1)$.\eject

\section{A bound using paths and flows}\label{sec-canonical-proof}

\subsection{Proof of Proposition~\texorpdfstring{\protect\ref{prop-canpaths}}{2.3}}\label{sec-canpaths}
Let $\Lambda:=\{1,\ldots,L\}\times\{1,\ldots,m\} $ and
$\Omega:=\Omega_{\Lambda,n^+}$. We introduce the canonical paths
$\gamma(\eta,\eta')$ from $\eta$ to $\eta'$ for every $
\eta,\eta'\in\Omega$. Define the diagonal lines in
$\Lambda_L=\{1,\ldots,L\}^2$
%
%
\begin{equation}
\label{eq-Ri} R_i=\{x\in\Lambda_L\dvtx
x_2=x_1 + L-i\}, \qquad i=1,\ldots,2L-1
\end{equation}
and let $\mathcal R$ denote the collection of the $R_i$. Number the
sites in $\Lambda$ following the lines $R_1,\ldots,R_{2L-1}$, so that
each line is read from southwest to northeast; at each site $x$ move
straight from $\eta_x$ to $\eta_x'$
by taking $|\eta_x-\eta_x'|$ unit steps. 
Note that since all heights satisfy $0\leq\eta_x\leq n^+$ one has
$|\gamma|\leq|\Lambda|n^+$. If $e=(\sigma,\sigma^{x_*,\pm})$ is
an edge
of a path, with $x_*\in R_{i_*}$, define $A$ as the set of
$x\in\Lambda$ such that $x< x_*$ and $B$ the set of $x>x_*$ (w.r.t. to
the order introduced above). Here $\sigma^{x_*,\pm}$ denotes the
configuration which coincides with $\sigma$ except that the height at
$x_*$ is changed by $\pm1$. Then by direct inspection one finds that
for any $\eta,\eta'\in\Omega$ such that $\gamma(\eta,\eta')\ni e$:
%
%
\begin{equation}
\label{canpath} \pi(\eta)\pi \bigl(\eta' \bigr) \leq\pi(\sigma )
\pi \bigl(\sigma^* \bigr)\exp{ \biggl(6\beta\sum_{x\in
R_{i_*}\cap\Lambda
}
\bigl| \eta_x-\eta'_x \bigr| \biggr)},
\end{equation}
where $\sigma$ satisfies $\sigma_A=\eta'_A$, $\sigma_B=\eta_B$, while
$\sigma ^*$ is the configuration obtained by setting
$\sigma^*_A=\eta_A$, $\sigma^*_B=\eta'_B$. Here $\sigma_{x_*}$ and
$\sigma ^*_{x_*}=\sigma_{x_*}\pm1$ are assigned according to the choice
of $e$. The crucial observation is that, given $e$, the map from
$(\eta,\eta')$ [such that $e\in\gamma(\eta,\eta')$] to
$(\sigma,\sigma^*)$ is an injective one. In particular, this implies:
%
%
\begin{equation}
\label{canpath1} \frac1{\pi(\sigma)}\sum_{\eta,\eta'\in\Omega} \bigl|
\gamma \bigl(\eta,\eta' \bigr) \bigr| {\pi(\eta)\pi \bigl(
\eta' \bigr)} {\mathbf1}_{e\in\gamma} \leq|\Lambda|n^+ \exp{
\bigl(6 \beta n^+ m \bigr)}.
\end{equation}
Note also that the inverse of the smallest nonzero one-step transition
probability for our chain is $|\Lambda|\exp(4\beta n^+)$. We apply then
(\ref{eq-18bis}) to obtain that the inverse spectral gap of the SOS
dynamics is upper bounded by\break  $c |\Lambda|^2 n^+\exp{(7 \beta n^+ m)}$
and~(\ref{eq-19}) follows.

\subsection{Proof of Theorem~\texorpdfstring{\protect\ref{th-pathflowsgeneral}}{2.4}}

For every $\xi$, $\xi'\in\Omega$, let
$\gamma_1$ be a path 
of length $T$ starting at $\xi$ and let $\gamma_2$ be a path of length
$T$ starting at $\xi'$. Write $\eta,\eta'$ for the corresponding
endpoints. Let $\gamma_c$ be a path from $\eta$ to $\eta'$ (to be
specified below) which depends only on $\eta$, $\eta'$ and not on
$\gamma_1$, $\gamma_2$. Call $\gamma$ the concatenation of $\gamma_1$,
$\gamma_c$, $\bar{\gamma_2}$, where $\bar{\gamma_2}$ is the path
$\gamma_2$, inverted in time. Note that $\gamma$ connects $\xi$ to
$\xi'$. If $\eta$, $\eta'\in G, $ then we let $\gamma_c$ be the path
$\tilde\gamma(\eta,\eta')$ which appears in the statement of the
theorem (recall that it stays in the set $G$) and define
\[
a(\gamma)= \frac{\pi(\xi)\mathbb{P}^\xi(\gamma_1)}{\mathbb
{P}^\xi(X(T)\in G
)}\frac{\pi(\xi')\mathbb{P}^{\xi'}(\gamma_2)}{\mathbb{P}^{\xi
'}(X(T)\in G )}.
\]
Otherwise, set $a(\gamma)=0$ and we do not need to specify $\gamma_c$
in this case. Here $\mathbb{P}^\xi(\gamma_1)$ is the probability
that\vspace*{1pt} the\vadjust{\goodbreak} process $(X(t))_t$ started at $\xi$ follows
exactly $\gamma_1$ up to time $T$, and similar for
$\mathbb{P}^{\xi'}(\gamma_2)$. Note that for fixed $\xi$,
$\xi'\in\Omega$, $\sum_{\gamma\dvtx\xi\sim\xi'}
a(\gamma)=\pi(\xi)\pi(\xi')$ where the sum is over $\eta$, $\eta'$,
$\gamma_1$, $\gamma_2$ for fixed $\xi$, $\xi'$.

Therefore, viewing the path $\gamma$ as a collection of oriented edges
$e=( \sigma,\sigma')$ and letting $\nabla_e f=f(\sigma)-f(\sigma
')$, we
have
%
%
\begin{eqnarray}
\operatorname{Var}(f) &=& \frac12\sum
_{\xi,\xi'}\pi(\xi)\pi \bigl(\xi' \bigr) \bigl(f(
\xi)-f \bigl(\xi' \bigr) \bigr)^2
\nonumber\\[-8pt]\\[-8pt]
&=&\frac12\sum
_{\xi,\xi'} \sum_{\gamma\dvtx\xi\sim\xi'} a(\gamma)
\biggl(\sum_{e\in\gamma} \nabla_e f
\biggr)^2\nonumber
\\
&\leq& \frac32 \sum_{\xi,\xi'} \sum
_{\gamma\dvtx\xi\sim\xi'} a(\gamma) \bigl(A_{\gamma}(f) +
B_{\gamma}(f) \bigr),\label{eq-varianza}
\end{eqnarray}
where
\begin{eqnarray*}
A_{\gamma}(f)&=& |\gamma_1|\sum
_{e\in\gamma_1}( \nabla_e f)^2 + |
\gamma_2|\sum_{e\in\gamma_2}(\nabla_e
f)^2,
\\
B_{\gamma}(f) &=& |\tilde\gamma|\sum_{e\in\tilde\gamma}(
\nabla_e f)^2
\end{eqnarray*}
and in the inequality we used Cauchy--Schwarz. Now we use the fact that
the Dirichlet form which appears in the definition (\ref{eq-gap}) of
the spectral gap can be written as
\[
\mathcal{E}(f):=\pi^0_\Lambda \bigl(f(I-P)f \bigr)=\frac12
\sum_{e=(\sigma,\sigma
')}\pi(\sigma)p \bigl(\sigma,
\sigma' \bigr) (\nabla_e f)^2.
\]
Recall that $p_{\min}$ denotes the smallest nonzero one-step transition
probability, and observe that
\begin{eqnarray*}
&& \sum_{\xi,\xi'} \sum_{\gamma\dvtx\xi\sim\xi'}
a(\gamma) A_{\gamma}(f)
\\
&&\qquad = 2\sum_{\xi} \sum
_{\gamma_1} |\gamma_1| \frac{\pi(\xi
)\mathbb{P}^{\xi}(\gamma_1)}{\mathbb{P}^\xi(X(T)\in G )}\sum
_{e\in\gamma
_1}(\nabla_e f)^2
\\
&&\qquad \leq 4\frac T{\alpha p_{\min}} \mathcal{E}(f)\sup_{e=(\sigma,\sigma')}
\bigl(\pi(\sigma)^{-1} \bigr)\sum_{\xi}
\sum_{\gamma_1} \pi(\xi) \mathbb{P} ^{\xi}(
\gamma_1) {\mathbf1}_{e\in\gamma_1},
\end{eqnarray*}
where we used the fact that $|\gamma_1|=T$.

Let $\mathbb{P}$
denote the law of the stationary process (started at equilibrium
$\pi$). From a union bound, one has
\[
\sum_{\xi} \sum_{\gamma_1}
\pi(\xi)\mathbb{P}^\xi(\gamma_1) {\mathbf
1}_{e\in\gamma_1} = \mathbb{P} \bigl( \exists t\in[0,T]\dvtx X(t)=\sigma,
X(t+1)=\sigma' \bigr)\leq T \pi(\sigma).
\]
It then follows that
\[
\sum_{\xi,\xi'} \sum_{\gamma\dvtx\xi\sim\xi'}
a(\gamma) A_{\gamma
}(f)\leq4 \frac{T^2}{\alpha p_{\min}} \mathcal{E}(f).
\]
As for the second term in (\ref{eq-varianza}), using stationarity of
$\pi$ one has that the sum of $\pi(\xi)\mathbb{P}^\xi(\gamma_1)$ over
all $\xi$ and paths $\gamma_1$ of length $T$ which connect $\xi$ to
$\eta$ gives $\pi(\eta)$, so that [with the definition (\ref{eq-W})]
\begin{eqnarray*}
&& \sum_{\xi,\xi'} \sum_{\gamma\dvtx\xi\sim\xi'}
a(\gamma) B_{\gamma}(f)
\\
&&\qquad \leq\frac{2}{\alpha^2} \frac12\sum
_{e=(\sigma,\sigma')}(\nabla_e f)^2\pi(\sigma) p
\bigl(\sigma,\sigma' \bigr)\sum_{\eta,\eta'\in G}
\frac{|\tilde\gamma
(\eta,\eta')|\pi(\eta)\pi(\eta')}{\pi(\sigma)
p(\sigma,\sigma')}{\mathbf1}_{e\in\tilde\gamma(\eta,\eta')}
\\
&&\qquad\leq\frac{2}{\alpha^2}W(G) \mathcal{E}(f). 
\end{eqnarray*}
Going back to (\ref{eq-varianza}) and to the definition of spectral gap
one immediately gets~(\ref{eq-pathsgen}).
\section{Upper bounds on equilibration times}\label{sec-mtub}

\subsection{Proof of Theorem~\texorpdfstring{\protect\ref{th-cascade}}{2}: Upper bound on
\texorpdfstring{$\tau_a$}{tau a} assuming Theorem~\texorpdfstring{\protect\ref{th-principale}}{1}}
Here we prove
that $\mathbb{P}^\sqcup(\tau_a\leq e^{c' L^a})\to1$ as $L\to\infty$
assuming $\tmix\leq e^{cL}$. The latter estimate will be proven
afterwards. Let us partition the box $\Lambda_L$ into nonoverlapping
squares $Q_i$ of side $C L^a$ with $C= \exp(4\beta K)$ where $K$ is the
constant appearing in Theorem~\ref{thm-equil-shape}. By monotonicity
the Glauber dynamics is higher than the auxiliary dynamics in which
each square $Q_i$ evolves independently from the others with $0$
boundary conditions on $\partial Q_i$. Using the assumption $\tmix\leq
e^{cL}$ and independence, it is standard to check that the mixing time
of this auxiliary dynamics is not larger than $e^{2c L^a}$ and
therefore, at time $T=e^{3c L^a}$, all the squares $Q_i$ are close to
their equilibrium (in total variation)
with an exponentially small error. 
Theorem~\ref{thm-equil-shape} implies that in each of them the density
of vertices higher than
\[
H \bigl(CL^a \bigr)-K= aH(L)
\]
is larger than $1-\varepsilon(\beta)$ with probability exponentially
close to one. In conclusion, apart from an exponentially small error,
$\mathbb{P}^\sqcup(\tau_a> e^{3c L^a})$ is bounded by the probability
that for some $i$ the square $Q_i$ has a density less that $1-
\varepsilon(\beta)$ of vertices higher than $H(CL^a)-K$. Thus, a union
bound suffices to conclude the proof.

\subsection{Proof of \texorpdfstring{$\tmix\leq e^{cL}$ for $n^+=\log L$}
{T MIX <= e cL for n+ = log L}}\label{sec-n+piccolo}
To prove the upper bound on $\tmix$ in Theorem~\ref{th-principale}, the
crucial point is to give the proof for $n^+=\log L$, so we assume this
is the case in this section. The general case $\log L\leq n^+\leq L$
can be then deduced via very soft arguments; see
Section~\ref{sec-genn+} below.\vadjust{\goodbreak}

For reasons that will be clear later, first of all we modify the SOS
model by considering the Boltzmann factor
$\exp[-\beta\mathcal{H}^\xi_{\Lambda_L}+f]$ instead of $\exp
[-\beta\mathcal{H}
^\xi
_{\Lambda_L}]$,
where $f$ is the external field term
%
%
\begin{equation}
\label{eq-10} f=\frac{1}L\sum_{y\in\Lambda_L}f_y
\qquad\mbox{with } f_y=\sum_{j=1}^{n^+-H}f_{y,j}:=
\sum_{j=1}^{n^+-H}c_j {
\mathbf1}_{\eta_y\leq H+ j}
\end{equation}
with $H=H(L)$ defined in (\ref{eq-hL}) and $c_j=\exp(-\beta j)$.
One\vspace*{-1pt} changes the partition function accordingly. We call
$\pi^{\xi,f}_{\Lambda_L} $ the corresponding equilibrium measure with
ceiling at $n^+=\log L$ and floor at $0$. Moreover, we will consider
the Glauber (heat bath) dynamics associated to
$\pi^{\xi,f}_{\Lambda_L}$.

%
%
\begin{remark}
Note that, if the b.c. are zero then the extra term $f$ in
(\ref{eq-10}) will not drastically change the global equilibrium
properties, since it tends to depress the heights that exceed the level
$H$ (and having $\eta_x\geq H+1$ is already an
unlikely event, for $\beta$ large). More precisely, 
$f$ equals the constant $(|\Lambda_L|/L) \sum_j c_j$ plus a (negative)
random term which one could prove, by refining the estimates of
Section~\ref{sec-eq}, to be of order $L\times\exp(-c \beta)$ for a
typical configuration (and therefore not extensive in the area of
$\Lambda_L$).
\end{remark}

The reason for modifying the equilibrium measure in such a peculiar way
is explained after Theorem~\ref{lem-dolore_infinito}.

%
%
\begin{lemma}
\label{lem-confrtmix} The ratio $\Delta$ of the mixing time of the
original system over the mixing time of the system modified as in
(\ref{eq-10}) satisfies for $L$ large
%
%
\begin{equation}
\label{eq-18} e^{-L}\leq\Delta\leq e^{L}.
\end{equation}
\end{lemma}

\begin{pf}
Going back to the definition (\ref{eq-gap}) of the spectral gap, it is
easy to see that the ratio $\tilde\Delta$ of relaxation times satisfies
\[
e^{-4|f|_\infty}\leq\tilde\Delta\leq e^{4|f|_\infty}
\]
with $f$ as in (\ref{eq-10}); see, for example, \cite{LPW}, Lemma
13.22, for such standard comparison bounds. Note that $|f|_\infty=O (L
e^{-\beta})$ if $\beta$ is large enough. Then (\ref{eq-18}) follows
from the comparison (\ref{eq-20bis}).
\end{pf}

Therefore, it is enough to prove Theorem~\ref{th-principale} for this
modified model. We denote its mixing time as $\tmix(L)$. It is 
important to realize that the Glauber dynamics for this modified SOS
model is still monotone (in the sense of Section~\ref{sec-mon}) and
that the FKG inequalities are still valid. This is because $f$ is the
sum of functions of a single height $\eta_x$. Therefore, we can apply
all the monotonicity arguments we need (including the Peres--Winkler
censoring inequality, Theorem~\ref{th-PW}).

%
%
\begin{definition} For $k\in\mathbb{N}$ and $a,A>0$, we define the
inductive statement $\mathcal{F}_k:=\mathcal{F}_{k,a,A}$: for every $L$
the mixing time satisfies
\[
T_{\mix}(L)\leq L^a e^{A L \log^{(k)}(L)},
\]
where $\log^{(k)}(x):=\max(1,\log(\log\cdots(x) ) )$ and $\log
(\log\cdots(x) )$ is the logarithm iterated $k$ times.
\end{definition}

%
%
\begin{theorem} Fix $\beta\geq\beta_0$ for some large enough constant
$\beta_0$, and $n^+=\log L$. Then $\mathcal{F}_k
\Rightarrow\mathcal{F}_{k+1}$ provided that $a=4$ and $A=C\beta$ for
some sufficiently large $C$. \label{th-induct}
\end{theorem}

\begin{pf*}{Proof of Theorem~\ref{th-principale} given Theorem~\ref{th-induct}}
For $k=1$, the statement $\mathcal{F}_1$ follows at once
from the ``canonical paths argument'', Proposition~\ref{prop-canpaths}
(with $a=3$ and $A=b\beta$, $b$ some explicit constant). Notice that
Proposition~\ref{prop-canpaths} applies with no change to the modified
model with the external field. Then, apply the theorem until
$\log^{(k)}(L)=1$. At that point we get the desired exponential mixing
time upper bound.
\end{pf*}

For the proof of Theorem~\ref{th-induct}, we need some notation. Recall
the definition (\ref{eq-Ri}) of the diagonal lines $R_i$. Define
$G^+_\ell\subset\Omega_L$ as the set of configurations $\eta$ such
that, for every $R\in\mathcal R$,
\[
\sum_{x\in R}[\eta_x-H]^+\leq L\ell
\]
(with $[x]^+=\max(x,0)$), $G^-_\ell\subset\Omega_L$ as the set of
configurations $\eta$ such that, for every $R\in\mathcal R$,
\[
\sum_{x\in R}[H-\eta_x]^+\leq L\ell
\]
and
finally $G_\ell\subset\Omega_L$ as the set of configurations $\eta$
such that, for every $R\in\mathcal R$,
%
%
\begin{equation}
\label{eq-Gk} \sum_{x\in R}|H-\eta_x|\leq
L\ell.
\end{equation}
Let also
%
%
\begin{equation}
\label{eq-pienodiB} \ell(k,L):=B \log^{(k)}(L)+\frac1{4\beta}\log A
\end{equation}
with $B$ a constant to be chosen sufficiently large (independently of
$\beta$) later, see discussion after (\ref{eq-23}) and
(\ref{eq-diritt}).

%
%
\begin{lemma}
\label{lem-dalbasso} Assume $\mathcal{F}_k$ with $a=4$ and $A=40B
\beta$ and take $T_1> e^{2L}$. Then
\[
\mathbb{P} \bigl(\eta^\sqcup(T_1)\in G^-_{\ell(k+1,L)
}
\bigr)\geq\tfrac34.
\]
\end{lemma}

%
%
\begin{lemma}
\label{lem-dallalto} Assume $\mathcal{F}_k$ with $a=4$ and $A=40 B
\beta$ and take $T_2> e^{B\beta L}$ with the same $B$ as in
(\ref{eq-pienodiB}). Then
\[
\mathbb{P} \bigl(\eta^\sqcap(T_2)\in G^+_{\ell(\infty,L)}
\bigr)\geq\tfrac34.
\]
\end{lemma}

Note that $\ell(\infty,L)=B+1/(4\beta)\log A$ is just a large constant.
We will actually see that, in both lemmas, the constant $\frac34$ can
be replaced by $1-o(1)$ where $o(1)$ vanishes for $L\to\infty$. We
refer to Sections~\ref{getup}~and~\ref{getdown} below for the proof of
Lemmas~\ref{lem-dalbasso}~and~\ref{lem-dallalto}, respectively.

\begin{pf*}{Proof of Theorem~\ref{th-induct} given Lemmas
\ref{lem-dalbasso}~and~\ref{lem-dallalto}}
Thanks to monotonicity, to
Lemmas~\ref{lem-dalbasso} and \ref{lem-dallalto}, and the fact that
$G^+_{\ell(\infty,L)}\subset G^+_{\ell(k+1,L)}$, we can set
$T^{\all}:=\max(e^{2L},e^{B\beta L} )$ and obtain that
%
%
\begin{equation}
\label{eq-11} \min_{\zeta}\mathbb{P} \bigl(\eta^\zeta
\bigl({T^{\all}} \bigr)\in G_{2\ell
(k+1,L)} \bigr)\geq\tfrac12.
\end{equation}
This is based on the fact that, if $\eta^1\leq\eta\leq\eta^2$ and
$\eta^1\in G^-_\ell,\eta^2\in G^+_{\ell'}$ then $\eta\in
G_{\ell+\ell'}$. Just write
\[
|\eta_x-H|=[\eta_x-H]^++[H-\eta_x]^+\leq
\bigl[\eta^2_x-H \bigr]^++ \bigl[H-\eta^1_x
\bigr]^+.
\]
At this point, we need the following consequence of
Theorem~\ref{th-pathflowsgeneral}.

%
%
\begin{proposition}\label{Glemma}
Let $\alpha,\ell>0$ and $T$ be such that
$\mathbb{P}(\eta^\zeta({T})\in G_\ell)\geq\alpha$ for all initial
configurations
$\zeta$. Then there exists a constant $c=c(\alpha,\beta)$ such
that
%
%
\begin{equation}
\label{Glemma1} \trel(L) \leq c \bigl[\exp{(15\beta\ell L)}+ L^{5\beta}
T^2 \bigr].
\end{equation}
\end{proposition}

\begin{pf} 
Theorem~\ref{th-pathflowsgeneral} gives
\[
\trel\leq\frac6\alpha \biggl(\frac{T^2}{p_{\min}}+\frac{W(G_\ell
)}\alpha \biggr),
\]
where in the definition of $W(G_\ell)$ we choose the canonical paths
introduced in Section~\ref{sec-canpaths}. We know that the inverse of
the minimal transition probability $p_{\min}$ is of order
$|\Lambda_L|\exp(4\beta n^+)$. Also, from the proof of
Proposition~\ref{prop-canpaths} and the definition (\ref{eq-Gk}) of
$G_\ell$, we see easily that $W(G_\ell)\leq\exp(15 \beta\ell L)$ and
then the claim follows.
\end{pf}

Proposition~\ref{Glemma} [applied with $T=T^{\all}$, $\ell$ replaced by
$2\ell(k+1,L)$ and recalling that $n^+=\log L$], together with
(\ref{eq-11}) and (\ref{eq-20bis}), implies that
%
%
\begin{equation}
\label{eq-23} T_{\mix}(L)\leq c'(
\beta) L^3 \bigl(L^{5\beta} \bigl(T^{\all}
\bigr)^2 + e^{30
B \beta L
\log^{(k+1)}(L)+8L \log A} \bigr).
\end{equation}
If one chooses $A=40 B\beta$ (and $B$ large but independent of
$k,\beta$), then the r.h.s. of~(\ref{eq-23}) is smaller than $L^4\exp(A
L \log^{(k+1)}(L))$ for every $L$ and the claim follows.
\end{pf*}

\subsection{Proof of \texorpdfstring{$\tmix\leq e^{c L}$ for $\log L\leq n^+\leq L$}
{T MIX <= e cL for log L <= n+ <= L}}\label{sec-genn+} Once we have the statement for $n^+=\log L$,
proving it for $\log L\leq n^+\leq L$ is quite easy, so we only sketch
the main steps. Thanks to (\ref{eq-22}), it is enough to prove that
%
%
\begin{eqnarray}
\label{eq-dagiu} \bigl\|\mu^\sqcup_t-\pi \bigr\|&\leq&
L^{-4},
\\
\label{eq-dasu} \bigl\|\mu^\sqcap_t-\pi \bigr\|&\leq&
L^{-4}
\end{eqnarray}
for some $t=\exp(O(\beta L))$. Here we write $\pi$ instead of
$\pi^0_\Lambda$ for simplicity. We first note that, if $\pi,\tilde \pi$
are the equilibria with ceiling at $n^+>\log L$ and at $\log L$,
respectively, then
%
%
\begin{equation}
\label{eq-vicini} \|\pi-\tilde\pi\|\leq L^{-c_0(\beta)}
\end{equation}
with $c_0(\beta)$ that diverges as $\beta\to\infty$. Indeed, to feel
the ceiling there must be some $x$ such that $\eta_x\geq\log L$ and
this has probability at most\break  $c|\Lambda_L|\exp(-2\beta\log L)$. This
can be seen as follows. By monotonicity lift the b.c. from $0$ to
$(\log L)/2$. In this situation, the probability that the SOS interface
reaches either height $0$ or $\log L$ is $O(|\Lambda_L| e^{-2\beta
\log
L})$, as follows from Proposition~\ref{p-heightBounds} and a union
bound, cf. the proof of Lemma~\ref{lem-no-ceil}.

As for (\ref{eq-dagiu}), from Theorem~\ref{th-PW} (applied with $k=1$,
$t_1=t$, $V_1=\Lambda_L$, $a_1=0$, $b_1=\log L$) we have $\| \mu^\sqcup
_t-\pi\| \leq\| \tilde\mu{}^\sqcup_t-\pi\|$, with $\tilde\mu_t$ the law
of the evolution $\tilde\eta(t)$ with ceiling at $\log L$.
Since we proved in Section~\ref{sec-n+piccolo} that the mixing time of
the dynamics $\tilde\eta(t)$ is $\exp(O(\beta L))$, if $t=\exp (c\beta
L)$ with $c$ large one gets from (\ref{eq-sottomol}) that
$\|\tilde\mu{}^\sqcup_t-\tilde\pi\|=o(L^{-4})$ and therefore
$\|\tilde\mu{}^\sqcup_t-\pi\|=O(L^{-c_0(\beta)})+\|\tilde\mu
^\sqcup_t-\tilde\pi\|=o(L^{-4})$ if $\beta$ is large enough.

As for (\ref{eq-dasu}), assume for definiteness that $n^+$ is a
multiple of $\log L$ and let
\[
h_i=n^+ - \frac{i-1}2 \log L,\qquad i=1,\ldots, M:=
\frac{2n^+}{\log L}-1.
\]
Let us apply Theorem~\ref{th-PW} with $k=M$, $V_i=\Lambda_L$,
$t_i=i\exp(c\beta L)$ with $c$ large enough, $b_i=h_i$ and
$a_i=h_i-\log L$. Let us\vspace*{1pt} also call $U_i$ the event that
$a_i\leq\eta_x\leq a_i+\frac12 \log L$ for all $x\in\Lambda_L$. Note
that, for the associated modified dynamics $\tilde\eta(t)$, in the time
interval $0< t\leq t_1=\exp(c\beta L)$ the floor is at height
$a_1=n^+-\log L$ and the ceiling at height $b_1=n^+$. Therefore, if $c$
is chosen large enough, at time $\exp(c\beta L)$ the system is within
variation distance say $e^{-L}$ from the equilibrium with such
floor/ceiling and in particular, except with probability smaller than
$L^{-c_0(\beta)}$, the configuration is in $U_1$ [the proof of this is
very similar to the proof of (\ref{eq-vicini}) above]. If
$\tilde\eta(t_1)\in U_1$, then in the second time-lag
$\{t_1+1,\ldots,t_2\}$ the situation is similar, except that the floor
is now at $a_2$ and the ceiling is at $b_2$ (note that if instead
$\tilde\eta(t_1)\notin U_1$ then some heights are frozen forever to
values larger than $a_1+\frac12 \log L$ and the dynamics
$\tilde\eta(t)$ will not approach equilibrium). The argument is
repeated $M$ times with the result that (via a union bound on $i$), at
time $t_M=\exp(O(\beta L))$, the variation distance from equilibrium is
smaller than $M L^{-c_0(\beta)}\ll L^{-4}$ and the proof is concluded.

\subsection{Rising from the floor: Proof of Lemma~\texorpdfstring{\protect\ref
{lem-dalbasso}}{6.5}}\label{getup}

We will make a union bound on $R_i\in\mathcal R$, that is, on
$i=1,\ldots,2L-1$.
We want to upper bound
%
%
\begin{equation}
\label{primo} \mathbb{P} \biggl(\sum_{x\in R_i} \bigl[H-
\eta^\sqcup_x(T_1) \bigr]^+ \geq L\ell(k+1,L)
\biggr).
\end{equation}

%
\begin{quatation}
In principle, the argument is very simple. Around every point $x\in
R_i$ one would like to consider a square $Q_x$ of side $L_k:=L/(A
\log^{(k)}(L))$. By monotonicity, the quantity
$[H-\eta^\sqcup_x(T_1)]^+ $ appearing in (\ref{eq-06}) gets larger if
we fix to $0$ the heights on $\partial Q_x$. From the
assumption~$\mathcal{F}_k$, we know that the mixing time in $Q_x$, with
zero b.c. on $\partial Q_x$, is of order $\exp( A
L_k\log^{(k)}(L_k))\approx \exp(L)\ll T_1\approx\exp(2L)$. Thus, at
time $T_1$ the dynamics in $Q_x$ is essentially at equilibrium (w.r.t.
zero b.c. on $\partial Q_x$), so that $\eta_x\sim1/(4\beta)\log
L_k\approx H-(1/4\beta)\log^{(k+1)}(L)$ w.h.p. By taking the constant
$B$ appearing in (\ref{eq-pienodiB}) large enough, we can make
$\ell(k+1,L)\gg(1/4\beta)\log^{(k+1)}(L)$. As a consequence, the event
in (\ref{primo})
describes a very unlikely 
deviation.

In practice, the proof is considerably more involved, in particular
because the size of the squares $Q_x$ has to be chosen as a function of
$x$ [cf.~(\ref{eq-mizzica})] in order to guarantee that $Q_x$ is fully
contained in the original domain $\Lambda_L$.
\end{quatation}
Set $\sigma_x:=\frac1{4\beta}\log d(x)$, where $d(x)$ is the $L^1$
distance of $x$ from the boundary of~$\Lambda_L$. One has
%
%
\begin{equation}
\label{eq-36} \bigl[H-\eta^\sqcup_x(T_1)
\bigr]^+\leq \bigl[\sigma_x-\eta^\sqcup_x(T_1)
\bigr]^++|\sigma_x-H|.
\end{equation}
Now, there exists $C_1$ such that for every $R_i\in\mathcal R$ one has
%
%
\begin{equation}
\label{eq-4} \sum_{x\in R_i}|\sigma_x-H|\leq
\frac{C_1}{\beta}L.
\end{equation}
By the way, this is the reason why we defined the lines $R_i$ as in
(\ref{eq-Ri}): if $R_i$ were parallel to the coordinate axes and too
close to the boundary of $\Lambda_L$, then (\ref{eq-4}) would be
false. 
To prove (\ref{eq-4}), suppose without loss of generality that the
diagonal line under consideration is $R_i$ with $i\leq L$, so that
$|R_i|=i$. One has
%
%
\begin{equation}
\label{eq-1} \quad\sum_{x\in R_i}|\sigma_x-H| = i
H -\sum_{x\in R_i}\sigma_x \leq\frac1{4
\beta}i \biggl(\log L -\frac1i\sum_{x\in R_i}\log d(x)
\biggr).
\end{equation}
%
If $i$ is even
\[
\sum_{x\in R_i}\log d(x) = 2\sum
_{k=1}^{i/2}\log k = i\log i + O(i)
\]
and a similar argument takes care of the case where $i$ is odd.
Therefore,
\[
\sum_{x\in R_i}|\sigma_x-H| \leq\frac1{4
\beta}L \biggl[ \frac iL\log(L/i)+C'\frac iL \biggr]\leq
\frac{C'' L}\beta
\]
for some constants $C',C''>0$ independent of $i,\beta,L$.

Let us go back to estimating (\ref{primo}). It is clear that the $x$
such that $d(x)\leq L/\log L$ can give altogether a contribution to
$\sum_x[\sigma_x-\eta_x^\sqcup(T_1)]^+$ which is at most $O(L)$. Then,
let $\tilde R_i$ be the subset of $R_i$ such that $d(x)>L/\log L$. We
can conclude that it is enough to estimate
%
%
\begin{eqnarray}
\label{eq-06} 
&& \mathbb{P} \biggl(\sum
_{x\in\tilde R_i} \bigl[ \sigma_x-\eta^\sqcup_x(T_1)
\bigr]^+ \geq L \bigl(\ell(k+1,L)-C'''
\bigr) \biggr)
\nonumber
\\[-8pt]
\\[-8pt]
&&\qquad \leq\mathbb{P} \biggl(\sum_{x\in\tilde R_i} \bigl[
\sigma_x-\eta^\sqcup_x(T_1) \bigr]^+
\geq\frac{|\tilde R_i|}2\ell(k+1,L) \biggr).
\nonumber
\end{eqnarray}

Now for every $x\in\tilde R_i$ define a (diagonal) interval
$I_x\subset R_i$, centered at $x$ and of length
%
%
\begin{equation}
\label{eq-mizzica} |I_x|= \frac1{2}\min \biggl(d(x),\frac L{ A
\log^{(k)}(L)} \biggr).
\end{equation}
Note that the minimal $|I_x|$ is of order $L/\log L$ and the maximal
one is at most $L/(2 A \log^{(k)}(L))$ [it can be much shorter if
$|R_i|\ll L/\log^{(k)}(L)$]. Note that condition (\ref{eq-mizzica})
guarantees that around each $I_x$ one can place a square $Q_x$ of side
$m_x=2|I_x|$ and fully contained in $\Lambda_L$. Considering all the
possible $i\leq L$ and the different intervals $I_x,x\in\tilde R_i$,
the number of such intervals is trivially smaller than~$|\Lambda_L|$.
Therefore, observing that $\tilde R_i$ can be covered by (possibly
overlapping) such intervals $I_x$ of total length at most $(3/2)|\tilde
R_i|$, it is enough to prove
%
%
\begin{equation}
\label{eq-6} \mathbb{P} \biggl(\sum_{y\in I_x} \bigl[
\sigma_y-\eta^\sqcup_y(T_1)
\bigr]^+ \geq\frac{|I_x|}3\ell(k+1,L) \biggr) \leq L^{-3}
\end{equation}
for every such interval and then apply a union bound to get that the
r.h.s. of (\ref{eq-06}) is $o(1/L)$, so that after summing over the
index of $R_i$ the probability in (\ref{primo}) is still $o(1)$.

It is easy but crucial to check that
%
%
\begin{eqnarray}\label{eq-7}
\quad&& \sum_{y\in I_x} \bigl[\sigma_y-
\eta^\sqcup_y(T_1) \bigr]^+
\nonumber\\[-8pt]\\[-8pt]
&&\qquad \leq\sum _{y\in I_x} \bigl[H(m_x)-\eta^\sqcup_y(T_1)
\bigr]^++\frac c\beta|I_x| \bigl(\log A+\log^{(k+1)}(L)
\bigr),\nonumber
\end{eqnarray}
where of course, as in (\ref{eq-hL}), $H(m_x)= 1/(4\beta)\log m_x$ is
just the typical equilibrium height of the SOS interface in the center
of the square $Q_x$ with zero boundary conditions on $\partial Q_x$
(here, for lightness of notation, we forget the integer part in the
definition of $H$). Indeed, since $m_x\leq d(x)/2$ one has for $y\in
I_x$
%
%
\begin{equation}
\label{eq-cheppalle} \bigl|\sigma_y-H(m_x)\bigr|=\frac1{4\beta}
\bigl(\log d(y)-\log(m_x) \bigr).
\end{equation}
If $\min(d(x), L/{ A \log^{(k)}(L)})=d(x)$, then the r.h.s. of
(\ref{eq-cheppalle}) is upper bounded by a constant. In the opposite
case, it is bounded by
\[
\frac1{4\beta} \biggl[\log L-\log \biggl(\frac L{ A \log^{(k)}(L)}
\biggr) \biggr] \leq\frac1{4\beta} \bigl(\log A+\log^{(k+1)}(L) \bigr)
\]
and (\ref{eq-7}) follows.
Therefore, it is enough to bound
%
%
\begin{equation}
\label{eq-diritt} \mathbb{P} \biggl(\sum_{y\in I_x}
\bigl[H(m_x)-\eta^\sqcup_y(T_1)
\bigr]^+\geq C_0|I_x| \biggr)\leq L^{-3}
\end{equation}
for all such intervals, for some $C_0$ independent of $\beta$. We can
assume
that $C_0$ is large [just choose $B$ large in (\ref{eq-pienodiB})].

Monotonicity implies that if we let evolve \textit{only the heights
inside $Q_x$} with \mbox{$0$-b.c.} on $\partial Q_x$, then the random
configuration obtained at time $T_1$ is stochastically lower than the
configuration obtained via the true evolution (where \textit{all the
heights} are updated). Again by monotonicity [the event in
(\ref{eq-diritt}) being decreasing] we can lower the ceiling in the box
$Q_x$ from height $n^+=\log L$ to height $\log m_x$ and also replace
the pre-factor $(1/L)$ with $(1/m_x)$ in front of the fields $f_y,y\in
Q_x$ in (\ref{eq-10}): the dynamics thus obtained (that we simply call
``the auxiliary dynamics'') gets stochastically lower. The reason is
that the fields $f_y$ are decreasing functions of $\eta$, which tend to
``push down'' the interface, and $1/m_x>1/L$, so that
$\exp((1/L-1/m_x)f_y)$ is an increasing function.

Since we are assuming that $\mathcal{F}_k$ holds (with $a=4$), the
mixing time
of the auxiliary dynamics in $Q_x$ (with $0$-b.c. on $\partial Q_x$) is
at most
\[
m_x^4\exp \bigl(A m_x \log^{(k)}(m_x)
\bigr)\leq L^4\exp(L).
\]
As a consequence, using (\ref{eq-sottomol}), at time $T_1=e^{2L}$ the
law of the auxiliary dynamics is within variation distance
$\exp({-e^{L/2}})$ from its invariant measure, call it $\pi_{Q_x}$,
which is nothing but a space translation of
$\pi^{0,f}_{\Lambda_{m_x}}$, where\vspace*{-2.5pt} we recall that, for a
generic~$L$, $\pi^{0,f}_{\Lambda_L}$ is the equilibrium measure in
$\Lambda_L$ with the field $f$, the floor/ceiling constraints
$0\leq\eta\leq\log L$ and b.c. at zero. For simplicity, for the rest of
this subsection, we shift the square $\Lambda_L$ so that its center
coincides with the origin of $\mathbb{Z}^2$.

In conclusion,
\begin{eqnarray*}
&& \mathbb{P} \biggl(\sum_{y\in I_x} \bigl[H(m_x)-
\eta^\sqcup_y(T_1) \bigr]^+\geq C_0 |I_x| \biggr)
\\
&&\quad  \leq e^{-e^{L/2}} +
\pi^{0,f}_{\Lambda_{m_x}} \biggl(\sum_{y\in I} \bigl[H(m_x)-\eta_y
\bigr]^+ \geq C_0 |I| \biggr)
\end{eqnarray*}
and $I$ is a diagonal segment of cardinality $|I|=|I_x|=m_x/2$,
centered at the origin of $\mathbb{Z}^2$. Thus, we need the following
equilibrium estimate.

%
\begin{lemma}
\label{lem-second} For any $m$, if $I$ is a diagonal segment of length
$|I|=m/2$ centered at the origin of $\mathbb{Z}^2$, then:
%
%
\begin{equation}
\label{eq-99} \pi^{0,f}_{\Lambda_m}(\mathcal{B}):=
\pi^{0,f}_{\Lambda_m} \biggl(\sum_{y\in I}
\bigl[H(m)-\eta_y \bigr]^+ \geq C_0 |I| \biggr) \leq c
\exp(-\beta m/c),
\end{equation}
where $c>0$ is a constant and $\Lambda_m$ denotes the side-$m$ square
centered at the origin.
\end{lemma}

This will then be applied with $m$ ranging from order $L/\log L$ to
order $L/\log^{(k)}(L)$ so in all cases the r.h.s. is much smaller than
$L^{-3}$ and, putting everything together, the inequality
(\ref{eq-diritt}) and therefore the claim of Lemma~\ref{lem-dalbasso}
follows.

\begin{pf*}{Proof of Lemma~\ref{lem-second}}
Suppose this is true for the model without the field $f$, that is, for
the standard SOS measure $\pi_{\Lambda_m}^0$ of (\ref{eq-Gibbs}).
Then,\vspace*{-1pt} the same estimate follows (for $\beta$ large, with
$c$ replaced by $c/2$) for $\pi^{0,f}_{\Lambda_m}$. This is so because,
uniformly,
\[
\frac1m\sum_{y\in\Lambda_m}f_y\leq
c'm e^{-\beta/c'}
\]
for some $c'$ independent of $\beta$. To show that
$\pi_{\Lambda_m}^0(\mathcal{B})$ is small, one first proves that
%
%
\begin{equation}
\label{eq-13} \hat\pi{}^{H(m)}_{\Lambda_m} (\mathcal{B} ) \leq\exp
\bigl(-(C_0/4)\beta m \bigr)
\end{equation}
say for every $|I|$ of size between $\frac12 m$ and $\frac23 m$, where
we recall\vspace*{-1pt} from Section~\ref{sec-model} that
$\hat\pi{}^{H(m)}_{\Lambda_m}$ is the SOS measure without floor/ceiling
and\vspace*{1pt} boundary conditions at height $H(m)$. This is based on
Peierls-type arguments and the proof is relegated to
Appendix~\ref{app-13}.

We conclude the proof of Lemma~\ref{lem-second} assuming (\ref{eq-13}).
Define $\Delta_i$, $i= 1,\ldots,m/2-1$ to be the boundary of the square
of side $m-2i$ centered at zero. Let $E_i$ be the event
%
%
\begin{equation}
\label{eq-12} E_i= \biggl\{\sum_{x\in\Delta_i}
\bigl[H(m)-\eta_x \bigr]^+\geq\delta m \biggr\}
\end{equation}
for some $\delta$ to be chosen small later. Suppose that at least one
of the $E_i, i\leq m/10$ is not realized, and let $j$ be the smallest
such $i$. In that case, we look at the
\mbox{$\pi^0_{\Lambda_m}$-}probability of $\mathcal{B}$, conditionally
on the configuration of $\eta$ on $\Delta_j$. For all $x\in\Delta_j$,
if $\eta_x>H(m)$ we can lower it to $H(m)$ by monotonicity (the event
$\mathcal{B}$~is decreasing). If instead $\eta_x<H(m)$, we still change
$\eta_x$ by brute force to $H(m)$: the price to pay is that in the
final estimate we get a multiplicative error
\[
\exp \biggl( c\beta\sum_{x\in
\Delta_j} \bigl[H(m)-
\eta_x \bigr]^+ \biggr) \leq e^{c\beta\delta m}
\]
for some explicit $c$ (independent of $\beta$ and $\delta$). What we
get is that, conditionally on $j\leq m/10$ being the smallest index
such that $E_j$ is not realized, the
\mbox{$\pi^0_{\Lambda_m}$-}probability of $\mathcal{B}$ is upper
bounded by
%
%
\begin{equation}
\label{eq-dep} e^{c\beta\delta m}\hat\pi{}^{H(m)}_{\Lambda_{m-2j}}(
\mathcal{B}|0\leq\eta\leq\log m)\leq e^{c\beta\delta m}\hat\pi{}^{H(m)}_{\Lambda_{m-2j}}(
\mathcal{B}| \eta\leq\log m),
\end{equation}
where the inequality is just monotonicity.
Notice that $\hat\pi{}^{H(m)}_{\Lambda_{m-2j}}(\eta\leq\log m)$ is
large (say, larger than $1/2$, cf. Proposition~\ref{p-heightBounds}).
Then,\vspace*{1pt} we can apply~(\ref{eq-13}), since the interval $I$
we are looking at is of length $m/2$, so that certainly
$\frac12(m-2j)\leq|I|\leq \frac23(m-2j)$ and we get that the r.h.s. of
(\ref{eq-dep}) is upper bounded by
\[
\exp \bigl(c\beta\delta m-(C_0/4)\beta(m-2j) \bigr).
\]
At this point it is enough to choose $\delta$ small enough, for
instance, $\delta=C_0/(20c)$, to conclude (recall that $j\leq m/10$).

Next, we have to show that
%
%
\begin{equation}
\label{eq-14} \pi^0_{\Lambda_m} \Biggl(\bigcap
_{i=1}^{m/10} E_i \Biggr)
\end{equation}
is very small. Indeed, that event implies that
%
%
\begin{equation}
\label{eq-15} \sum_{x\in\Lambda_m} \bigl[H(m)-
\eta_x \bigr]^+\geq\delta m^2/10=C_0m^2/(200c).
\end{equation}
Write
%
%
\begin{equation}
\label{eq-16} \sum_{x\in\Lambda_m} \bigl[H(m)-
\eta_x \bigr]^+=\sum_{k>0}k
N_k,
\end{equation}
where $N_k$ is the number of points where $[H(m)-\eta_x]^+=k$. From
Theorem~\ref{thm-equil-shape}, we know that there exists some integer
$K$ such that $N_k\leq m^2e^{-2\beta k}$, except with probability
$\exp(-m\exp(\beta k))$, for $k\geq K$. Then, except with probability
of order $\exp(-c\beta m)$ one has $\sum_{k\geq1} k N_k<C_0m^2/(200c)$
if $C_0$ is chosen large enough [recall that, as discussed after
(\ref{eq-diritt}), we can assume that $C_0$ is large].
\end{pf*}




\subsection{Falling down from the ceiling: Proof of Lemma~\texorpdfstring{\protect
\ref{lem-dallalto}}{6.6}}\label{getdown} This is the part which requires the
more subtle equilibrium estimates. Let $T_2=\exp(c\beta L)$ where $c$
will be determined along the proof. We want to prove that
%
%
\begin{equation}
\label{eq-9} \mathbb{P} \bigl(\eta^\sqcap({T_2})\in
G^+_{\ell(\infty,L)} \bigr)>\tfrac34.
\end{equation}
We recall that $\ell(\infty,L)=B+1/(4\beta)\log A$ is a constant that
we can assume to be large. For simplicity, we write $\ell$ instead of
$\ell(\infty,L)$.

%
\begin{quatation}
Ideally the proof would work as follows. At equilibrium, the event
$G^+_\ell$ has probability almost $1$, see Lemma~\ref{lemma:snello}
below (since $G^+_\ell$ is decreasing, in Lemma~\ref{lemma:snello} we
lift the boundary conditions on $\partial\Lambda_L$ from $0$ to
${H'}=H+1$, the reason for the ``${+}1$'' being that, in this way, for
$\beta$ large the floor has little influence on the interface at the
typical height ${H'}$.) It is therefore sufficient to prove that at
time $T_2$ the dynamics (with b.c. $0$) is close to equilibrium. For
this purpose we will apply Theorem~\ref{th-PW} (which is allowed since
we start from the maximal configuration $\sqcap$) with the following
censoring schedule.

Cover $\Lambda_L$ with overlapping, parallel rectangles $V_i$, $i\leq
M=O(\log L)$, ordered from left to right, with longer vertical side $L$
and shorter horizontal side $(L/(\log L))$ and such that $V_i\cap
V_{i+1}$ is a rectangle $L\times(L/(2\log L))$. Now consider the
``bricks'' $B_i$ which have base $V_i$ and height $n^+=\log L$.

We first let $B_1$ evolve for a time $t_1=\exp((c/2)\beta L)$. This is
the SOS dynamics with b.c. $0$ on the left, top and bottom boundary of
$V_1$, and with b.c. $n^+$ on the right boundary. As we justify below,
we can pretend that at time $t_1$, the system in $B_1$ has reached its
own equilibrium. This equilibrium, restricted say to the left half of
$B_1$, should be extremely close to the true equilibrium in $\Lambda_L$
with $0$ b.c. This can be justified as follows. The b.c. around $V_1$
impose the presence of open contours at heights $1,\ldots, n^+$, with
endpoints at the endpoints of the r.h.s. of $V_1$. These
contours behave roughly like random walks and will stay within distance
say $L^{1/2+\varepsilon}$ from the r.h.s. of $V_1$ and only
with tiny probability will intersect the left half of $V_1$.

Next, we let $B_2$ evolve for the same amount of time $t_1$, after
which a similar argument shows that the ``true equilibrium'' is reached
in the left half of $V_2$, that is, on the right half of $V_1$, and so
on. When the $M$th block has been updated, the system should be very
close to equilibrium everywhere. In practice, there are two major
obstructions that prevent this strategy from being implemented directly
and which cause much technical pain. The first has to do with the
presence of the floor constraint at zero and will be discussed in
greater detail in the Reader's Guide~\ref{externalf} below. The second
difficulty can be understood in the following simplified
situation.

Take the SOS in a $L\times m$ rectangle $R$, with $\sqrt L\ll m\ll L$
(for us, $R$ would be $V_1$ so that $m=L/\log L$) with b.c. $1$ on one
of the size-$L$ sides and b.c. $0$ everywhere else, without any
floor/ceiling. There is an open $1$-contour joining the endpoints of
the side with $1$ b.c. The probability of such contour $\gamma$ can be
shown, via cluster expansion, to be proportional to
$\exp(-\beta|\gamma|+\Psi_R(\gamma))$ where the ``decoration'' term
$\Psi_R(\gamma)$ is of order $|\gamma|$ times a constant which is small
with $\beta$. In absence of decorations, $\gamma$ would behave as a
random walk and it would be very unlikely that it reaches distance $\gg
\sqrt L$ from the side with $1$-b.c. In presence of the decorations,
this might \textit{in principle} fail. Indeed, the decorations depend
also on how close the contour is to the boundary of $R$ (see
Appendix~\ref{clusters}), and this could induce a pinning effect of the
contour on the size-$L$ side with $0$-b.c. The way out we found to
exclude this scenario is a series of monotonicity arguments which in
practice boil down to transforming $R$ into a rectangle with
\textit{both} sides of order $L$. In this situation, since the side
with $1$-b.c. is very far from the opposite side, the ``pinning
effect'' can be shown not to occur.
\end{quatation}

To prove (\ref{eq-9}) we couple $\eta^\sqcap({T_2})$ with a suitable
equilibrium distribution as follows. Let $\Lambda$ be the $2L\times L$
rectangle obtained by attaching a square of side $L $ to the left of
the original square $\Lambda_L$. Let $\pi^{{H',f}}_{\Lambda}$ denote
the SOS equilibrium distribution in $\Lambda$ with boundary conditions
${H'}:=H+1$. Such equilibrium measure contains the field $f$, cf.
(\ref{eq-10}) (where the sum now is over $y\in\Lambda$ and the
pre-factor is still $1/L$) and floor/ceiling constraints $0\leq\eta
\leq n^+$. One has the following
lemma.%

%
\begin{lemma} \label{lemma:snello} If $\ell$ is large enough,
then
%
%
\begin{equation}
\label{eq-9.1} \lim_{L\to\infty}\pi^{H',f}_\Lambda
\bigl(\eta\in G^+_{\ell} \bigr)=1.
\end{equation}
\end{lemma}

The proof is deferred to Appendix~\ref{app-13}.

Therefore, using that the event $G^+_{\ell}$ is decreasing,
(\ref{eq-9}) follows if we prove that there exists a coupling of
$(\eta,\eta^\sqcap({T_2}))$, where $\eta$ is the restriction to
$\Lambda_L$ of the configuration distributed according to $
\pi^{{H',f}}_\Lambda$, such that
%
%
\begin{equation}
\label{eq-9.2} \mathbb{P} \bigl(\eta^\sqcap({T_2})\leq\eta
\bigr)=1+o(1).
\end{equation}
To this end, we will apply Theorem~\ref{th-PW} with exactly the
censoring described above. We first let evolve the system in $B_1$ for
a time-lag $t_1$, with $n^+$ b.c. on the r.h.s. of~$V_1$ and
$0$ b.c. elsewhere.\vadjust{\goodbreak} Then we let evolve the system in $B_2$, for another
time-lag $t_1$. For $B_2$ we have the maximal b.c. $n^+$ on the right
boundary, zero b.c. on top and bottom and the b.c. on the left boundary
is given by the configuration, say $\tau_1$,
inherited from the previous evolution on $B_1$. We repeat this
procedure for the other bricks $B_{i}$, $i<M$, with maximal b.c. on the
right boundary, zero b.c. on top and bottom and the b.c. $\tau_{i-1}$
on the left boundary; the final brick $B_M$, unlike the previous ones,
has a zero b.c. on the right boundary as well as on the top and bottom
boundaries, and b.c. $\tau_{M-1}$ on the left boundary.

We let $\tilde\eta$ denote the configuration at the end of the above
described procedure. Note that altogether the time spent is $M t_1\leq
T_2=\exp(c\beta L)$.
Theorem~\ref{th-PW} implies that we can couple $\tilde\eta$ and $
\eta^\sqcap({T_2})$ in such a way that $\mathbb{P}(\eta^\sqcap
({T_2})\leq\tilde\eta)=1$. Thus, it remains to prove that
(\ref{eq-9.2}) is satisfied with $\tilde\eta$ replacing
$\eta^\sqcap({T_2})$.

The mixing time of a brick is bounded above by $\exp((c/4)\beta L)$,
for a suitable choice of $c>0$, see Proposition~\ref{prop-canpaths}.
Therefore, after time $t_1$ the chain is extremely close to its
equilibrium in $B_1$ with the given boundary conditions. Up to a global
error term of order $e^{-L}$ we can thus assume that after each
updating of a brick, the corresponding random variable is given exactly
by the equilibrium distribution on that brick with the prescribed
boundary conditions [see equation (\ref{eq-sottomol})]. Let
$\tilde\eta_i$ denote the configuration after the updating of brick
$B_i$, restricted to the left half of the brick, that is, the brick
with basis $V_i':=V_i\cap(V_{i+1})^c$. Thus,
using monotonicity, it is sufficient to exhibit a coupling such that
%
%
\begin{equation}
\label{wtetai} \mathbb{P}(\tilde\eta_i\leq\eta_i, i=1,
\ldots,M-1)=1+o(1),
\end{equation}
where $\eta_i$ denotes the configuration $\eta$ with distribution
$\pi^{{H',f}}_{\Lambda}$, restricted to~$V_i'$.

To prove the latter estimate, we proceed as follows. Let $\mathcal
{V}^i$ denote the portion of $\Lambda_L$ covered by rectangles
$V_1,\ldots,V_i$, and set $\mathcal{V}^0:=V'_1$. For $i=0,\ldots,M$,
call $\Lambda^i$ the rectangle obtained by attaching a square of side
$L$ to the left of $\mathcal{V}^i$ (this corresponds to the ``rectangle
enlarging procedure'' outlined above), and let $\xi$ denote the b.c.
equal to:
%
%
\begin{equation}
\label{bcxi} \xi_x= %
\cases{ n^+, &\quad if $x$ belongs
to the right boundary of $\Lambda^i$, \vspace*{2pt}
\cr
H', &\quad otherwise.} %
\end{equation}
Since $V_i'\subset\Lambda^{i-1}$, by monotonicity and a repeated
application of the DLR property for the measure $\pi^{{H',f}}_{\Lambda
}$, we see that the desired claim (\ref{wtetai}) is a consequence of
the next equilibrium result.

%
%
\begin{theorem}
\label{lem-dolore_infinito} For every $C>0$, there exists $\beta_0$
such that for all $\beta\geq\beta_0$, for all $i=1,\ldots,M$,
%
%
\begin{equation}
\label{eq-vera} \bigl\|\pi^{\xi,f}_{\Lambda^i}-\pi_{\Lambda}^{{H',f}}
\bigr\|_{\Lambda^{i-1}}\leq L^{-C},
\end{equation}
where $\|\cdot\|_{\Lambda^i}$ denotes total variation of the marginal on
$\Lambda^i$.
\end{theorem}

%
\begin{quatation}\label{externalf}
We now explain why~(\ref{eq-vera}) should be true and why we crucially
need the field $f$, which is absent in the standard SOS
measure~(\ref{eq-Gibbs}). For simplicity, suppose that the boundary
height $\xi$ at the right vertical side of $\Lambda^i$ is
\vadjust{\goodbreak} ${H'}+1$ instead of $n^+=\log L$. There is an open
$({H'}+1)$-contour with endpoints at the endpoints of the side with
b.c. ${H'}+1$. The probability that this contour equals $\gamma$ should
be approximately given by the product of three factors:
\begin{longlist}[(iii)]
\item[(i)] the factor $\exp(-\beta(|\gamma|-L))$ (the minimal length
of the open contour is $L$ and one pays for the excess length);

\item[(ii)] a factor $\exp( +a(\beta) A(\gamma)/L)$ [with $A(\gamma)$
    the area to the right of the contour]; this is due to the entropic
    repulsion and $a(\beta)$ should be approximately
    $a(\beta)=\exp(-4\times2\times\beta)$, where the factor $2$ is due
    to the fact that ${H'+1}-H=2$;

\item[(iii)] $\exp(-b(\beta) A(\gamma)/L)$ where $b(\beta)$ is
approximately given (for $\beta$ large) by
$b(\beta)=c_2(\beta)=\exp(-2\beta)$ which appears in (\ref{eq-10}).
\end{longlist}
Therefore, if $\beta$ is large the third term beats the second one and
one pays both excess length and excess area, and it should be very
unlikely that the contour reaches distance $L/(\log L)\gg\sqrt{L}$ from
the right rectangle side to which it is attached. We will find this
probability to be roughly as small as $\exp(-cL/(\log L)^2)$, as would
be the case for a random walk. Once we know the contour $\gamma$ does
not go much farther than $\sqrt L$ away from the side of the rectangle,
a suitable coupling argument will prove the theorem; see
Section~\ref{accoppiamento}. Remark that without the Hamiltonian
modification (\ref{eq-10}) (i.e., with $f_y\equiv0$) the area gain
kills the length penalization, and the contour \textit{would indeed}
invade the rectangle $\Lambda^i$.
\end{quatation}

\section{Proof of Theorem~\texorpdfstring{\protect\ref{lem-dolore_infinito}}{6.12}}\label{sec-dolore}
The proof of Theorem~\ref{lem-dolore_infinito} is based on the
following lemma. Fix $i=1,\ldots,M$ and set $R:=\Lambda^i$,
$R':=\Lambda^{i-1}$, so that the rectangle $R\setminus R'$ has
horizontal length $2\ell$, where $\ell:=L/(4\log L)$. Let also $R''$
denote the rectangle of points in $R$ at distance at least $\ell$ from
the right boundary. Note that $R\supset R''\supset R'$ and
$d(R\setminus R'',R')=\ell$, see
Figure~\ref{BBfig}.
%
%
\begin{figure}

\includegraphics{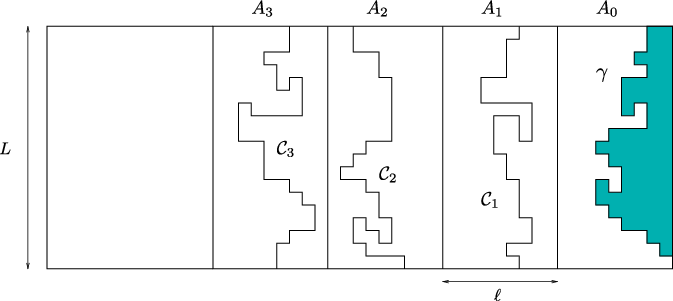}

\caption{Schematic drawing of the rectangles $R$, $R''=R\setminus
A_0$, $R'=R''\setminus A_1$. Here the contour $\gamma=\gamma_{{H'}+1}$
illustrates the event $\mathcal{B}$ in Lemma~\protect\ref{lem-1}, while
the chains $\mathcal{C}_i$ in $A_i$ illustrate the vertical crossings used in the proof
of Theorem~\protect\ref{lem-dolore_infinito}. The shaded region
corresponds to $\operatorname{Int}(\gamma)$, while $\Lambda(\gamma)=R\setminus
\operatorname{Int}(\gamma)$. The boundary between $A_2$ and $A_1$ is
$\partial R'$.}\label{BBfig}
\end{figure}

Let $\gamma_{j}(\eta)$, $j={H'}+1,\ldots,n^+$ denote the unique open
$j$-contour in the rectangle $R$ enforced by the boundary conditions,
attached to the right boundary.

%
%
\begin{lemma}
\label{lem-1} Let $\mathcal{B}$ be the event that $\gamma
_{{H'}+1}(\eta)$ does not intersect the rectangle~$R''$. For every
$C>0$, $\Pi_{\Lambda^i}^{\xi,f}(\mathcal{B}^c)=O(L^{-C})$ where $\Pi
_{\Lambda^i}^{\xi,f}$ is as in Notation~\ref{notation} and $\xi$ is as
in (\ref{bcxi}).
\end{lemma}

We first show how to obtain Theorem~\ref{lem-dolore_infinito} from the
estimate in Lemma~\ref{lem-1}. The proof of Lemma~\ref{lem-1} is given
in Section~\ref{barbecue}.

\subsection{From Lemma~\texorpdfstring{\protect\ref{lem-1}}{7.1} to Theorem~\texorpdfstring{\protect\ref{lem-dolore_infinito}}{6.12}}\label{accoppiamento}

%
\begin{quatation}
Let us first give a rough sketch of the coupling argument to be used.
By conditioning on the value $\gamma$ of the contour $\gamma_{H'+1}$
one can roughly replace the measure $\pi_{\Lambda^i}^{\xi,f}$ appearing
in\vspace*{-2pt} Theorem~\ref{lem-dolore_infinito} by the
measure~$\Pi_{\Lambda (\gamma)}^{H',f}$, where
$\Lambda(\gamma):=\Lambda\setminus\operatorname{Int}(\gamma)$ is the
region to the left of $\gamma$. Strictly speaking this is not true but
we shall reduce to a similar situation by way of monotonicity
arguments. Also, thanks to the argument of Lemma~\ref{lem-no-ceil}, one
can neglect the influence of the ceiling constraint. Thus, one
essentially wants to\vspace*{-2pt} couple
$\Pi_{\Lambda(\gamma)}^{H',f}$ and $\Pi_{\Lambda}^{H',f}$ on the region
$\Lambda^{i-1}=R'$. Thanks to Lemma~\ref{lem-1}, one can assume that
the rectangle $A_1$ is contained in $\Lambda(\gamma)$. From the Markov
property, it is sufficient to couple $\Pi_{\Lambda(\gamma)}^{H',f}$ and
$\Pi_{\Lambda}^{H',f}$ on the interface separating\vspace*{1pt} the
rectangles $A_1$ and $A_2$, see Figure~\ref{BBfig}. Thus, the desired
estimate would follow if one could exhibit a coupling such that with
large probability there exist chains $\mathcal{C}_1,\mathcal{C}_2$ of
sites in the rectangles $A_1$, $A_2$, respectively, where both
configurations are at constant height $H'$; see Figure~\ref{BBfig}. If
there were no external fields and no wall constraint, this would be a
simple consequence of Lemma~\ref{lem-peierls_c} (recall that, for
$\beta$ large, the interface is rigid and there is a density close to
$1$ of sites where the height equals the boundary height). However, due
to the presence of the external fields and the floor at zero,
establishing this fact requires extra work. The idea here is to reduce
the effective size of the system by imposing boundary conditions ${H'}$
on vertical crossings in the rectangle $A_3$, and in the rectangle
$A_0$. More precisely, let $\rho,\rho_1$ denote two
vertical crossings in $A_3$, and let $\rho_2$ denote a vertical
crossing in $A_0$. Using monotonicity and the estimate of
Lemma~\ref{quasi-rect.1} of Appendix~\ref{poschains}, we replace $
\Pi_{\Lambda(\gamma)}^{H',f}$ by $\Pi_{\Lambda(\rho,\gamma)}^{H',f}$
and $ \Pi_{\Lambda}^{H',f}$ by $\Pi_{\Lambda(\rho_1,\rho_2)}^{H',f}$,
where $\Lambda(\rho,\gamma)$ is the\vspace*{2pt} region between the chains $\rho$
and $\gamma$, while $\Lambda(\rho_1,\rho_2)$ denotes the region between
the chain $\rho_1$ and the chain $\rho_2$. Once this reduction has been
achieved, the system is contained in the union of the four rectangles
$\bigcup_{i=0}^4A_i$, a~$L\times4\ell$ rectangle, and one can easily
show that since $\ell$ is much smaller than $L$, and since $H'= H+1$,
the external field and the wall constraint can be neglected; see the
proof of Lemma~\ref{lem-V} below. At this point, one can use
Lemma~\ref{lem-peierls_c} to obtain the existence of chains
$\mathcal{C}_1,\mathcal{C}_2$ with the properties mentioned above.
\end{quatation}

We turn to the details of the proof. It is sufficient to couple
$\pi_{\Lambda^i}^{\xi,f}$ and $\pi_\Lambda^{H',f}$ on $\partial
R'$, the set of
points in $R'$ with a nearest neighbor in $R\setminus R'$, that is,
\[
\bigl\|\pi_{\Lambda^i}^{\xi,f}-\pi_\Lambda^{H',f}
\bigr\|_{R'}= \bigl\|\pi_{\Lambda^i}^{\xi,f}-\pi_\Lambda^{{H',f}}
\bigr\|_{\partial R'}.
\]
Note that,\vspace*{-2pt} because $\pi_{\Lambda^i}^{\xi,f}$ has maximal b.c. $n^+$ on
a side of $R$ and b.c. coinciding with that of $\pi_\Lambda^{{H',f}}$
on the other sides, $\pi_{\Lambda^i}^{\xi,f}$ stochastically dominates
$\pi_\Lambda^{{H',f}}$ on $\partial R'$ and therefore, by a union
bound, one has
%
%
\begin{equation}
\label{ubxn} \bigl\|\pi_{\Lambda^i}^{\xi,f}-\pi_\Lambda^{{H',f}}
\bigr\|_{R'} \leq\sum_{x\in\partial R'}\sum
_{v=0}^{n^+-1} \bigl[ \pi_{\Lambda^i}^{\xi,f}
(U_{x,v})-\pi_\Lambda^{{H',f}}(U_{x,v})
\bigr],
\end{equation}
where we define the events $U_{x,v}:=\{\eta_x>v\}$. Next, we remove the
ceiling constraint from the measures $\pi_{\Lambda^i}^{\xi,f},
\pi_\Lambda^{{H',f}}$. Since
$U_{x,v}$ are monotone events, we can 
estimate $\pi_{\Lambda^i}^{\xi,f} (U_{x,v})\leq\Pi_{\Lambda
^i}^{\xi,f}
(U_{x,v})$. Moreover, as in Section~\ref{sec-genn+}, one has
$\pi_\Lambda^{{H',f}}(U_{x,v}) = \Pi_\Lambda^{{H',f}}(U_{x,v}) +
O(L^{-C})$ where $C$ is as large as we wish provided $\beta$ is
sufficiently large.

Let\footnote{In the sequel of the proof, we introduce local notation
for various\vspace*{-1pt} conditional marginals of the measures
$\Pi_{\Lambda^i}^{\xi,f}$, $\Pi_{\Lambda}^{H',f}$ in order to keep
formulas readable.} $\tilde\mu{}^\gamma$ denote the marginal on $R''$
of $\Pi_{\Lambda^i}^{\xi,f}$ conditioned to have
$\gamma_{{H'}+1}(\eta)=\gamma$. Let $\operatorname{Int}(\gamma)$ denote
all sites enclosed by the contour $\gamma$ and the right boundary of
$R$, cf. Figure~\ref{BBfig}. Let $\Delta^-_\gamma$ denote the set of
sites $x\in R\setminus\operatorname{Int}(\gamma)$ that
have\vspace*{-1pt} either a nearest neighbor in
$\operatorname{Int}(\gamma)$, or a site at distance $\sqrt2$ in
$\operatorname{Int}(\gamma)$ in either the south--west or north--east
direction. Since conditioning on $ \gamma_{{H'}+1}(\eta)=\gamma$ forces
all sites in $\Delta^-_\gamma $ to be at height $\eta_x\leq{H'}$
(recall Definition~\ref{contourdef} of an $h$-contour), by monotonicity
one has $\tilde\mu{}^\gamma(U_{x,v})\leq\mu^\gamma(U_{x,v})$ if $\mu
^\gamma$ denotes the marginal on $R''$ of $\Pi_{\Lambda^i}^{\xi,f}$
conditioned to have height exactly
${H'}$ on all sites $x\in\Delta^-_\gamma$. 
Writing $\mathcal{B}$ for the event that
$\gamma_{{H'}+1}(\eta)$ does not intersect the rectangle $R''$, one
has, uniformly in $x,v$:
%
%
\begin{eqnarray}\label{couple}
&& \pi_{\Lambda^i}^{\xi,f} (U_{x,v})-
\pi_\Lambda^{{H',f}}(U_{x,v})
\nonumber\\[-8pt]\\[-8pt]
&&\qquad \leq L^{-C}+
\Pi_{\Lambda^i}^{\xi,f} \bigl(\mathcal{B}^c \bigr)+ \max
_{\gamma
\in\mathcal{B}
}{\mu^\gamma}(U_{x,v})-
\Pi_\Lambda^{H',f}(U_{x,v}).\nonumber
\end{eqnarray}
Lemma~\ref{lem-1} says that $ \Pi_{\Lambda^i}^{\xi,f}(\mathcal{B}
^c)=O(L^{-C})$, so that we are left with the upper bound on
${\mu^\gamma}(U_{x,v})-\Pi_\Lambda^{H',f}(U_{x,v})$ for $\gamma\in
\mathcal{B} $. We now implement the system reduction mentioned in the
sketch of the proof above.

Let $A_i$, $i=0,1,2,3$, denote the $L\times\ell$ rectangles in $R$
depicted in Figure~\ref{BBfig}. Write $A_i^t$, for the external top
boundary of $A_i$, that is, the set of sites $x\notin R$ such that $x$
has a nearest neighbor on the top side of the rectangle $A_i$.
Similarly, write $A_i^b$ for the external bottom boundary of $A_i$.
Call $\mathcal{E}(A_i)$ the set of $\mathbb{Z}^2$-bonds $e$ such that
$e$ has at least one endpoint in $A_i$ and at most one endpoint in
$A_i^t\cup A_i^b$. A~vertical crossing in $A_i$ is a connected set
$\mathcal{C}\subset\mathcal{E}(A_i)$ that connects $A_i^t$ and $A_i^b$;
see Figure~\ref{BBfig}.

Let $\mathcal{F}_-$ (resp., $\mathcal{F}_+$) denote the event that
there exists a vertical crossing $\mathcal{C}$ in $A_3$ such that
$\eta_x\leq{H'}$ for all $x\in\mathcal{C}$ (resp., a crossing
$\mathcal{C}$ in $A_3$ and a crossing $\mathcal{C}'$ in $A_0$ such that
$\eta_x\geq{H'}$, $x\in\mathcal{C}\cup\mathcal{C}'$). On the event
$\mathcal{F}_-$ one may consider the leftmost vertical crossing in
$A_3$ with the required property, where leftmost is defined according
to lexicographic order. From the Markov property of $\mu^\gamma$, and
using monotonicity,
\[
\mu^\gamma(U_{x,v})\leq\mu^\gamma \bigl(
\mathcal{F}_-^c \bigr) + \max_{\rho
}\mu
^{\gamma,\rho}(U_{x,v}),
\]
where $\mu^{\gamma,\rho}$ stands for the measure $\mu^\gamma$
conditioned to have
height ${H'}$ on $\rho$, and $\rho$ ranges over all possible vertical
crossings in $A_3$.
Similarly, on the event $\mathcal{F}_+$ denote $\rho_1$ (resp., $\rho
_2$) the rightmost (resp., leftmost) crossing in $A_0$ (resp., $A_3$)
and write
%
%
\begin{eqnarray}
\label{eq-35} \Pi_\Lambda^{H',f}(U_{x,v})&\geq&
\bigl(1- \Pi_\Lambda^{H',f} \bigl(\mathcal{F}_+^c
\bigr) \bigr) \min_{\rho
_1,\rho_2}\mathcal{Q}^{\rho_1,\rho_2}(U_{x,v})
\nonumber\\[-8pt]\\[-8pt]
&\geq&\min_{\rho _1,\rho _2}\mathcal{Q}^{\rho _1,\rho_2}(U_{x,v})-
\Pi_\Lambda^{H',f} \bigl(\mathcal{F}_+^c \bigr),\nonumber
\end{eqnarray}
where $\mathcal{Q}^{\rho_1,\rho_2}$ stands for the measure $\Pi
_\Lambda^{H',f}$
conditioned to have height ${H'}$ on $\rho_1,\rho_2$, and $\rho
_1,\rho_2$ range
over all possible vertical crossings in $A_0,A_3$. Altogether,
%
%
\begin{eqnarray}
\label{mugnu}
\qquad && \mu^\gamma(U_{x,v})- \Pi_\Lambda^{H',f}
(U_{x,v})
\nonumber\\[-8pt]\\[-8pt]
&&\qquad \leq\mu^\gamma \bigl(\mathcal{F}_-^c
\bigr) + \Pi_\Lambda^{H',f} \bigl(\mathcal{F}_-^c
\bigr) + \max_{\rho,\rho_1,\rho_2} \bigl\llvert\mu^{\gamma,\rho
}(U_{x,v})-\mathcal{Q} ^{\rho
_1,\rho_2}(U_{x,v}) \bigr\rrvert.\nonumber
\end{eqnarray}
It follows\vspace*{-1pt} from Lemma~\ref{quasi-rect.1} of Appendix~\ref{poschains}
that $\mu^\gamma(\mathcal{F}_-^c)$ and $\Pi_\Lambda
^{H',f}(\mathcal{F}_+^c)$ are
$O(e^{-L^{1-\varepsilon}})$.
Notice that $ \mu^{\gamma,\rho}$ (resp., $\mathcal{Q}^{\rho_1,\rho_2}$)
are SOS measures with exactly ${H'}$ b.c. around the domain whose
boundary is determined by $\rho$ (resp., $\rho_1$) on the left and by
$\Delta^-_\gamma$ (resp., $\rho_2$) on the right. Such domain has (by
construction) horizontal size of order $\ell$ and vertical size $L$. To
simplify the notation, we shall write $\mu^\gamma,\mathcal{Q}$ for $
\mu^{\gamma,\rho}, \mathcal{Q}^{\rho_1,\rho_2}$.

We now turn our attention to vertical crossings in the rectangles
$A_1,A_2$. Consider the independent coupling $\mathbb{P}$ of $\mu
^\gamma,\mathcal{Q}$ on $A_1\cup A_2$. Writing $(\eta,\eta')$ for the
corresponding random variables, let $\mathcal{A}_i$, $i=1,2$ denote the
event that there exists a vertical crossing $\mathcal{C}$ in $A_i$ such
that $\nabla_e\eta=\nabla_e\eta'=0$ for all bonds $e$ with both
endpoints in $\mathcal{C}$. Note that if $\mathcal{C}$ is a vertical
crossing in $A_i$ as above, then $\eta_\mathcal{C}=\eta
'_\mathcal{C}={H'}$, because of the boundary conditions equal to ${H'}$
on the top and bottom boundary of $A_i$, $i=1,2$. On the event
$\mathcal{A}_1\cap\mathcal{A}_2$, one may consider the leftmost
vertical crossing $\mathcal{C}_2$ in $A_2$ and the rightmost vertical
crossing $\mathcal{C}_1$ in $A_1$. From the Markov property of the
Gibbs measures $\mu^\gamma,\mathcal{Q}$ and the fact that
$\mathcal{Q}(\cdot|\eta_{\mathcal{C}_1}=\eta_{\mathcal{C}_2}={H'})$
and $\mu^\gamma(\cdot|\eta_{\mathcal{C}_1}=\eta_{\mathcal{C}_2}={H'})$
have the same marginal on $\partial R'$ (observe that $\partial R'$ is
just at the boundary between $A_1$ and $A_2$), one obtains that
%
%
\begin{equation}
\label{couple1} \bigl\llvert{\mu^\gamma}(U_{x,v})-
\mathcal{Q}(U_{x,v}) \bigr\rrvert
\leq\mathbb{P} \bigl(
\mathcal{A}_1^c \bigr)+\mathbb{P} \bigl(
\mathcal{A}_2^c \bigr).
\end{equation}

We shall focus on the event $\mathcal{A}_1^c$, since the event
$\mathcal{A}_2^c$ can be treated in the same way. To estimate
$\mathbb{P}(\mathcal{A}_1^c)$, we use the fact (see, e.g.,
\cite{Grimmett}, Lemma 11.21) that nonexistence of a vertical crossing
in $A_1$ implies the existence of a horizontal dual crossing in $A_1$.
More precisely, let $A_1^r$ denote the r.h.s. of $A_1$, that
is, the set of dual bonds $e'$ such that $e'$ crosses an edge of the
form $e=(x,y)$ with $x\in A_1$ and $y\in R\setminus R''$. Similarly,
let $A_1^\ell$ denote the l.h.s. of $A_1$. We say that a dual
bond $e'$ is in $A_1$ if $e'$ crosses a bond $e\in\mathcal{E}(A_1)$.
Then, the event $\mathcal{A}_1^c$ implies that there exists a connected
set $\mathcal{D}$ of dual bonds $e'$ in $A_1$ which connects the lines
$A_1^r$~and~$A_1^\ell$, and such that for every $e'\in\mathcal{D}$
either $\nabla_{e'}\eta\neq0$ or $\nabla_{e'}\eta'\neq0$. Here we use
the notation $\nabla_{e'}\eta:=\nabla_e\eta$ if $e'$ is the dual bond
that crosses $e$. Moreover, for a given $\mathcal{D}$ as above, there
must be a set $V\subset\mathcal{D}$ such that $|V|\geq|\mathcal{D}|/2$
and such that either $E_V:=\{\nabla_{e'}\eta\neq0$ for all $ e'\in V\}$
or $ F_V:=\{\nabla_{e'}\eta'\neq0$ for all $e'\in V\}$. Thus, using
a~union bound, one obtains
%
%
\begin{equation}
\label{V01} \mathbb{P} \bigl(\mathcal{A}_1^c \bigr)\leq
\sum_{\mathcal{D}} \mathop{\sum
_{V\subset
\mathcal{D}\dvtx}}_{|V|\geq|\mathcal{D}|/2} \bigl( \mu^\gamma(E_V)+
\mathcal{Q}(F_V) \bigr),
\end{equation}
where the first sum is over all connected sets of dual bonds
$\mathcal{D}$ connecting $A_1^r$~and~$A_1^\ell$ as above. We will need
the following lemma.

%
%
\begin{lemma}
\label{lem-V} There exist constants $C,c,\beta_0>0$ independent of
$\beta$ such that, for every set $V$ of dual bonds in $A$ with
$|V|\geq
\ell/2$, one has for
all $\beta\geq\beta_0$ 
%
%
\begin{equation}
\label{V1} \max \bigl\{\mu^\gamma(E_V),
\mathcal{Q}(F_V) \bigr\}\leq Ce^{-c \beta|V|}.
\end{equation}
\end{lemma}
Let us conclude the proof of Theorem~\ref{lem-dolore_infinito} assuming
for a moment the validity of Lemma~\ref{lem-V}. From (\ref{V01}),
summing over the possible (connected) sets $\mathcal{D}$, and using
$|\mathcal{D} |\geq\ell\geq L^{1-\varepsilon}$, for all
$\varepsilon>0$, if $\beta\geq\beta_0$:
%
%
\begin{eqnarray}
\label{VV} \mathbb{P} \bigl(\mathcal{A}_1^c \bigr)&
\leq&2C \sum_{k\geq\ell}\sum_{\mathcal
{D}\dvtx |\mathcal{D}|=k}\mathop{\sum_{V\subset\mathcal{D}\dvtx}}_{|V|\geq k/2}e^{-c\beta |V|}\nonumber
\\
&\leq&2C\sum_{k\geq\ell}\sum _{\mathcal{D}\dvtx |\mathcal{D}|=k}2^ke^{-c\beta k/2}\nonumber
\\
&\leq& 2C\sum _{k\geq\ell}6^ke^{-c\beta k/2}
\\
&\leq& C' e^{-c\beta\ell/4}\nonumber
\\
&=& O \bigl(\exp \bigl(-L^{1-\varepsilon}
\bigr) \bigr).
\nonumber
\end{eqnarray}
Since the constants implied in (\ref{VV}) are uniform in $x,v$ and the
choice of \mbox{$\gamma\in\mathcal{B}$}, the claim of
Theorem~\ref{lem-dolore_infinito} follows from (\ref{couple}) and
(\ref{ubxn}). It remains to prove Lemma~\ref{lem-V}. This is where the
reduction from $\mu^\gamma$ to $\mu^{\gamma,\rho}$ and
$\Pi_\Lambda^{H',f}$ to $\mathcal{Q}^{\rho_1,\rho_2}$ becomes
important.

\begin{pf*}{Proof of Lemma~\protect\ref{lem-V}}
We shall prove the bound concerning $\mu^\gamma=\mu^{\gamma,\rho}$
only, since the same proof works for
$\mathcal{Q}=\mathcal{Q}^{\rho_1,\rho_2}$. Consider the region
$\Lambda_0\subset R$ delimited on the left by $\rho$ and on the right
by $\gamma$. Since $\rho$ is a vertical crossing in~$A_3$, one has
$A_1\subset \Lambda_0$. A crucial fact is that $|\Lambda_0|\leq4L\ell$.
Let as usual $\hat\pi{}^{H'}_{\Lambda_0}$ denote the SOS measure on
$\Lambda_0$ with boundary condition ${H'}$ outside of
$\Lambda_0$, 
with no floor, no ceiling and no external fields. From
Lemma~\ref{lem-peierls_c} one has $\hat\pi{}^{H'}_{\Lambda _0}(E_V)\leq
e^{-\beta|V|/2}$ for any $V$. Thus, it suffices to show that
%
%
\begin{equation}
\label{V11} \mu^\gamma(E_V)\leq Ce^{C\ell} \hat
\pi{}^{H'}_{\Lambda_0}(E_V)
\end{equation}
for some constant $C$ independent of $\beta$. Note that the external
fields contribute with the term $0\leq\frac1L\sum_{x\in\Lambda_0}
f_{x}\leq C\ell$ to the Hamiltonian, and therefore, at the price of
a~factor $e^{C\ell}$ we can remove all external fields in our measure
$\mu^\gamma$. Then
%
%
\begin{equation}
\label{V12} \mu^\gamma(E_V)\leq e^{C\ell}
\frac{\hat\pi{}^{H'}_{\Lambda_0}(E_V)}{\hat\pi{}^{H'}_{\Lambda
_0}(\eta_x \geq 0\ \forall x\in\Lambda_0)}.
\end{equation}
Next, from the FKG inequality, one has
%
%
\begin{equation}
\label{V13} \qquad\hat\pi{}^{H'}_{\Lambda_0}(\eta_x \geq0
\ \forall x\in \Lambda_0)\geq\prod_{x\in\Lambda_0}
\hat \pi{}^{H'}_{\Lambda_0}(\eta_x \geq0) \geq\prod
_{x\in\Lambda_0} \bigl(1-C e^{-4\beta{H'}} \bigr),
\end{equation}
where we use the equilibrium estimate
$\hat\pi{}^{H'}_{\Lambda_0}(\eta_x < 0) =
\hat\pi{}^0_{\Lambda_0}(\eta_x
> {H'}) \leq\break  C e^{-4\beta{H'}}$; see Proposition~\ref{p-heightBounds}.
Since $e^{-4\beta {H'}}=e^{-8\beta }/L$, one has
%
%
\begin{equation}
\label{V130} \prod_{x\in\Lambda_0} \bigl(1-C
e^{-4\beta{H'}} \bigr)\geq C_1^{-1}e^{-C_1\ell}
\end{equation}
for a suitable constant $C_1>0$. The desired conclusion follows from
(\ref{V12}). This ends the proof.
\end{pf*}

\subsection{Proof of Lemma~\texorpdfstring{\protect\ref{lem-1}}{7.1}}\label{barbecue}

%
\begin{quatation}
Roughly speaking, the proof of Lemma~\ref{lem-1} works as follows.
There are $n^+-{H'}$ open contours attached to the r.h.s. of
$R$ (call it $r$) and let $\gamma_j$, $j\in\{{H'}+1,\ldots,n^+\}$
denote the $j$-contour. First, one proves that the $n^+$-contour cannot
reach distance say $L/(\log L)^2$ from $r$. For this, one lifts from
${H'}$ to $n^+-1$ the b.c. around the three sides of $R$ different from
$r$ (this is allowed by monotonicity). This way, there is now \textit{a
single open contour} and the estimate follows from
Proposition~\ref{dkspro}. Next, we want to prove that $\gamma_{n^+-1}$
cannot reach distance $L/(\log L)^2$ from $\gamma_{n^+}$, that is,
distance $2L/(\log L)^2$ from $r$. Morally the proof works as for the
previous case, except that now the b.c. $n^+$ at $r$ is replaced by the
b.c. $n^+-1$ at $\gamma_{n^+}$. The argument is then repeated
iteratively and the statement of the lemma follows when $j={H'}+1$. In
practice, there are many additional difficulties, which is why the
proof is so much involved. The main obstacles are the following:
\begin{enumerate}[(3)]
\item[(1)] Proposition~\ref{dkspro} cannot be applied directly,
    because it holds when both the floor constraint $\eta\geq0$ and
    the field $f$ are absent. However, Proposition~\ref{clam} will
    show that (morally) the field compensates the effect of the
    floor (which would tend to push the contours away from $r$).

\item[(2)] Once $\gamma_{j}$ is fixed, it is not true that the next
contour (i.e., $\gamma_{j-1}$) sees boundary conditions $j-1$
in a new domain determined by $\gamma_{j}$. The point is that,
from definition of contours, we only know that the heights just
to its left are \textit{at most} $j-1$, not \textit{exactly}
$j-1$. We will use monotonicity to be able to change to
$j-1$-b.c.

\item[(3)] Applying Proposition~\ref{dkspro} as outlined above to
estimate the probability of large deviations of $\gamma_{j-1}$
given $\gamma_j$ requires that the right boundary of the system
(i.e., the configuration of $\gamma_j$), where b.c. are $j-1$,
is not too wild. In practice, one needs it to be a path
connecting top and bottom of the rectangle $R$, with
transversal fluctuations at most of order say $L^{\varepsilon}$
for some small $\varepsilon$. We will apply the results of
Appendix~\ref{poschains} to infer that, indeed, to the left of
$\gamma_j$ and not far away from it there is a chain of sites,
with transversal fluctuations of the required order, where
heights are exactly $j-1$.
\end{enumerate}
\end{quatation}

We use a sort of induction on the index of the open contours
$\gamma_j$, $j\in\{{H'}+1,\ldots,n^+\}$, where $n^+=\log L$. Let $A_0$
denote the rightmost $L\times\ell$ rectangle inside $R$ as in
Figure~\ref{BBfig}, and write $A_0=\bigcup_{j={H'}+1}^{n^+}B_j$ where
$B_j$ are\vspace*{-1pt} nonoverlapping\vadjust{\goodbreak} $L\times\ell_0$ rectangles,
ordered from left to right, such that $\ell_0=\ell/(n^+-{H'})\approx
L/(4(\log L)^2)$. Every rectangle $B_j$ is further divided into two
nonoverlapping rectangles $B_j^1, B_j^2$, ordered from left to right,
such that $B_j^1$ is a $L\times\ell_1$ rectangle with
$\ell_1=L^{\delta}$, for some (arbitrarily) small $\delta>0$, and
$B_j^2=B_j\setminus B_j^1$ is a $L\times\ell_2$ rectangle, with
$\ell_2=\ell_0-\ell_1\sim\ell _0$; see Figure~\ref{bjs}.
%
%
\begin{figure}

\includegraphics{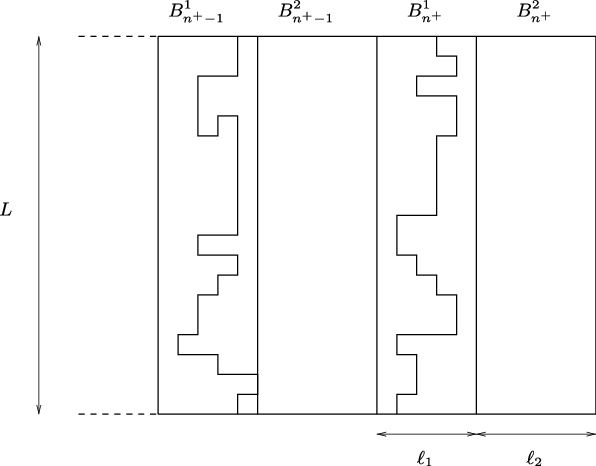}

\caption{The rectangles $B_j^1$, $B_j^2$, with the associated vertical
crossings $\rho_j$, for $j=n^+-1$ and $j= n^+$.}\label{bjs}
\end{figure}

Define vertical crossings in a rectangle as in
Section~\ref{accoppiamento}. For $j\in\{{H'}+1,\ldots,n^+\}$, consider
the event $\mathcal{B}_j$ that there exists a vertical crossing
$\mathcal{C}_j$ in $B_j^1$, such that $\eta_x\leq j-1$ for all
$x\in\mathcal{C}_j$. In particular, on $\mathcal{B}_{{H'}+1}$, there
exists a vertical crossing $\mathcal{C}_{{H'}+1}$ in $A_0$ with
$\eta_x\leq{H'}$ for all $x\in\mathcal{C}_{{H'}+1}$. Thus
$\mathcal{B}\supset\mathcal{B}_{{H'}+1}$, and it will be sufficient to
estimate from above the probability
$\Pi_{\Lambda^i}^{\xi,f}(\mathcal{B}_{{H'}+1}^c)$. Clearly,
\[
\Pi_{\Lambda^i}^{\xi,f} \bigl(\mathcal{B}_{{H'}+1}^c
\bigr)\leq\sum_{j={H'}+1}^{n^+}\Pi
_{\Lambda^i}^{\xi,f} \bigl(\mathcal{B}_j^c
\cap\mathcal{B}_{j+1} \bigr),
\]
where $\mathcal{B}_{n^++1}$ denotes the whole probability space. On the
event $\mathcal{B}_{j+1}$, let $\mathcal{C}_{j+1}$ denote the rightmost
vertical crossing $\mathcal{C}$ in $B_{j+1}^1$ such that $\eta_x\leq
j$, $x\in\mathcal {C}$. By conditioning on the event
$\{\mathcal{C}_{j+1}=\rho_{j+1}\}$, and using that the events
$\mathcal{B}_j^c$ are increasing, one has
%
%
\begin{equation}
\label{bjs2} \Pi_{\Lambda^i}^{\xi,f} \bigl(\mathcal{B}_{{H'}+1}^c
\bigr)\leq\sum_{j={H'}+1}^{n^+}\max
_{\rho_{j+1}}\mu_{\rho_{j+1}} \bigl(\mathcal{B}_j^c
\bigr),
\end{equation}
where $\rho_{j+1}$ ranges over all possible vertical crossings in
$B_{j+1}^1$ (for $j=n^+$, it is understood that $\rho_{j+1}$ coincides
with the right boundary of $R$), and $\mu_{\rho_{j+1}}$ stands for the
SOS Gibbs measure on the region $\Lambda_{(j)}\subset R$ defined as the
set of sites $x\in R$ to the left of the crossing $\rho_{j+1}$, with
\begin{itemize}
\item boundary condition $\eta_x=j$ for $x\in\rho_{j+1}$ and
$\eta_x=j-1$ on all other boundary sites. Note that a portion
of the boundary height has been lifted from ${H'}$ to
$j-1\geq{H'}$. The advantage is that, this way, there is a
unique open contour under the measure $\mu_{\rho_{j+1}}$,
rather than $j-{H'}$ of them;

\item floor constraint $\eta_x\geq0$;

\item external field
%
%
\begin{equation}
\label{hfact} \frac1L\sum_{x\in\Lambda_{(j)}} f_{x,j-1-H},
\end{equation}
where we recall that $f_{x,j}=\exp(-\beta j){\mathbf1}_{\eta_x\leq
H+j}$, cf.~(\ref{eq-10}). Note that the fields in~(\ref{eq-10})
with index different from $j-1-H$ have been removed. This is
allowed since the function $\eta\mapsto f_{x,a}(\eta)$ is
decreasing. The effect of the field $f_{x,a}$ is to depress the
area of the $(a+H+1)$-open contour and, since there is just one
open contour, we need only the term with $a=j-1-H$.
\end{itemize}
%

Lemma~\ref{lem-1} is then a consequence of (\ref{bjs2}) and the
following claim.

%
\begin{claim}
For $j\geq{H'}+1$, uniformly in the vertical crossing $\rho_{j+1}$ in
$B_{j+1}^1$ and for every $C>0$,
%
%
\begin{equation}
\label{Blemma01} \mu_{\rho_{j+1}} \bigl(\mathcal{B}_j^c
\bigr)=O \bigl( L^{-C} \bigr).
\end{equation}
\end{claim}
Let $\gamma_j$ denote the unique open $j$-contour for a configuration
$\eta$ in the ensemble $\mu_{\rho_{j+1}}$.
By construction, $\gamma_j$ is to the left of $B_{j+1}^2$, and it may
intersect the rectangle $B_j^{2}$ or even $B_j^{1}$. Let $E_j$ denote
the event that $\gamma_j$ intersects $B_{j}^1$.
Conditionally on the event $E_j^c$, the contour stays to the right of
$B_j^1$, and the estimate $\mu_{\rho_{j+1}}(\mathcal{B}_j^c|E_j^c)=O(
\exp(-L^{\delta}))$ follows from (\ref{eq-27}), which is applicable
since the shorter side of $R$ is at most of length $L$. Thus the claim
(and hence Lemma~\ref{lem-1}) follows once we prove the following
lemma.
%
%
\begin{lemma}\label{Blemma}
For $j\geq{H'}+1$, uniformly in the vertical crossing $\rho_{j+1}$ in
$B_{j+1}^1$, and for all $C>0$:
%
%
\begin{equation}
\label{Blemma1} \mu_{\rho_{j+1}}(E_j)=O \bigl(L^{-C}
\bigr).
\end{equation}
\end{lemma}
We will actually give an upper bound of order $\exp(-L^{1-\varepsilon
})$ for
every $\varepsilon>0$.

\begin{pf*}{Proof of Lemma~\ref{Blemma}}
For this proof,\vspace*{1pt} the crossing $\rho_{j+1}$ in $B_{j+1}^1$ is fixed, and
we simply write $\mu$ instead of $\mu_{\rho_{j+1}}$. Fix a contour
$\Gamma$ and consider the event $\gamma_j=\Gamma$. Set $\Lambda
_+=\operatorname{Int}(\Gamma)\cap\Lambda_{(j)}$, and
$\Lambda_-=\Lambda_{(j)}\setminus\Lambda_+$, so that $\Gamma$ is the
set of dual bonds separating $\Lambda_-$ and $\Lambda_+$ within
$\Lambda_{(j)}$ [with $\operatorname{Int}(\Gamma)$ defined a few lines
after~(\ref{ubxn})]. For any $\Gamma$, one may write
%
%
\begin{equation}
\label{Blemma2} \mu(\gamma_j=\Gamma)\propto
e^{-\beta|\Gamma|}Z_{j,\Lambda_-}Z_{j,\Lambda_+}.
\end{equation}
Here, 
$Z_{j,\Lambda_-}$ (resp., $Z_{j,\Lambda_+}$) is the partition function
of the SOS model on $\Lambda_-$ (resp.,~$\Lambda_+$), with floor at
height $0$, field as in (\ref{hfact}), b.c. $j-1$ on
$\partial\Lambda_-$ (resp., b.c. $j$ on $\partial\Lambda_+$) and with
the extra constraint that $\eta_x\leq j-1$ for all
$x\in\Delta^-_\Gamma$ (resp., $\eta_x\geq j$ for all $x\in
\Delta^+_\Gamma$), where $\Delta^-_\Gamma$ (resp., $\Delta ^+_\Gamma$)
is the set of $x\in\Lambda_-$ either at distance 1 from $\Lambda_+$
(resp., $\Lambda_-$) or at distance $\sqrt2$ from a vertex
$y\in\Lambda_+$ (resp., $y\in\Lambda_-$) in the south west or north
east direction. These constraints are imposed by the definition of
$j$-contour; see Definition~\ref{contourdef}.

Next, let $ \mathcal{Z}_{\Lambda_-}^0$ (resp., $ \mathcal
{Z}_{\Lambda
_+}^0$) denote the partition function of the SOS model on $\Lambda_-$
(resp., $\Lambda_+$) with b.c. $0$, no floor and no external fields,
with the constraint that $\eta_x\leq0$ for all $x\in\Delta^-_\Gamma$
(resp., $\eta_x\geq0$, $x\in\Delta^+_\Gamma$). Let $\omega_-$,
$\omega_+$ be the corresponding Gibbs measures.
With these definitions, one rewrites (\ref{Blemma2}) as
%
%
\begin{eqnarray}
\label{Blemma3} \mu(\gamma_j=\Gamma)& =&\frac1{\mathcal{Z}}
e^{-\beta|\Gamma|} \mathcal{Z}_{\Lambda_-}^0\mathcal{Z}_{\Lambda_+}^0
\nonumber
\\
&&{}\times\omega_-
\bigl(e^{(\mathcal{K}/L)\sum_{x\in\Lambda_-}{\mathbf 1}_{\eta
_x\leq 0}}; \eta\geq-(j-1) \bigr)
\\
&&{}\times{} \omega_+ 
\bigl(e^{(\mathcal{K}/L)\sum_{x\in\Lambda_+}{\mathbf1}_{\eta
_x\leq -1}}; \eta\geq- j \bigr),
\nonumber
\end{eqnarray}
where $\mathcal{Z}$ is the normalization and
${\mathcal{K}}=c_{j-1-H}=\exp(-\beta(j-1-H))$, see
(\ref{eq-10}).\vadjust{\goodbreak}

We first observe that
the very same arguments of 
Proposition~\ref{p-heightBounds}
proves that 
for all $x\in\Lambda_\pm$,
%
%
\begin{equation}
\label{ekub} \omega_\pm(\eta_x\geq j)\asymp
e^{-4\beta j},
\end{equation}
that is, $C^{-1}e^{-4\beta j} \leq\omega_\pm(\eta_x\geq j)\leq C
e^{-4\beta j}$ for some absolute constant $C>0$. This is possible
thanks to the fact that even in the presence of the constraints on
$\Delta^\pm_\Gamma$ the arguments of Lemma~\ref{l:contourBound} can be
used without modifications. Next, let $\hat\pi$ stand for the infinite
volume limit of the SOS measure with zero boundary condition.
Proposition~\ref{p-heightBounds} implies that $\hat\pi(\eta_0\geq
j)\asymp e^{-4\beta j}$. Moreover, we observe that there exist
constants $c,t_0>0$ such that for any $x\in\Lambda_\pm$ at distance at
least $t>t_0$ from the boundary $\partial\Lambda_\pm$, for any $k$:
%
%
\begin{equation}
\label{ekubdec} \bigl\llvert\omega_\pm(\eta_x\geq k) -
\hat\pi(\eta_0\geq k) \bigr\rrvert\leq e^{-c t}.
\end{equation}
Let us prove (\ref{ekubdec}) in the case $x\in\Lambda_-$. The case
$x\in\Lambda_+$ is obtained with the same argument. Thanks to the
exponential decay of correlations for the $0$-b.c. SOS model at large
$\beta$ (see \cite{BW}), (\ref{ekubdec}) is equivalent to the statement
obtained by replacing $\hat\pi$ by $\hat\pi{}_{\Lambda_-}^0$.
Observe that by monotonicity $\omega_-(\eta_x\geq k)\leq
\hat\pi{}_{\Lambda_-}^0(\eta_x\geq k)$. Next, by the same argument of
Lemma~\ref{l:contourBound}, the $\omega_-$-probability of a contour
$\gamma$, is bounded above by $e^{-\beta|\gamma|}$, and thus with
probability at most $e^{-c t}$ there is no chain $\mathcal{C}$ of
heights all greater or equal to zero in a shell of width $t/2$ around
$x$, and at distance larger than $t/2$ from $x$. On the other hand, if
$E$ is the event that such a chain exists then by monotonicity and
decay of correlations one has $ \omega_-(\eta_x\geq k;
E)\geq\hat\pi{}_{\Lambda_-}^0(\eta_x\geq k) +e^{-c t}$. This proves
(\ref{ekubdec}).

We turn to a rough estimate that allows one to rule out very long
contours. Namely, if $G$ denotes the event that $|\gamma_j|\leq
L^{1+\varepsilon_0}$, then for all $\beta$ large enough
%
%
\begin{equation}
\label{Blemmagep} \mu(G) = 1-O \bigl(e^{-L^{1+\varepsilon_0}} \bigr).
\end{equation}
In what follows, we may fix $\varepsilon_0>0$ as small as we wish. To prove
(\ref{Blemmagep}), observe that from a trivial bound on the external
fields and the FKG property for $\omega_\pm$
one has
%
%
\begin{equation}
\label{Blemma404}
\qquad\omega_- 
\bigl(e^{(\mathcal{K}/L)\sum_{x\in\Lambda_-}{\mathbf1}_{\eta
_x\leq0}}; \eta\geq- (j-1)
\bigr) \geq\prod_{x\in\Lambda_-} \omega_- 
\bigl(
\eta_x\geq-(j-1) \bigr).
\end{equation}
From (\ref{ekub})
\[
\omega_- \bigl(\eta_x\geq-(j-1) \bigr)\geq \bigl(1-ce^{-4\beta(j-1)}
\bigr)\geq\exp \bigl(-c'/L \bigr),
\]
since $j\geq H$. Then (\ref{Blemma404}) is bounded below by $e^{-CL}$
for some $C>0$. The same estimate holds for the last term in
(\ref{Blemma3}), and therefore one has
%
%
\begin{equation}
\label{Blemma4} \mu(\gamma_j=\Gamma)\leq\nu(\Gamma)
e^{ CL}
\end{equation}
for some constant $C>0$, where $\nu$ is the probability measure on
contours $\Gamma$ given~by
%
%
\begin{equation}
\label{dksnu} \nu(\Gamma)\propto e^{-\beta|\Gamma|} \mathcal {Z}_{\Lambda_-}^0
\mathcal{Z}_{\Lambda_+}^0.
\end{equation}
Notice that $\nu$ is the distribution of the unique open contour of the
SOS measure on $\Lambda_-\cup\Lambda_+$ with no floor constraint, with
Dobrushin boundary conditions, namely with b.c. $\eta_y=0$ or
$\eta_y=1$ depending on whether $y$ has a nearest neighbor in
$\Lambda_-$ or in $\Lambda_+$, respectively. It follows from
(\ref{expan1}) that $\nu(\Gamma)$ has the standard form
\[
\nu(\Gamma)\propto e^{-\beta|\Gamma|+\Psi(\Gamma)},
\]
where the decoration term $\Psi$ satisfies $|\Psi(\Gamma)|\leq
ce^{-\beta}|\Gamma|$. The usual Peierls' argument shows that
%
%
\begin{equation}
\label{ggep} \nu \bigl(|\gamma|\geq L^{1+\varepsilon_0} \bigr) = O
\bigl(e^{-L^{1+\varepsilon_0}} \bigr)
\end{equation}
and (\ref{Blemmagep}) follows.

Thanks to (\ref{Blemmagep}), we can now restrict the summation in the
normalization $\mathcal{Z}$ in (\ref{Blemma3}) to contours $\Gamma
\in G$. Define
%
%
\begin{eqnarray}
\Phi_- &:=& \frac{\mathcal{K}}L
|\Lambda_-| \hat\pi(
\eta_0> 0), 
\qquad\Psi_-:=
|\Lambda_-| \hat
\pi \bigl( \eta_0< -(j-1) \bigr),\label{deffi}
\\
\Phi_+&:=&\frac{\mathcal{K}}L 
|\Lambda_+| \hat\pi(\eta_0
\leq-1), 
\qquad q\Psi_+:=
|\Lambda_+| \hat\pi(
\eta_0< -j). \label{deffit}
\end{eqnarray}

%
%
\begin{proposition}\label{clam}
There exists $\alpha<1$ such that for all $\Gamma\in G$ one has the expansions
(with error terms uniform in $\Gamma\in G$):
%
%
\begin{eqnarray}
&& \omega_- \bigl(e^{(\mathcal{K}/L)\sum_{x\in
\Lambda _-}{\mathbf 1}_{\eta_x\leq 0}}; \eta\geq- (j-1) \bigr)
\nonumber\\[-8pt] \label{Blemma30} \\[-8pt]
&&\qquad = \exp \bigl(\mathcal{K}|\Lambda_-|/L -\Phi_- -\Psi_- + O \bigl(L^{\alpha}
\bigr) \bigr)\nonumber
\\
&& \omega_+ \bigl(e^{(\mathcal{K}/L)\sum_{x\in\Lambda_+}{\mathbf
1}_{\eta_x\leq-1}}; \eta\geq- j \bigr)
\nonumber\\[-8pt]\label{Blemma31} \\[-8pt]
&&\qquad = \exp \bigl( \Phi_+ - \Psi_+ + O \bigl(L^{\alpha} \bigr) \bigr).\nonumber
\end{eqnarray}
\end{proposition}

Let us conclude the proof of Lemma~\ref{Blemma} assuming for the moment
the validity of Proposition~\ref{clam}.
First, observe 
that the functions in (\ref{deffi}) and (\ref{deffit}) satisfy, for
some $\alpha<1$, uniformly in $\Gamma\in G$:
%
%
\begin{eqnarray}
\label{Blemma32} -\Phi_-+\Phi_+ &=&\frac{\mathcal{K}}L \bigl(|\Lambda _+|-|\Lambda_-|
\bigr)\hat\pi(\eta_0> 0), 
\\
\qquad\Psi_-+\Psi_+ &=& |\Lambda_-| \hat\pi(\eta_0\geq j )+|
\Lambda_+| \hat\pi(\eta_0\geq j+1). 
\label{Blemma323}
\end{eqnarray}

From~(\ref{Blemma30})--(\ref{Blemma323}), setting $|\Lambda_{(j)}|=
|\Lambda_-|+|\Lambda_+|$, $\delta_k(\beta)=\hat\pi(\eta_0\geq k)$ and
$\bar\delta_k(\beta)=L\delta_k(\beta)$:
\begin{eqnarray*}
&& \omega_- \bigl(e^{(\mathcal{K}/L)\sum_{x\in\Lambda
_-}{\mathbf 1}_{\eta _x\leq 0}}; \eta\geq- (j-1) \bigr) \omega_+
\bigl(e^{(\mathcal{K}/L)\sum_{x\in\Lambda_-}{\mathbf 1}_{\eta
_x\leq -1}}; \eta\geq- j \bigr)
\\
&&\qquad = \exp \biggl( \frac{|\Lambda_{(j)}|}L \bigl(\mathcal{K} \bigl(1-\delta
_1(\beta) \bigr)-\bar\delta_{j}(\beta) \bigr)
\\
&&\hspace*{54pt}{} +\frac{|\Lambda_+|}L \bigl(\mathcal{K} \bigl(-1+2\delta_1(\beta)
\bigr)+ \bar\delta_j(\beta)-\bar\delta_{j+1}(\beta) \bigr)
+ O \bigl(L^{\alpha} \bigr) \biggr).
\end{eqnarray*}
Observe that
\[
\mathcal{K}=e^{-\beta(j-1-H)}\gg e^{-4\beta(j-1-H)}=L e^{-4\beta
(j-1)}\asymp\bar
\delta_{j-1}(\beta).
\]
Therefore, for large $\beta$ one sees that
\[
-\mathcal{K}\leq\mathcal{K} \bigl(-1+2\delta_1(\beta) \bigr)+\bar
\delta_j(\beta)-\bar\delta_{j+1}(\beta)\leq- \mathcal{K}/2.
\]
Since the term proportional to $|\Lambda_{(j)}|$ is independent of
$\Gamma$, it
plays no role in (\ref{Blemma3}). Therefore,
%
%
\begin{equation}
\label{Blemma40} \quad\mu(E_j|G)\leq\exp \bigl(O \bigl(L^\alpha
\bigr) \bigr)\times\frac{\sum_{\Gamma\in E_j\cap
G}\nu(\Gamma)
\exp( -(\mathcal{K}/2L)|\Lambda_+| )} {
\sum_{\Gamma\in G}\nu(\Gamma)
\exp( -(\mathcal{K}/L)|\Lambda_+| )},
\end{equation}
where we recall that $E_j$ is the event in (\ref{Blemma2}), $G$ the
event in~(\ref{Blemmagep}) and $\nu$ the measure in~(\ref{dksnu}). At
this point an upper bound on $\mu(E_j|G)$ follows
from~(\ref{Blemma40}) by neglecting the negative exponent in the
numerator and using Jensen's inequality for the denominator. Using
also~(\ref{ggep}), this gives
%
%
\begin{equation}
\label{Blemma51} \mu(E_j|G)\leq\nu(E_j) \exp \biggl(
\frac{\mathcal{K}}{L}\nu \bigl(|\Lambda_+| \bigr) + O \bigl(L^{\alpha} \bigr) \biggr).
\end{equation}
%
It follows then from Proposition~\ref{dkspro} that for every $\beta$
sufficiently large, for all $\varepsilon>0$, if $L$ is large enough:
%
%
\begin{equation}
\label{Blemma5} \nu(E_j)\leq\exp \bigl( -L^{1-\varepsilon} \bigr).
\end{equation}
Essentially, under $\nu$ the contour $\Gamma$ behaves like a random walk
and the event $E_j$ imposes a large deviation of order $L/(\log L)^2$
which is much larger than the typical diffusive fluctuation $\sqrt L$.
Moreover, again from Proposition~\ref{dkspro} one has
%
%
\begin{equation}
\label{Blemma52} \nu \bigl(|\Lambda_+| \bigr)=O \bigl(L^{(3/2)+\varepsilon} \bigr).
\end{equation}
Then~(\ref{Blemma51}), (\ref{Blemma5}) and~(\ref{Blemma52}) end the
proof of Lemma~\ref{Blemma}.
\end{pf*}

\begin{pf*}{Proof of Proposition~\ref{clam}}
Let us start with the lower bounds. Using first Jensen's inequality and
then the FKG property for $\omega_-$ one has
%
%
\begin{eqnarray}
\label{Blemma44} \qquad &&\omega_- \bigl(e^{(\mathcal{K}/L)\sum
_{x\in\Lambda_-}{\mathbf 1}_{\eta
_x\leq
0}}; \eta\geq- (j-1) \bigr)
\nonumber
\\
&&\qquad\geq\exp \biggl[\frac{\mathcal{K}}L\sum_{x\in\Lambda
_-}\omega_- \bigl(\eta_x\leq0|\eta\geq- (j-1) \bigr) \biggr]\omega_-
\bigl(\eta\geq- (j-1) \bigr)
\\
&&\qquad\geq\exp \biggl[\frac{\mathcal{K}}L |\Lambda_-| - \tilde \Phi_--\tilde
\Psi_- \biggr],
\nonumber
\end{eqnarray}
where
\begin{eqnarray*}
\tilde\Phi_-&:=&\frac{\mathcal{K}}L\sum_{x\in\Lambda_-}
\omega_- \bigl(\eta_x> 0|\eta\geq- (j-1) \bigr),
\\
\tilde\Psi_-&:=&-\sum_{x\in\Lambda_-}\log \bigl(1-\omega_-
\bigl( \eta_x< -(j-1) \bigr) \bigr).
\end{eqnarray*}
Similarly,
%
%
\begin{equation}
\label{Blemma55} \omega_+ \bigl(e^{(\mathcal{K}/L)\sum_{x\in
\Lambda_+}{\mathbf
1}_{\eta
_x\leq
-1}}; \eta\geq- j \bigr) \geq\exp[
\tilde\Phi_+ -\tilde\Psi_+ ],
\end{equation}
where
\begin{eqnarray*}
\tilde\Phi_+&:=&\frac{\mathcal{K}}L\sum_{x\in\Lambda_+}\omega
_+(\eta_x\leq-1|\eta\geq-j),
\\
\tilde\Psi_+&:=&-\sum_{x\in\Lambda_+}\log \bigl(1-\omega_+(
\eta_x< -j) \bigr).
\end{eqnarray*}
%
From (\ref{ekubdec}) one sees that 
both $|\Psi_-- \tilde\Psi_-|$ and $|\Psi_+- \tilde\Psi_+|$ are
$O(L^\alpha)$, for some $\alpha<1$, uniformly in $\Gamma\in G$.
Therefore, the
lower bound in (\ref{Blemma30}) follows once we establish that on $G$
%
%
\begin{equation}
\label{Blemma5550} |\Phi_\pm- \tilde\Phi_\pm| = O
\bigl(L^\alpha \bigr)
\end{equation}
for some $\alpha>0$. To prove (\ref{Blemma5550}), we use the following
comparison estimate. Let us consider the case $|\Phi_+-\tilde\Phi_+|$.
By FKG, one has $\omega_+(\eta_x\leq-1|\eta\geq-j)\leq
\omega_+(\eta_x\leq-1)$. On the other hand, whenever $x\in\Lambda_-$
is at
distance at least $L^\delta$, for some $\delta>0$, from $\partial
\Lambda
_+$, then
we claim that
%
%
\begin{equation}
\label{Blemma555} \omega_+(\eta_x\leq-1) \leq\omega_+(
\eta_x\leq-1|\eta\geq-j) + O \bigl(L^{\alpha-1} \bigr).
\end{equation}
These observations and (\ref{ekubdec}) are sufficient to prove
(\ref{Blemma5550}). In turn, (\ref{Blemma555}) is a~consequence of the
technique developed below, cf. the comment after~(\ref{Blef}).

To prove the upper bounds in (\ref{Blemma30}) and (\ref{Blemma31}),
observe that from the FKG property of $\omega_\pm$ one has
\begin{eqnarray*}
&& \omega_- \bigl(e^{(\mathcal{K}/L)\sum_{x\in\Lambda
_-}{\mathbf 1}_{\eta_x\leq 0}}; \eta\geq- (j-1) \bigr)
\\[-3pt]
&&\qquad \leq\omega_- \bigl(e^{(\mathcal{K}/L)\sum_{x\in\Lambda_-}{\mathbf
1}_{\eta _x\leq 0}} \bigr) \omega_- \biggl(\prod_{x\in\Lambda_-}{
\mathbf1}_{\eta_x\geq- (j-1)} \biggr),
\\[-3pt]
&& \omega_+ \bigl(e^{(\mathcal{K}/L)\sum_{x\in\Lambda_+}{\mathbf
1}_{\eta_x\leq-1}}; \eta\geq- j \bigr)
\\[-3pt]
&&\qquad \leq\omega_+
\bigl(e^{(\mathcal{K}/L)\sum_{x\in\Lambda_+}{\mathbf
1}_{\eta_x\leq-1}} \bigr) \omega_+ \biggl(\prod_{x\in\Lambda_+}{
\mathbf1}_{\eta_x\geq- j} \biggr).
\end{eqnarray*}
Rewriting
\begin{eqnarray*}
\omega_- \bigl(e^{(\mathcal{K}/L)\sum_{x\in\Lambda
_-}{\mathbf
1}_{\eta
_x\leq
0}} \bigr) 
&=&\exp \biggl(\frac{\mathcal{K}}L |\Lambda_-| \biggr)\omega_-
\biggl(\prod_{x\in\Lambda
_-}(1-\varphi_x) \biggr),
\end{eqnarray*}
where $\varphi_x:=1- e^{-(\mathcal{K}/L){\mathbf1}_{\eta_x> 0}}$, and
setting $\psi_x={\mathbf1}_{\eta_x<-(j-1)}$ the bound (\ref{Blemma30})
is then implied by
%
%
\begin{eqnarray}
\label{Blemma6} \omega_- \biggl(\prod_{x\in\Lambda_-}(1-
\varphi_x) \biggr)&\leq&\exp \bigl(-\Phi_- + O \bigl(L^{\alpha}
\bigr) \bigr),
\\[-3pt]
\label{Blemma61}\omega_- \biggl(\prod_{x\in\Lambda_-}(1-\psi_x)
\biggr)&\leq&\exp \bigl(-\Psi_- + O \bigl(L^{\alpha} \bigr) \bigr).
\end{eqnarray}
Similarly, the bound (\ref{Blemma31}) is implied by
%
%
\begin{eqnarray}
\label{Ble60} \omega_+ \biggl(\prod_{x\in\Lambda_+}(1-\bar
\varphi_x) \biggr)&\leq&\exp \bigl(\Phi_+ + O \bigl(L^{\alpha}
\bigr) \bigr),
\\
\omega_+ \biggl(\prod_{x\in\Lambda_+}(1-\bar
\psi_x) \biggr)&\leq&\exp \bigl(-\Psi_+ + O \bigl(L^{\alpha}
\bigr) \bigr) \label{Ble610}
\end{eqnarray}
with the notation $\bar\varphi_x:=1- e^{(\mathcal{K}/L){\mathbf
1}_{\eta_x< 0}}$, and $\bar\psi_x={\mathbf1}_{\eta_x<-j}$. Below, we
establish~(\ref{Blemma6})--(\ref{Ble610}) and~(\ref{Blemma555}). All
these estimates can be achieved once one has an approximate
factorization of the measure $\omega_+$ on a mesoscopic scale $L^u$,
$u\in(0,\frac12)$. To illustrate this point, consider the expression
(\ref{Blemma6}), and suppose the product is confined to $Q_u$, a square
with side $L^u$, contained in $\Lambda_-$. Then
%
%
\begin{eqnarray}
\label{Blemma7} \omega_- \biggl(\prod_{x\in Q_u}(1-
\varphi_x) \biggr)&=&\sum_{A\subset
Q_u}(-1)^{|A|}
\omega_- \biggl(\prod_{x\in A}\varphi_x
\biggr)
\nonumber
\\
&=&1-\sum_{x\in Q_u}\omega_-(\varphi_x)+O
\biggl(\sum_{k\geq2}\pmatrix{L^{2u}
\vspace*{2pt}
\cr
k}L^{-k} \biggr)
\\
&\leq&\exp \biggl(- \frac{\mathcal{K}}L\sum_{x\in Q_u}
\omega_-(\eta_x>0)+ O \bigl(L^{2(2u-1)} \bigr) \biggr),
\nonumber
\end{eqnarray}
where we have separated the contributions of sets $A$ with $|A|\leq1$
and $|A|\geq2$, and used the fact that $\varphi_x=\frac{\mathcal
{K}}L{\mathbf1}_{\eta_x> 0} + O(L^{-2})$. In particular if one could
factorize (\ref{Blemma6}) into a product of (\ref{Blemma7}) over all
$Q_u\subset\Lambda_-$, then the desired bound would follow using also
(\ref{ekubdec}).

To implement this idea, we use the following geometric construction.
Partition $\mathbb{Z}^2$ into squares $P$ with side $r=L^u+2L^\delta$,
where $0<\delta<u<\frac12$ (we assume for simplicity that
$L^u,L^\delta$ are both integers). Consider squares $Q$ of side $L^u$
centered inside the squares $P$ in such a way that each square $Q$ is
surrounded within $P$ by a shell of thickness $L^\delta$, see
Figure~\ref{shell}. Define the set $\mathcal{S}$ of dual bonds
associated to a nonzero height gradient, cf. Appendix~\ref{clusters}.
The set $\mathcal{S}$ is decomposed into connected components
(clusters) $S$.
We call $\mathcal{I}(\delta)$ the collection of clusters $S$ in
$\mathcal{S}$ such that
$|S|\geq L^\delta$. Note that a cluster may have a nonempty interior.

Consider the set of sites $V\subset\Lambda_-$ defined as 
what remains after we remove from $\Lambda_-$ all clusters $S$ in
$\mathcal{I} (\delta)$ together with their interior. A square
$Q\subset
V$ is called \textit{good} if the square $P\supset Q$ has empty
intersection with $\partial\Lambda_-\cup\mathcal{I}(\delta)$;
see Figure~\ref{shell}. We write $\mathcal{G}$ for the collection of
good squares $Q$. The crucial observation is that if $Q\in\mathcal{G}$,
then there exists a circuit $\mathcal{C}$ of bonds of $\mathbb{Z}^2$
surrounding $Q$ and contained in the square $P\supset Q$, such that
$\eta_\mathcal{C}\equiv0$. To see this, observe that there must be a
circuit $\mathcal{C}$ of bonds surrounding $Q$ such that gradients of
$\eta$ along the circuit are~$0$, since otherwise there would be a path
of dual bonds connecting $Q$ with $P^c$ with cross-gradients different
from zero, and therefore a cluster $S$ with size larger than~$L^\delta$
intersecting $P$. Now, this implies that $\eta$ is constant on
$\mathcal{C}$, and this constant must be zero, since otherwise $Q$
would belong to the interior of a cluster $S$ of size larger than
$L^\delta$ because of the zero boundary condition on
$\partial\Lambda_-$.
%
%
\begin{figure}

\includegraphics{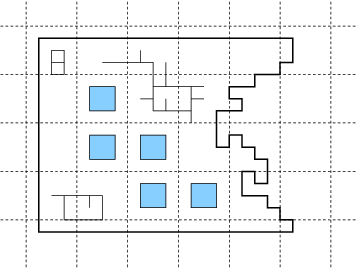}

\caption{A drawing of the region $\Lambda_-$. In the background the
squares $P$ (dashed lines). The clusters inside represent the set
$\mathcal{I} (\delta)$ and the shaded squares are the set $\mathcal
{G}$ of good squares~$Q$.}\label{shell}
\end{figure}

Next, we estimate $1-\varphi_x\leq1$ for all $x$ which do not belong
to some $Q\in\mathcal{G}=\mathcal{G}(\mathcal{I}(\delta))$.
Therefore, summing over all
possible realizations $W$ of~$\mathcal{I}(\delta)$:
%
%
\begin{eqnarray}\label{Blemma70}
&& \omega_- \biggl(\prod_{x\in\Lambda_-}(1-
\varphi_x) \biggr)\nonumber
\\
&&\qquad \leq\sum_{W} \omega_- \bigl(\mathcal{I}( \delta)=W \bigr) \omega_-
\biggl(\prod _{Q\in\mathcal{G}}\prod_{x\in Q}(1-\varphi _x)|\mathcal{I}
( \delta)=W \biggr)
\\
&&\qquad \leq\sum_{W}\omega_- \bigl(\mathcal{I}(
\delta)=W \bigr) \prod_{Q\in\mathcal{G}}\sup_{\mathcal{C}}
\hat\pi{}_{\mathcal{C}}^0 \biggl(\prod
_{x\in Q}(1- \varphi_x) \biggr),\nonumber
\end{eqnarray}
where, for an arbitrary circuit $\mathcal{C}$ surrounding $Q$ within
the square $P\supset Q$ and with a slight abuse of notation, we write
$\hat \pi_{\mathcal{C}}^0$ for the SOS equilibrium measure on the
interior of the circuit $\mathcal{C}$ with zero boundary conditions
(without floor, ceiling and no fields). With the same argument of
(\ref{Blemma7}) one has, uniformly in $\mathcal{C}$
%
%
\begin{equation}
\label{Blemma9} \hat\pi{}_{\mathcal{C}}^0 \biggl(\prod
_{x\in Q}(1-\varphi_x) \biggr)\leq\exp \biggl(-
\frac{\mathcal{K}}L\sum_{x\in Q}\hat\pi{}_{\mathcal{C}}^0(
\eta_x>0)+ O \bigl(L^{2(2u-1)} \bigr) \biggr).
\end{equation}
Let $\mathcal{I}(\delta)=W$ be fixed. For any fixed square $Q\in
\mathcal{G}$, let $Q'\subset Q$ be the square centered inside $Q$ in
such a way that $Q'$ is surrounded by a shell of thickness $L^\delta$
within $Q$. Thus, if $x\in Q'$, then $x$ is at distance at least
$L^\delta$ from $\mathcal {C}$, and therefore as in (\ref{ekubdec}),
for any $p>0$, uniformly in $\mathcal{C}$:
%
%
\begin{equation}
\label{Blemma10}
\qquad \hat\pi{}_{\mathcal{C}}^0(\eta_x>0) =
\hat\pi(\eta_x>0) +O \bigl(L^{-p} \bigr),\qquad x\in Q'.
\end{equation}
From (\ref{Blemma10}),
%
%
\begin{eqnarray}\label{Blemma99}
\qquad && \hat\pi{}_{\mathcal{C}}^0 \biggl(\prod
_{x\in Q}(1-\varphi_x) \biggr)
\nonumber\\[-8pt]\\[-8pt]
&&\qquad \leq\exp \biggl(-\frac{\mathcal{K}}L\sum_{x\in Q'}\hat\pi(
\eta_x>0)+ O \bigl(L^{2(2u-1)} \bigr) + O \bigl(L^{-p+2u-1}
\bigr) \biggr).\nonumber
\end{eqnarray}
There are at most $O(L^{2-2u})$ squares $Q$. Therefore, from
(\ref{Blemma70}) one obtains
%
%
\begin{eqnarray}\label{Blemma71}
\qquad && \omega_- \biggl(\prod_{x\in\Lambda_-}(1-\varphi_x)
\biggr)
\nonumber\\[-8pt]\\[-8pt]
&&\qquad \leq\sum_{W}\omega_- \bigl(\mathcal{I}(
\delta)=W \bigr) \exp \biggl(- \frac{\mathcal{K}}L\sum
_{Q\in\mathcal{G}}\sum_{x\in
Q'}\hat\pi(\eta_x>0)+O \bigl(L^{2u} \bigr) \biggr).\nonumber
\end{eqnarray}
Next, we need to add back the contributions to the exponent in
(\ref{Blemma71}) from all removed vertices, where each vertex
contributes at most $1/L$. The contribution of a single removed shell
$P\setminus Q'$ is $O(L^{\delta+u-1})$, and they are at most
$O(L^{2-2u})$, so~that all removed shells give at most
$O(L^{1+\delta-u})=O(L^\alpha )$ for $\alpha<1$ since $u>\delta$. To
estimate the contribution from all other removed sites, we observe that
a site can be removed if it belongs to a square $P$ that intersects
either the boundary $\partial\Lambda_-$ or the clusters of
$\mathcal{I}(\delta)$, or if it belongs to the interior of a cluster of
$\mathcal{I}(\delta)$. If $A(\mathcal{I}(\delta))$ denotes the total
number of sites in the interior of the clusters
$S\in\mathcal{I}(\delta)$, then these contribute at most
$\mathcal{K}L^{-1}\times A(\mathcal{I}(\delta))$. Moreover, one has at
most $L^{2u}\times |\mathcal{I}(\delta)|$ sites that can be removed
from intersections with $\mathcal{I} (\delta)$. These contribute at
most $\mathcal{K}L^{2u-1} |\mathcal{I}(\delta )|$. Finally, one
estimates roughly by $L^{2u}|\partial\Lambda_-|=O(L^{2u+1+\varepsilon
_0})$ the number of sites removed from squares intersecting
$\partial\Lambda_-$, since on the event $G$ one has $|\partial\Lambda
_-|=O(L^{1+\varepsilon_0})$. Thus, the contribution from the boundary
squares is $O(L^{2u+\varepsilon_0})=O(L^\alpha)$, if
$2u+\varepsilon_0<1$. Therefore, using $\mathcal{K} \leq 1$:
%
%
\begin{eqnarray}\label{Blemma72}
&& \frac{\mathcal{K}}L\sum_{Q\in\mathcal{G}}\sum_{x\in Q'}\hat\pi(\eta_x>0)
\nonumber\\[-8pt]\\[-8pt]
&&\qquad \geq\Phi_- - L^{-1} A \bigl(\mathcal{I}(\delta) \bigr)-L^{2u-1} \bigl|
\mathcal{I}( \delta) \bigr| + O \bigl(L^{\alpha} \bigr).\nonumber
\end{eqnarray}
Thus, we have obtained
%
%
\begin{eqnarray}\label{Blemma73}
\qquad&& \omega_- \biggl(\prod_{x\in\Lambda_-}(1-
\varphi_x) \biggr)
\nonumber\\[-8pt]\\[-8pt]
&&\qquad \leq\exp \bigl(- \Phi_- + O \bigl(L^{\alpha}
\bigr) \bigr) \omega_- \bigl[\exp \bigl(L^{-1} A \bigl(\mathcal{I}(
\delta) \bigr) + L^{2u-1} \bigl|\mathcal{I}(\delta) \bigr| \bigr) \bigr].\nonumber
\end{eqnarray}
If $\{S_i\}_{i=1}^m$ denotes the collection of clusters of $ \mathcal{I}
(\delta)$,
with $|\mathcal{I}(\delta)|=\sum_i|S_i|$, then
%
%
\begin{equation}
\label{isopero} A \bigl(\mathcal{I}(\delta) \bigr)\leq\frac 14\sum
_i|S_i|^2\leq\frac14 \biggl(\sum
_i|S_i| \biggr)^2,
\end{equation}
where the last bound follows from $\sum_i x_i^2\leq(\sum_i x_i)^2$ for
all $x_i\geq0$. Recalling that $|S_i|\geq L^\delta$ for all $i$, using
(\ref{peierls_c1}), letting $m$ represent the number of clusters
$S_1,\ldots,S_m$, and summing over their starting points
$x_1,\ldots,x_m$, one has the estimate
\[\hspace*{-10pt}
\omega_- \biggl(\sum_i|S_i| \geq k
\biggr) \leq\sum_{m\geq1}\sum
_{x_1,\ldots,x_m}\sum_{S_1\ni x_1}\cdots\sum
_{S_m\ni x_m}C e^{-\beta(|S_1|+\cdots+|S_m|)/2}\chi(S_1,
\ldots,S_m),
\]
where $\chi(S_1,\ldots,S_m)=1$ if $|S_i|\geq L^\delta$ for all
$i=1,\ldots,m$ and $|S_1|+\cdots+|S_m|\geq k$, and
$\chi(S_1,\ldots,S_m)=0$ otherwise. Therefore,
%
%
\begin{eqnarray}
\label{peierls_c2} \omega_- \biggl(\sum
_i|S_i| \geq k \biggr)&\leq& C
e^{-\beta k/4} \sum_{m\geq1} \biggl( \sum
_{x}\sum_{S\ni x, |S|\geq L^\delta}
e^{-\beta|S|/4} \biggr)^m
\nonumber
\\
&\leq& C e^{-\beta k/4}\sum_{m\geq1}
L^{2m} \biggl(\sum_{j\geq
L^\delta
}C^je^{-\beta j/4}
\biggr)^m
\\
&\leq& e^{-\beta k/4} \nonumber
\end{eqnarray}
for any $\beta$ large enough and for all $L$ sufficiently large. From
(\ref{isopero}) and (\ref{peierls_c2}), one has
\[
\omega_- \bigl(A \bigl(\mathcal{I}(\delta) \bigr)\geq\ell \bigr)\leq\omega_-
\biggl( \sum_i|S_i|\geq2\sqrt\ell
\biggr)\leq e^{-\beta\sqrt{\ell}/2}
\]
for all $\ell>0$ and therefore
%
%
\begin{equation}
\label{isopero1} \omega_- \bigl[\exp \bigl(2 L^{-1} A \bigl(
\mathcal{I}(\delta) \bigr) \bigr) \bigr] \leq\sum_{\ell=0}^{L^2}
\exp{(2\ell/L-\beta\sqrt{\ell}/4)}\leq L^2+1,
\end{equation}
since $2\ell/L\leq\beta\sqrt{\ell}/4$, for $\beta$ large and
$\ell\leq
L^2$. Using (\ref{isopero1}), a Cauchy--Schwarz inequality and
(\ref{peierls_c2}), it follows that
%
%
\begin{equation}
\label{Blef} \omega_- \bigl[\exp \bigl(L^{-1} A \bigl(\mathcal{I}(
\delta) \bigr) + L^{2u-1}\bigl|\mathcal{I} (\delta)\bigr| \bigr) \bigr]\leq CL\leq
e^{O(L^\alpha)}
\end{equation}
for any $\alpha<1$. This ends the proof of (\ref{Blemma6}). To prove
(\ref{Ble60}) one repeats the same argument with the region $\Lambda_-$
replaced by $\Lambda_+$.

We turn to the proof of the estimates (\ref{Blemma61}) and
(\ref{Ble610}). A minor modification of the same argument proves also
the inequality (\ref{Blemma555}). Here one has to replace the expansion
(\ref{Blemma9}) by the following bound:
%
%
\begin{eqnarray}\label{Blemma90}
&& \hat\pi{}_{\mathcal{C}}^0 \biggl(\prod
_{x\in Q}(1-\psi_x) \biggr)
\nonumber\\[-8pt]\\[-8pt]
&&\qquad \leq\exp \biggl(- \sum_{x\in Q} \hat\pi{}_{\mathcal{C}}^0(
\psi_x)+ O \bigl(L^{-(3/2)+2u+c(\beta)} \bigr)+O \bigl(L^{6u-3}\bigr) \biggr),\nonumber
\end{eqnarray}
where $\psi_x={\mathbf1}_{\eta_x\geq j}$ and $c(\beta)>0$ can be made
arbitrarily small by taking $\beta$ large enough. Once this estimate
(together with the corresponding statement for $\bar\psi_x={\mathbf
1}_{\eta_x>j}$) is available, it is not hard to check that exactly the
same arguments we used to prove (\ref{Blemma6}) and (\ref{Ble60}) allow
one to conclude. Here the term $\mathcal{K}L^{-1}\hat\pi(\eta_x>0)$
appearing in (\ref{Blemma72}) must be replaced by $\hat\pi(\eta _x\geq
j)$, which (thanks to $j\geq H+1$) is again less than $L^{-1}$ for
$\beta$ large enough by (\ref{ekub}). In particular, one can use the
argument in (\ref{isopero1})--(\ref{Blef}) to conclude as above.

It remains to prove (\ref{Blemma90}). We cannot proceed as in
(\ref{Blemma7}) since $\psi_x$ is not pointwise $O(1/L)$. From
Bonferroni's inequality (inclusion--exclusion principle), one has
%
%
\begin{eqnarray}
\label{Blemma07} \hat\pi{}_{\mathcal{C}}^0 \biggl(\prod
_{x\in Q}(1- \psi_x) \biggr)&=&\sum
_{A\subset Q}(-1)^{|A|} \hat\pi{}_{\mathcal{C}}^0
\biggl(\prod_{x\in A} \psi_x \biggr)
\nonumber
\\[-8pt]
\\[-8pt]
&\leq&1-\sum_{x\in Q}\hat\pi{}_{\mathcal{C}}^0(
\psi_x)+ \frac12\mathop{\sum_{x,y\in Q\dvtx}}_{x\neq y}
\hat\pi{}_{\mathcal{C}}^0(\psi_x\psi_y).\nonumber
\end{eqnarray}
Next, observe that
%
%
\begin{eqnarray}\label{griglia88}
&& \mathop{\sum_{x,y\in Q\dvtx}}_{x\neq y}
\hat\pi{}_{\mathcal{C}}^0(\psi_x\psi_y)
\nonumber\\[-8pt]\\[-8pt]
&&\qquad = \biggl(\sum_{x\in Q}\hat\pi{}_{\mathcal{C}}^0(
\psi_x) \biggr)^2 + \mathop{\sum
_{x,y\in Q\dvtx}}_{x\neq y}\hat\pi{}_{\mathcal{C}}^0(
\psi_x;\psi_y) + O \bigl(L^{-2+2u} \bigr),\nonumber
\end{eqnarray}
where $\hat\pi{}_{\mathcal{C}}^0(\psi_x;\psi_y):=\hat\pi{}_{\mathcal
{C}}^0(\psi_x\psi
_y)-\hat\pi{}_{\mathcal{C}}^0(\psi_x)\hat\pi{}_{\mathcal{C}}^0(\psi_y)$,
and we use $\hat\pi{}_{\mathcal{C}}^0(\psi_x)=O(1/L)$. We need the
following bound.
For some $c(\beta)\to0$ as $\beta\to\infty$, one has
%
%
\begin{equation}
\mathop{\sum_{x,y\in Q\dvtx}}_{x\neq y}\hat
\pi_{\mathcal{C}}^0(\psi_x;\psi_y)=O
\bigl(L^{-3/2+2u+c(\beta)} \bigr). \label{griglia8}
\end{equation}
%
The bound (\ref{Blemma90}) follows immediately from
(\ref{Blemma07})--(\ref{griglia8}) and the fact that $ (\sum_{x\in
Q}\hat\pi{}_{\mathcal{C}}^0(\psi_x) )^3 = O(L^{6u-3})$.

To prove (\ref{griglia8}), first notice that by exponential decay of
correlations \cite{BW}:
\[
\hat\pi{}_{\mathcal{C}}^0(\psi_x;\psi_y)
\leq c_1e^{-c_2|x-y|}
\]
for some constants $c_1,c_2>0$. Therefore (\ref{griglia8}) follows if
we prove that for any constant $C>0$,
\[
\sum_{0\neq|y|\leq C\log L}\hat\pi{}_{\mathcal{C}}^0(
\psi_0;\psi_y)=O \bigl(L^{-3/2+c(\beta)} \bigr).
\]
In particular, it suffices to show that uniformly in $y\neq0$:
%
%
\begin{equation}
\label{griglia81} \hat\pi{}_{\mathcal{C}}^0(\psi_0
\psi_y)=O \bigl(L^{-3/2+c'(\beta)} \bigr)
\end{equation}
for some constant $c'(\beta)\to0$ as $\beta\to\infty$. The proof of
(\ref{griglia81}) goes as follows. Let $E_k$ denote the event that
there exists some $k$-contour $\gamma$ that contains both $0,y$. Then
$E_{k+1}\subset E_k$, $k\geq1$, and
\[
\hat\pi{}_{\mathcal{C}}^0(\psi_0\psi_y)=
\hat\pi{}_{\mathcal{C}}^0 \bigl(\psi_0\psi
_y;E_1^c \bigr) + \sum
_{k=1}^{j}\hat\pi{}_{\mathcal{C}}^0
\bigl(\psi_0\psi_y; E_k\cap
E_{k+1}^c \bigr) + \hat\pi{}_{\mathcal{C}}^0(
\psi_0\psi_y;E_{j+1}).
\]
Now, if $\psi_0\psi_y\cap E_1^c$ occurs, then there must be two
separate families of nested contours reaching level $j$, one around $0$
and the other around $y$. By repeating the argument in the proof of
Proposition~\ref{p-heightBounds}, one has
\[
\hat\pi{}_{\mathcal{C}}^0 \bigl(\psi_0
\psi_y;E_1^c \bigr) =O
\bigl(e^{-8\beta j} \bigr) = O \bigl(L^{-2} \bigr).
\]
If $\psi_0\psi_y\cap E_k\cap E_{k+1}^c$ occurs, then there must be
nested contours around $0$ and around $y$ separately from level $k+1$
to level $j$ and there must be nested contours from level $1$ to level
$k$ comprising both $0$ and $y$. In this case, the argument in the
proof of Proposition~\ref{p-heightBounds} yields
\[
\hat\pi{}_{\mathcal{C}}^0 \bigl(\psi_0
\psi_y; E_k\cap E_{k+1}^c
\bigr)=O \bigl(e^{-8\beta
(j-k)}e^{-\beta(1-c(\beta)) k\ell_{y}} \bigr),
\]
where $\ell_{y}$ denotes the length of the shortest contour comprising
both $0,y$ and $c(\beta)$ decays as fast as $1/\beta$. Since $\ell
_{y}\geq
6$, the above expression is\break  $O(e^{-6\beta(1-c(\beta)) j}) =
O(L^{-3/2+c'(\beta)}) $ for every $k\leq j$. Finally, the same argument
shows that
\[
\hat\pi{}_{\mathcal{C}}^0(\psi_0\psi_y;E_{j+1})
= O \bigl(e^{-6\beta(1-c(\beta)) j} \bigr) = O
\bigl(L^{-3/2+c'(\beta)} \bigr).
\]
These estimates imply (\ref{griglia81}). This ends the proof of
Proposition~\ref{clam}.
\end{pf*}

\section{Mixing time in absence of entropic repulsion: Proof of
Theorem~\texorpdfstring{\protect\ref{th-nomuro}}{3}}

Like for Theorem~\ref{th-principale}, it is sufficient to give the
proof when $n^+=\log L$, the general case following easily (one needs
to generalize the approach of Section~\ref{sec-genn+} in the obvious
way). For simplicity of notations, we call the SOS equilibrium measure
with zero b.c. on $\partial\Lambda_L$ and floor/ceiling at $\pm\log L$
simply $\pi$. Recall the definition (\ref{eq-Ri}) of the diagonal lines
$R_i$ and define, for $1\leq j\leq N=(2L-1)/L^{1/2+\varepsilon}$
(assume for simplicity that $N$ and $L^{1/2+\varepsilon}$ are
integers), the subset $W_j$ of $\Lambda_L$ as
\[
W_j=\bigcup_{i\in\mathcal{I}_j} R_i,\qquad
\mathcal{I}_j= \biggl\{ \frac{j-1}2 L^{1/2+\varepsilon}<i\leq
\frac
{j+1}2 L^{1/2+\varepsilon} \biggr\}.
\]
Note that $W_j\cap W_{j+1}$ is a roughly rectangular-shaped region of
smaller side of order $L^{1/2+\varepsilon}$. Let also $S_j$ denote the
``brick'' with horizontal projection $W_j$ and floor/ceiling at
$\pm\log L$. From (\ref{eq-22}) and symmetry, we know that we have only
to show that
%
%
\begin{equation}
\label{eq-8} \bigl\|\mu^\sqcap_T-\pi \bigr\|\leq L^{-3}
\end{equation}
with $T=\exp(c\beta L^{1/2+2\varepsilon})$ and some large constant $c$.

The proof is somewhat similar (but definitely simpler) to that of
Lemma~\ref{lem-dallalto}, so we will be very sketchy. The
simplification is that, since the floor at $-\log L$ has essentially no
effect at equilibrium, it is not necessary to introduce the field
term~(\ref{eq-10}) to compensate the entropic repulsion.

We apply Theorem~\ref{th-PW} with the following censoring protocol. We
let $\Delta T=T/N$ and we let evolve first the brick $S_1$ for a
time-lag $\Delta T$, then $S_2$ for another time-lag $\Delta T$, and so
on up to $S_N$. From Proposition~\ref{prop-canpaths} (which is
immediately adapted to the case where $\Lambda$ is not \textit{exactly}
a $L\times m$ rectangle but rather is \textit{included} in some,
possibly tilted, $L\times m$ rectangle) we have that the mixing time in
each brick $S_j$, uniformly on the b.c. around it, is $\exp(O(\beta n^+
L^{1/2+\varepsilon}))$. Therefore, if $c$ in the definition of $T$ is
sufficiently large, we can assume (modulo a negligible error term) that
after the $j$th time-lag the $j$th brick is exactly at equilibrium,
with $0$ b.c. on $\partial W_j\cap\partial\Lambda_L$, b.c. $n^+=\log L$
on $\partial W_j\cap W_{j+1}$ and, on $\partial W_j\cap W_{j-1}$, a
b.c. determined by the result of the evolution in the $(j-1)$th
time-lag. Theorem~\ref{th-PW} then guarantees that the l.h.s. of
(\ref{eq-8}) is smaller than $\|\tilde\mu{}^\sqcap_T-\pi\|$, with
$\tilde\mu_T^\sqcap$ the law at time $T$ of the censored dynamics. The
inequality (\ref{eq-8}) then follows (via a repeated application of
DLR) provided that one proves that, if $\pi_j$ denotes the equilibrium
on $U_j=W_1\cup\cdots\cup W_j$ with $0$ b.c. on $\partial U_j\cap
\partial \Lambda_L$ and $n^+$ b.c. on $\partial U_j\cap W_{j+1}$, then
%
%
\begin{equation}
\label{eq-24} \|\pi_j-\pi\|_{U_{j-1}}=O \bigl(L^{-4}
\bigr),
\end{equation}
that is, the marginals of the two measures on $U_{j-1}$ are very close.
%
%
\begin{figure}

\includegraphics{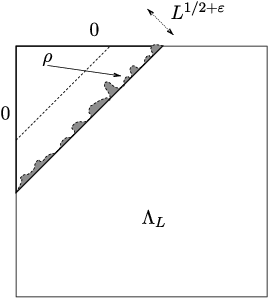}

\caption{A drawing of the triangular region $\Lambda'$ and of the chain
$\rho$.}\label{fig:sec8}
\end{figure}

In analogy with the way Theorem~\ref{lem-dolore_infinito} follows from
Lemma~\ref{lem-1} (cf. Section~\ref{accoppiamento}), to get
(\ref{eq-24}) it is sufficient to prove that the open $1$-contour does
not intersect $W_{j-1}$, except with probability $O(L^{-C})$. In turn
(and in analogy to how Lemma~\ref{lem-1} follows from
Lemma~\ref{Blemma}), the desired upper bound on the deviation of the
$1$-contour follows if we prove the following. Consider a diagonal line
$R_i$, with $i\geq L^{1/2+\varepsilon}$. Let $\Lambda'=\bigcup_{a\leq
i}R_a$ and let $\rho$ be a chain of sites in $\Lambda'$, connecting two
adjacent sides of $\Lambda_L$, and at distance at most
$L^{\varepsilon}$ from $R_i$, see Figure~\ref{fig:sec8}.

The chain $\rho$ disconnects $\Lambda_L$ into two subsets and call
$\Lambda_-$ the one containing the Northwest corner of $\Lambda_L$. Let
$\pi'$ be the SOS measure on $\Lambda_-$, with $0$ b.c. on
$\partial\Lambda_-\cap \partial\Lambda_L$ and $1$ b.c. on $\rho$. Then,
the $\pi'$-probability that the unique $1$-contour reaches distance
$L^{1/2+\varepsilon}$ from $\rho$ is smaller than any inverse power of
$L$.

This is much easier to prove than the somewhat similar estimate of
Lemma~\ref{Blemma}. The reason is that, since the fields (\ref{eq-10})
are absent and the floor has a negligible effect (recall that the floor
was instead at height zero in Lemma~\ref{Blemma}), the desired estimate
follows directly from a suitable modification of
Proposition~\ref{dkspro}, where the square $Q_L$ is replaced by a
triangular domain.


\begin{appendix}
\section{Peierls' estimates and low-temperature~expansion}\label{clusters}
Here we collect some rather standard facts concerning the
low-temperature expansion of the SOS model. With a small abuse of
notation let $Z_\Lambda$ be the partition function corresponding to the
measure $\hat\pi{}_\Lambda^0$.
Following \cite{BW}, Section~2, we will write $Z_\Lambda$ as a sum over
compatible cluster configurations.

%
\begin{definition}
A \textit{cluster} $X$ is a tuple $(\gamma,h_1,\ldots,h_{|\gamma|})$,
where $\gamma$ is a finite connected set of dual lattice bonds, and
$h_i\in\mathbb{Z} \setminus\{0\}$. A cluster configuration is
a~collection of clusters $\{X_1,\ldots,X_m\}$.
\end{definition}

Let $\Omega^0_\Lambda$ be the set of height functions $\eta\in
\mathbb{Z} ^{\mathbb{Z}^2}$ with $\eta_x=0$ for every $x\notin
\Lambda$. Given $\eta\in\Omega^0_\Lambda$ one can define the associated
cluster configuration $ \{X_1,\ldots,X_m\}$, $m=m(\eta)$, as follows.
Fix an arbitrary orientation of the edges $e=(x,y)$ of $\mathbb{Z}^2$.
Let $\mathcal{S}=\mathcal{S}(\eta)$ be the collection of all dual edges
$e'$ such that the gradient of $\eta$
along the edge $e=(x,y)$ crossing $e'$ satisfies $h_e:=
\eta_y-\eta_x\neq0$. Let $\gamma_1,\ldots,\gamma_m$ denote the
connected components of $\mathcal{S}$. For each $j=1,\ldots,m$ let
$X_j=(\gamma_j,\{ h_e\})$ denote the associated cluster, where
$\{h_e\}$ denotes the collection of gradients of $\eta$ along edges $e$
that cross a dual edge $e'\in\gamma_j$.

We define $\mathcal{L}(\Lambda)=\bigcup_{\eta\in\Omega_\Lambda
^0}\{
X_1,\ldots,X_m\}$ to be the collection of all possible clusters. Two
clusters $X,X'$ are called compatible, in symbols $X\sim X'$, iff
$\gamma\cup\gamma'$ is not a connected set of dual edges, where
$\gamma,\gamma'$ denote the geometric part of $X,X'$, respectively.
Otherwise, $X,X'$ are said to be incompatible, in symbols $X\nsim X'$.
Also, let $\mathcal{D}(\Lambda)$ denote the collection of all pairwise
compatible cluster configurations, that is, of configurations
$\{X_1,\ldots,X_m\}$ with $m\geq0$, $X_i\in\mathcal{L}(\Lambda)$ for
$i=1,\ldots,m$ and $X_i$ compatible with $X_j$ for every $i\ne j$.
Then, one has
%
%
\begin{equation}
\label{Zclusters} \quad Z_\Lambda=\sum_{\{X_1,\ldots,X_m\}\in\mathcal
{D}(\Lambda)}\prod
_{j=1}^m\rho(X_j),\qquad
\rho(X_j)=\exp{ \biggl[-\beta\sum_{e}|h_e|\biggr]},
\end{equation}
where the sum over $e$ extends over all $|\gamma_j|$ edges $e$ which cross
a dual edge $e'\in\gamma_j$.

\subsection{Peierls' estimate}
As above, $\mathcal{S}$ denotes the random set of dual edges crossing
a nonzero gradient.\vadjust{\goodbreak}

%
%
\begin{lemma}
\label{lem-peierls_c} There exists $\beta_0>0$ such that for all
$\beta
\geq\beta_0$, for all finite connected $\Lambda\subset\mathbb{Z}^2$,
and all set $V$ of dual edges,
%
%
\begin{equation}
\label{peierls_c1} \hat\pi{}^0_\Lambda(
\mathcal{S} \supset V )\leq e^{-(\beta
-\beta_0)|V|}.
\end{equation}
\end{lemma}

\begin{pf}
We suppose that $V$ is connected, since the general case follows by
a~standard generalization. Let $e'$ be a dual edge in $V$ and let
$\mathcal{S}_0$ denote the largest connected component of $\mathcal{S}$
containing $e'$. Then
\[
\hat\pi{}^0_\Lambda(V\subset\mathcal{S} )\leq\sum
_{S\dvtx
S\supset V}\hat\pi{}^0_\Lambda(
\mathcal{S}_0=S ),
\]
where the sum is over all connected sets $S$ of dual edges, such that
$S\supset V$. Any \mbox{$\eta\in\Omega_\Lambda^0$} such that $\mathcal
{S}_0=S$ corresponds to a cluster configuration
$\{X_S,X_1,\ldots, X_m\}\in\mathcal{D}(\Lambda)$, where $X_S$ is
a cluster of the form $X_S=(S,h_1,\ldots,\break h_{|S|})$. For a~fixed~$S$ one
has $\sum_{h_1\neq0,\ldots,h_{|S|}\neq0}\rho(X_S)\leq
(4e^{-\beta})^{|S|}$, if $\beta\geq\log2$. Therefore,
using~(\ref{Zclusters}), neglecting the constraints on $X_S$, one has
%
%
\begin{equation}
\label{pei1} \hat\pi{}^0_\Lambda(\mathcal{S}_0=S
) \leq \bigl(4e^{-\beta} \bigr)^{|S|}.
\end{equation}
Summing over all $S$ as above and estimating by $C^\ell$ the number of
connected $S\ni e'$ with $|S|=\ell$ gives
\[
\hat\pi{}^0_\Lambda(V\subset\mathcal{S} )\leq\sum
_{\ell\geq|V|} \bigl(4Ce^{-\beta} \bigr)^{\ell} \leq
e^{-(\beta-\beta_0)|V|}.
\]\upqed
\end{pf}

\subsection{Cluster expansion}
We shall use a standard expansion for partition functions, adapted
from~\cite{KP,DKS}.
For $U$ \textit{a subset of the inner boundary of $\Lambda$}, we write
$Z_{\Lambda,U}$ for the partition function with the sum over $\eta$
restricted to those $\eta\in\Omega_\Lambda^0$ such that $\eta
_x\geq0$
for all $x\in U$. One can write $Z_{\Lambda,U}$ similarly to (\ref
{Zclusters}): one defines $\mathcal{L}(\Lambda,U)$ as the set of all
possible clusters (arising from height configurations respecting the
positivity constraint in $U$) and $\mathcal{D}(\Lambda,U)$ as the
collection of all pairwise compatible cluster configurations, with the
same notion of compatibility as before. Then, one can check that
(\ref{Zclusters}) holds for $Z_{\Lambda,U}$, just with
$\mathcal{D}(\Lambda)$ replaced by $\mathcal{D}(\Lambda,U)$. We
emphasize that it is here that one uses that $U$ is a subset of the
inner boundary of $\Lambda$: the identity would be false, for example,
if $U$ were the whole $\Lambda$.

%
\begin{lemma}
\label{lem-dks} There exists $\beta_0$ such that for all $\beta\geq
\beta_0$, for all finite connected $\Lambda\subset\mathbb{Z}^2$ and any
subset of its inner boundary $U\subset\Lambda$,
%
%
\begin{equation}
\label{zetaexp} \log Z_{\Lambda,U}=\sum_{V\subset\Lambda}
\varphi_{U}(V),
\end{equation}
where the potentials $\varphi_{U}(V)$ satisfy 
\begin{longlist}[(iii)]
\item[(i)] $\varphi_{U}(V)=0$ if $V$ is not connected.

\item[(ii)] $\varphi_{U}(V)=\varphi_0(V)$ if $\operatorname{dist}(V,
U)\neq0$, for some shift invariant potential
$V\mapsto\varphi_0(V)$, that is,
\[
\varphi_0(V)= \varphi_0(V+x)\qquad\forall x\in
\mathbb{Z}^2.
\]\vadjust{\goodbreak}

\item[(iii)] There exists a constant $\beta_0>0$ such that
\[
\sup_{\Lambda\supset V}\sup_{U} \bigl|
\varphi_{U}(V) \bigr| \leq\exp \bigl(-(\beta-\beta_0) \,d(V)
\bigr),
\]
where $d(V)$ is the cardinality of the smallest connected set of
bonds of $\mathbb{Z}^2$ containing all the boundary bonds of $V$ (i.e.,
bonds connecting $V$ to $V^c$).
\end{longlist}
\end{lemma}

\begin{pf}
We shall apply the main theorem from \cite{KP}.
Following \cite{KP}, we define $\mathcal{C}$ as the set of all cluster
configurations $C$ that cannot be decomposed as $C=C_1\cup C_2$ with
two nonempty cluster configurations $C_1,C_2$ such that $\{X_1,X_2\}$
is compatible for every $X_1\in C_1$ and $X_2\in C_2$. For a cluster
$X=(\gamma,h_1,\ldots,h_{|\gamma|})$, define the function
$a(X)=\lambda|X|$, where $\lambda>0$ is to be specified later and
$|X|:=\sum_{i=1}^{|\gamma|}|h_i|$. Note that, for a fixed
$X=(\gamma,h_1,\ldots,h_{|\gamma|})$, one has
%
%
\begin{eqnarray}\label{kp2}
\sum_{X'\dvtx X'\nsim X}e^{2\lambda|X'|}\rho
\bigl(X' \bigr)&\leq& \sum_{\gamma'\dvtx \gamma\cup\gamma'\ \mathrm{connected}}c(\beta,
\lambda)e^{-(\beta-2\lambda)|\gamma'|}
\nonumber\\[-8pt]\\[-8pt]
&\leq& c'(\beta,\lambda)|\gamma|,\nonumber
\end{eqnarray}
where, for example, $c(\beta,\lambda)=2(1-e^{-(\beta-2\lambda)})^{-1}$
and $c'(\beta,\lambda)=3e^{-(\beta-2\lambda)}c(\beta,\lambda)$.
So if
$\beta\geq2\lambda+1$, and $\lambda$ is larger than some absolute value
$\lambda_0$, (\ref{kp2}) implies
%
%
\begin{equation}
\label{kp3} \sum_{X'\dvtx X'\nsim X}e^{2\lambda|X'|}\rho
\bigl(X' \bigr) \leq a(X).
\end{equation}
Equation (\ref{kp3}) corresponds to equation~(1) in \cite{KP}. The main
theorem there then allows one to write
%
%
\begin{equation}
\label{kp4} \log Z_{\Lambda,U} = \sum_{C\dvtx C\subset\mathcal
{L}(\Lambda,U)}
\Phi(C)
\end{equation}
for a function $\Phi$ on cluster configurations satisfying $\Phi(C)=0$
if $C\notin\mathcal{C}$ and
%
%
\begin{equation}
\label{kp5} \sum_{C\dvtx C\nsim X} \bigl|\Phi(C) \bigr|e^{a(C)}
\leq a(X)
\end{equation}
for every cluster $X$, where $a(C):=\sum_{i=1}^n a(X_i)$ if
$C=\{X_1,\ldots,X_n\}$ and the notation $C\nsim X$ indicates that
$X_i\nsim X$ for some $X_i\in C$. The potentials $\Phi$ depend on $U$
but for lightness of notations we keep this implicit. Taking $X$ to be
the elementary unit square cluster such that $|X|=4$ in (\ref{kp5}) one
finds in particular that for every cluster configuration $C$ one has
%
%
\begin{equation}
\label{kp6} \bigl|\Phi(C) \bigr|\leq4e^{-a(C)}.
\end{equation}
To write $Z_{\Lambda,U}$ as in (\ref{zetaexp}), we follow \cite{DKS},
Section~3.9. For any cluster configuration $C\in\mathcal{C}$,
$C=\{X_1,\ldots,X_n\}$ with $X_i=(\gamma_i,h_1,\ldots,h_{|\gamma_i|})$,
write $C_g$ for the geometric part of $C$, that is,
$C_g=(\gamma_1,\ldots,\gamma_n)$. For any
$G:=(\gamma_1,\ldots,\gamma_n)$, define
\[
\psi(G)=\sum_{C\in\mathcal{C}\dvtx C_g= G}\Phi(C).
\]
Using (\ref{kp6}), if $\lambda\geq\lambda_0$, one has
%
%
\begin{equation}
\label{kp67} \bigl|\psi(G)\bigr|\leq4e^{-(\lambda/2)\sum_i|\gamma_i|}.
\end{equation}
Finally, set
%
%
\begin{equation}
\label{kp7} \varphi_{U}(V)=\mathop{\sum
_{G=(\gamma_1,\ldots,\gamma_n)\dvtx}}_{\bigcup_i
\operatorname{Int}\gamma_i = V} \psi(G).
\end{equation}
From (\ref{kp4}), one obtains the expansion (\ref{zetaexp}). The
properties (i)--(ii)--(iii) follow as in \cite{DKS} from an explicit
representation of the function $\Phi(C)$, and from the exponential
decay (\ref{kp67}).
\end{pf}

\subsection{Distribution of an open contour}
Here we apply the expansion of Lemma~\ref{lem-dks} to derive an
expression for the law of an open contour in the presence of a stepped
boundary condition. Suppose a finite connected
$\Lambda\subset\mathbb{Z} ^2$ is given together with a boundary
condition $\xi$ with values in $\{0,1\}$ and such that it induces a
\textit{unique} open $1$-contour $\gamma$. If $\gamma=\Gamma$, for some
connected set of dual edges $\Gamma$, then $\Lambda$ is partitioned
into two connected regions $\Lambda_+,\Lambda_-$ separated by $\Gamma$.
Moreover,
%
%
\begin{equation}
\label{agamma} \hat\pi{}^{\xi}_\Lambda(\gamma=\Gamma)\propto
e^{-\beta|\Gamma|}Z_{\Lambda_-,\Delta_\Gamma^-}Z_{\Lambda
_+,\Delta_\Gamma^+},
\end{equation}
where $\Delta_\Gamma^\pm$ are the sets defined after (\ref{Blemma2}),
and we use the notation $Z_{\Lambda,U}$ that was introduced in
Lemma~\ref{lem-dks}.
By expanding the partition functions as in (\ref{zetaexp}), and
retaining only terms depending on $\Gamma$, one finds that
%
%
\begin{equation}
\label{aexpan1} \hat\pi{}^{\xi}_\Lambda(\gamma=\Gamma) \propto
\exp \bigl(-\beta|\Gamma| + \Psi_{\Lambda
}(\Gamma) \bigr),
\end{equation}
where
\[
\Psi_\Lambda(\Gamma)=- \mathop{\sum_{V\subset\Lambda}}_{V\cap
\Gamma\neq\varnothing}
\varphi_0(V)+ \mathop{\sum_{V\subset\Lambda_+}}_{V\cap\Gamma\neq
\varnothing}
\varphi_{\Delta_\Gamma^+}(V)+ \mathop{\sum_{V\subset\Lambda
_-}}_{V\cap\Gamma\neq\varnothing}
\varphi_{\Delta_\Gamma^-}(V).
\]
Here, the notation $V\cap\Gamma\neq\varnothing$ simply means that
$V\cap(\Delta_\Gamma^-\cup\Delta_\Gamma^+)\neq\varnothing$. It is
convenient to rewrite this expansion in the form
%
%
\begin{equation}
\label{expansi} \Psi_\Lambda(\Gamma)= \mathop{\sum
_{V\subset\Lambda}}_{V\cap\Gamma\neq\varnothing} \phi(V;\Gamma),
\end{equation}
where the ``decorations'' $\{\phi(V;\Gamma)\}_{V\subset\Lambda}$
satisfy (cf. Lemma~\ref{lem-dks}):
\begin{longlist}[(iii)]
\item[(i)] $\phi(V;\Gamma)=0$ if $V$ is not connected.

\item[(ii)]
$\phi$ is shift invariant in the sense that
\[
\phi(V;\Gamma)=\phi(V+x;\Gamma+x)\qquad\forall x\in\mathbb{Z}^2.
\]

\item[(iii)] There exists a constant $\beta_0>0$ such that
\[
\sup_{\Gamma} \bigl|\phi(V;\Gamma) \bigr|\leq\exp \bigl(-(
\beta-\beta_0) \,d(V) \bigr),\vadjust{\goodbreak}
\]
where $d(V)$ is defined as in Lemma~\ref{lem-dks}.\vadjust{\goodbreak}
\end{longlist}
It is standard to check that these properties imply the existence of
$\beta_0$ such that, for any $\beta\geq\beta_0$ and any $\ell\geq1$,
%
%
\begin{equation}
\label{eq-uffa} \mathop{\sum_{V\ni0}}_{d(V)\geq\ell}
\sup_\Gamma \bigl|\phi(V;\Gamma) \bigr|\leq\exp \bigl(-(
\beta- \beta_0)\ell \bigr).
\end{equation}

\section{Large deviations of the contour}\label{largecontour}

We begin by fixing some notation. $\mathbb{S}=\{1,2,\ldots,L\}\times
\mathbb{Z}$ will denote the infinite vertical strip of width $L$. We
denote by $A,B$ the points of coordinates $(1,0)$ and $(L,0)$,
respectively. The $L\times L$ square with corners $A,B,C,D$, where
$C=(L,L)$ and $D=(1,L)$ will be denoted by $Q_L$. Next we fix an open
contour $\Gamma_*$ inside $Q_L$ joining $A$ with $B$ with the property
that $\Gamma_*$ stays above the line at zero height and does not reach
height $L^{\delta}$, for some $\delta<1/2$ that in the applications
will be taken small. The region inside $\mathbb{S}$ above $\Gamma_*$ is
denoted by $\Lambda$ and we set $Q=Q_L\cap\Lambda$.
We let $\nu_Q$ be the law of the open $1$-contour $\Gamma$ joining $A$
with $B$, for the SOS model without floor/ceiling in $Q$, with $1$ b.c.
along $\Gamma_*$ and $0$ b.c. otherwise. We know that $\nu_Q$ can be
written as
%
%
\begin{equation}
\label{expan1} \nu_Q(\Gamma) \propto\exp \bigl(-\beta|\Gamma| +
\Psi_Q(\Gamma) \bigr),
\end{equation}
where $\Psi_Q$ is the function appearing in (\ref{expansi}).
Fix $a\in(1/2,1)$ and $\ell\in[L^{a}, L/\log(L)^2]$ and define
$E_\ell$ as the event that the path $\Gamma$ reaches height $\ell$ (note
that $\ell\gg L^\delta$).
%
%
\begin{proposition}\label{dkspro} Uniformly in $\Gamma_*$ as above,
there exists $\beta_0$ independent of $(\ell, L)$ such that, for all
$\beta>\beta_0$ and all $L$ large enough
\[
\nu_Q(E_\ell)\leq c' \exp \bigl(-c
\ell^{2}/L \bigr)
\]
for some constants $c,c'$.
\end{proposition}
%

\subsection{Proof of Proposition~\texorpdfstring{\protect\ref{dkspro}}{B.1}}
As a first preliminary step we remove the dependence on the upper
boundary of $Q$. Let $\nu_\Lambda$ be the probability distribution on
contours in $\Lambda$ joining $A,B$ given by
\[
\nu_\Lambda(\Gamma)\propto\exp \bigl(-\beta|\Gamma| + \Psi
_\Lambda(\Gamma) \bigr).
\]

%
%
\begin{claim} For any $\beta$ large enough
\[
\nu_Q(E_\ell)\leq3 \nu_\Lambda(E_\ell)
+ e^{-cL}
\]
for a suitable constant $c=c(\beta)$.
\end{claim}

\begin{pf}
Let $\mathcal{G}_0$ and $\mathcal{G}_1$ be the set of contours which
stay below height $L-\log(L)^2$ and height $L$, respectively. Then
%
%
\begin{equation}
\label{eq-LD1} \nu_Q(E_\ell)\leq\frac{\nu_\Lambda(\chi_{E_\ell
} \chi
_{\mathcal{G}
_0}
e^{\Delta
\Psi(\Gamma)})} {
\nu_\Lambda(\chi_{\mathcal{G}_0} e^{\Delta\Psi(\Gamma)})} +
\frac{\nu_\Lambda(\chi_{E_\ell} (1-\chi_{\mathcal{G}_0})
\chi_{\mathcal{G}
_1}
e^{\Delta
\Psi(\Gamma)})}{\nu_\Lambda(\chi_{\mathcal{G}_0} e^{\Delta\Psi
(\Gamma)})},
\end{equation}
where
\[
\Delta\Psi(\Gamma)= \Psi_Q(\Gamma)-\Psi_\Lambda(\Gamma)
\]
and the inequality sign comes from restricting the average in the
denominator from contours in $\mathcal{G}_1$ to contours $\mathcal
{G}_0$. Since $\min_{\Gamma\in\mathcal{G}_0^c}|\Gamma|-\min_{\Gamma\in
\mathcal{G}_0}|\Gamma|\geq L$, a standard Peierls argument shows that
\[
\nu_\Lambda \bigl(\mathcal{G}_0^c \bigr)\leq
e^{-(
\beta-\beta_0)L}
\]
for some $\beta_0$. Moreover, thanks to the exponential decay of the
decorations (\ref{expansi}),
\[
\bigl|\Delta\Psi(\Gamma) \bigr|\leq\tfrac12\qquad\forall\Gamma\in\mathcal{G}_0
\]
and
\[
\bigl|\Delta\Psi(\Gamma)\bigr|\leq e^{-(\beta-\beta
_0)}L\qquad\forall\Gamma \in\mathcal{G}_1.
\]
Therefore, the first term in the r.h.s. of (\ref{eq-LD1}) is smaller
than $3 \nu_\Lambda(E_\ell)$ while the second one is bounded from above
by $e^{-c L}$ for some constant $c=c(\beta)$ diverging as
$\beta\to\infty$.
\end{pf}
Back to the proof of the proposition: since the event $E_\ell$ is
increasing, we can change the b.c. from $0$ to $1$ along the lateral
sides of $\Lambda$, up to height $(3/4)\ell$ [note that in this
situation the endpoints of $\Gamma$ are shifted upward by $(3/4)\ell
$]. We still call~$\nu_\Lambda$ the measure of $\Gamma$ in this
situation. Again by FKG, we have
%
%
\begin{equation}
\label{jafaremo} \nu_\Lambda(E_\ell)\leq\frac{\nu_\Lambda
(E_\ell;G^+)}{\nu
_\Lambda(G^+)},
\end{equation}
where $G^+$ is the increasing event that $\Gamma$ stays at distance at
least $L^\varepsilon$ from $\Gamma_*$, for some small but positive constant
$\varepsilon$.

Thanks to the decay properties of the potentials $\phi(V;\Gamma)$, for
every $\Gamma$ in $G^+$ we can replace $\Psi_Q(\Gamma)$ with $\Psi
_{\mathbb
S}(\Gamma)$, up to a negligible error term. Then, the ratio
(\ref{jafaremo}) equals
%
%
\begin{equation}
\label{eq-jafaremo2} \bigl(1+o(1) \bigr) \frac{\nu_{\mathbb
S}(E_\ell;G^+)}{\nu_{\mathbb
S}(G^+)}\leq \bigl(1+o(1)
\bigr) \frac{\nu_{\mathbb S}(E_\ell)}{\nu_{\mathbb S}(G^+)}
\end{equation}
with $\nu_{\mathbb S}$ the measure of the contour for SOS in the strip
$\mathbb S$, with $0$ b.c. above $A+(0,(3/4)\ell), B+(0,(3/4)\ell)$ and
$1$ b.c. below it.

Note that, if the complementary event $(G^+)^c$ happens, it means that
the contour $\Gamma$ makes a downward deviation at least $(1/2)\ell$
from its natural height $(3/4)\ell$. Therefore, $\nu_{\mathbb
S}(G^+)\geq1-\nu_{\mathbb S}(E_\ell)$. As a consequence, it suffices
to prove:
%
%
\begin{claim}
For any $c>0$ and all $\beta$ large enough depending on $c$, one has $
\nu_{\mathbb S}(E_\ell)\leq e^{-c \ell^2/L} $.\vadjust{\goodbreak}
\end{claim}

\begin{pf} 
By translation invariance, we can assume that the path $\Gamma$ starts
at $A=(1,0)$, ends at $B=(L,0)$ and replace $E_\ell$ with $E_{\ell/4}$.
We would like to appeal to the results of Section~4.15 in \cite{DKS}.
For this purpose we need to tackle the fact that decorations touching
the boundary of $\mathbb{S}$ may behave differently from decorations
inside $\mathbb{S}$. We therefore introduce a third (!) probability
measure on \textit{all} paths $\Gamma$ between $A$ and $B$ (even those
going outside $\mathbb{S}$) denoted\vspace*{1pt} simply by
$\mathbb{P}(\cdot)$ and corresponding to the weight $e^{-\beta|\Gamma|
+ \Psi_{\mathbb{Z} ^2}(\Gamma)}$ and we write
\[
\nu_{\mathbb S}(\Gamma\mbox{ reaches height } \ell/4)
= \frac{\mathbb{E} (\Gamma\mbox{ reaches height } \ell/4;
\Gamma\in\mathbb{S};
e^{\Psi_\mathbb{S}(\Gamma)-\Psi_{\mathbb{Z}^2}(\Gamma)}
)}{\mathbb{E} (\Gamma\in\mathbb{S};
e^{\Psi_\mathbb{S}(\Gamma)-\Psi_{\mathbb{Z}^2}(\Gamma)} )}.
\]
Using Section~4.15 of \cite{DKS}, we get that
\[
\mathbb{P}(\Gamma\mbox{ reaches height } \ell/4)\leq e^{-c\ell^2/L}
\]
for some constant $c>0$. On the other hand, (\ref{eq-uffa}) implies
that
\[
\bigl|\Psi_{\mathbb{Z}^2}(\Gamma)-\Psi_\mathbb{S}(\Gamma)\bigr|\leq
e^{-(\beta-\beta
_0)} \bigl|b\in\Gamma\dvtx\operatorname{dist} \bigl(b,
\mathbb{S}^c \bigr) \leq\log(L)^2 \bigr|,
\]
which implies
\[
\mathbb{E} \bigl(e^{2 |\Psi_{\mathbb{Z}^2}(\Gamma)-\Psi_\mathbb
{S}(\Gamma)| } \bigr)\leq e^{c'\log(L)^2}
\]
for some constant $c'$, thanks to the large deviation results of
Section~4.15 in \cite{DKS}. Finally, thanks to~(\ref{eq-uffa}) and
Proposition~4.18 in \cite{DKS}, if $\mathcal{C}$ is the cigar-shaped
region with tips at $A,B$ defined by
\[
\mathcal{C}= \biggl\{(x_1,x_2)\in\mathbb{R}^2
\dvtx|x_2|\leq \biggl(\frac
{x_1(L-x_1)}{L} \biggr)^{1/2+\kappa}
\biggr\},
\]
\begin{eqnarray*}
\mathbb{E} \bigl(\Gamma\in\mathbb{S}; e^{\Psi_\mathbb
{S}(\Gamma)-\Psi_{\mathbb{Z}^2}(\Gamma)} \bigr)&\geq&\mathbb{E}
\bigl(\Gamma\in\mathcal{C}; e^{\Psi_\mathbb{S}(\Gamma)-\Psi
_{\mathbb{Z}^2}(\Gamma)} \bigr)
\\
&\geq& C \mathbb{P}(\Gamma\in\mathcal{C})
\geq e^{-c'' \log
(L)^{2/\kappa}},
\end{eqnarray*}
where the constant $C$ is a deterministic lower bound on $
e^{\Psi_\mathbb{S}(\Gamma)-\Psi_{\mathbb{Z}^2}(\Gamma)}$ for
$\Gamma\in\mathcal{C} $ obtained again using (\ref{eq-uffa}).
\end{pf}

\section{}\label{poschains}
Fix $a\in(0,1)$. Let $R$ be the intersection between
$\mathbb{Z}^2$ and a $L\times L^a$ rectangle, not necessarily parallel
to the
coordinate axes. Let $\Lambda\subset\mathbb{Z}^2$ be such that
$\Lambda$
contains $R$ and is contained in some $2L\times2L$ square.
A subset $\mathcal{C}=\{x_1,x_2,\ldots,x_k\}$ of $R$ will be called a
\textit{spanning chain} if
\begin{longlist}[(ii)]
\item[(i)] $d(x_i,x_{i+1})=1$ for all $i=1,\ldots, k-1$;

\item[(ii)] $\mathcal{C}$ connects the two shorter sides of $R$.
\end{longlist}
For a fixed $n\geq0$ let $\mathcal{F}_+$ ($\mathcal{F}_-$) be the event
that there exists a spanning chain where the surface height is at least
(at most) $n$.\vadjust{\goodbreak}

%
%
\begin{lemma}\label{quasi-rect.1}
For $\beta$ large enough 
%
%
\begin{equation}
\label{eq-222} \Pi_\Lambda^{n} \bigl(\mathcal{F}^c_+
\bigr)\leq\Pi_\Lambda^{n,f} \bigl(\mathcal{F}^c_+
\bigr)\leq e^{-cL^a}.
\end{equation}
Assume moreover that $\ell(\Lambda) e^{-4\beta(n+1)}\leq1$, where
$\ell(\Lambda)$ is the shortest side of the smallest rectangle
containing $\Lambda$. Then
%
%
\begin{equation}
\label{eq-27} \Pi_\Lambda^{n,f} \bigl(\mathcal{F}^c_-
\bigr)\leq\Pi_\Lambda^{n} \bigl(\mathcal{F}^c_-
\bigr)\leq e^{-cL^a}.
\end{equation}
\end{lemma}
Here, as in 
(\ref{eq-10}), the field is $f=\frac{1}L\sum_{y\in\Lambda} f_y$.

\begin{pf*}{Proof of Lemma~\ref{quasi-rect.1}}
We first observe that $\mathcal{F}_+$($\mathcal{F}_-$) is an increasing
(decreasing) event and therefore the first inequalities in
(\ref{eq-222}), (\ref{eq-27}) are trivial because the fields $f_y$ are
decreasing functions. Again by monotonicity $\Pi_\Lambda
^{n,f}(\mathcal{F}^c_+)$ is bounded from above by the probability
w.r.t. the SOS model $\hat\pi{}_\Lambda^{n,f}$ \textit{without} floor.
Moreover, $\mathcal{F}^c_+$($\mathcal{F}^c_-$) occurs iff there exists
a *-chain $\{y_1,\ldots,y_n\}$ connecting the two long opposite sides
of $R$ and such that $\eta_{y_i}\leq n-1$ ($\eta_{y_i}\geq n+1$) for
all $i$. In turn that implies the existence of a $(n-1)$-contour
($(n+1)$-contour) larger than~$L^{a}$.

As in the proof of Lemma~\ref{l:contourBound},
we get that
\[
\hat\pi{}_\Lambda^{n,f} \bigl(\gamma\mbox{ is a } (n-1)
\mbox{-contour } \bigr)\leq e^{-\beta
|\gamma|+ 1/L \sum_{x\in\Lambda_\gamma}\|f_x\|_{\infty}}\leq e^{-\beta/2|\gamma|},
\]
where in the last inequality we used $\|f_x\|_{\infty}\leq e^{-c
\beta}$ together with $|\Lambda_\gamma|\leq2L|\gamma|$. Simple counting
of $\gamma$ finishes the proof of (\ref{eq-222}).

Similarly, it follows from Proposition~\ref{p-starBound} that
\[
\Pi_\Lambda^{n} \bigl(\gamma\mbox{ is a $(n+1)$-contour}
\bigr) \leq e^{-\beta|\gamma|
+ C e^{-4\beta(n+1)}|\Lambda_\gamma|}.
\]

Isoperimetry gives $|\Lambda_\gamma|\leq\ell(\Lambda)|\gamma|$ which,
combined with the assumption $\ell(\Lambda) e^{-4\beta(n+1)}\leq1$,
implies
\[
\Pi_\Lambda^{n} \bigl(\gamma\mbox{ is a $(n+1)$-contour}
\bigr) \leq e^{-\beta/2|\gamma|}
\]
and the proof of (\ref{eq-27}) follows.
\end{pf*}

\section{Proof of inequalities (\texorpdfstring{\lowercase{\protect\ref{eq-13}}}{6.22}) and (\texorpdfstring{\lowercase{\protect\ref{eq-9.1}}}{6.29})}\label{app-13}

%
%
%

\begin{pf*}{Proof of Lemma~\ref{lemma:snello}}
Fix $\ell> 2$. By removing the field $f$ of (\ref{eq-10}), we only
increase the surface so to bound the\vspace*{-1pt} probability of the
decreasing event $G^+_\ell$ we may work in the model
$\pi^{H'}_{\Lambda_L}$, that is, the standard SOS model on $\Lambda_L$
with no field and floor/ceiling at height $0/n^+=\log L$. If $G^+_\ell$
fails, then for some $R\in\mathcal{R}$ we can find contours
$\{(\gamma_s,h_s)\}_{s\in\mathscr{S}}$ satisfying the hypothesis of
Proposition~\ref{p-starBound} each with $|\Lambda_{\gamma_s}\cap
R|\geq1$ and $h_s \geq H'+1$ such that $\bigcap_{s\in\mathscr{S}}
\mathscr{C}_{\gamma_s,h}$ holds, and that
%
%
\begin{equation}
\label{e-diagonalCovering} \sum_{s\in\mathscr{S}} |
\Lambda_{\gamma_s}\cap R | \geq\ell L/2. 
\end{equation}
For a given ensemble of contours as above, define a sequence of subsets
$W_i\subseteq\Lambda$~by
\begin{eqnarray*}
W_0 &=& \Lambda,
\\
W_i &=& \bigcup_{s\in\mathscr{S}\dvtx h_s=H'+i}
\Lambda_{\gamma_s}\qquad\mbox{for $i=1,\ldots,n^+-H'$}.
\end{eqnarray*}
Let $\mathcal{A}$ denote the set of all possible such collections of contours
$\{(\gamma_s,h_s)\}$. For all $i\geq0$, let
\[
a_i = |W_i \cap R|.
\]
Let $\mathcal{A}(\vec{a}) = \mathcal{A}(a_1,a_2,\ldots,a_{n^+-H'})$
denote all collections of contours matching a~given sequence of
$a_i$'s. Then~(\ref{e-diagonalCovering}) is equivalent to
$\sum_{i\geq1} a_i \geq\ell L/2$. Since $R$~is a diagonal,
$|\Lambda_{\gamma_s}\cap R| \leq\frac14 | \gamma_s| $ and so
%
%
\begin{equation}
\label{e-diagonalIso} \sum_{s\in\mathscr{S}} |
\gamma_s| \geq4\sum_{i=1}^{n^+ - H'}
a_i.
\end{equation}
For any $W\subseteq\Lambda$ let
\[
\mathcal{B}(W)=\sum_{(\gamma_1',\gamma_2',\ldots,\gamma'_m)} e^{-(\beta
/4) |\gamma|},
\]
where the sum is over all collections of edge-disjoint contours
$\{\gamma'_i\}$, with pairwise disjoint interiors
$\{\Lambda_{\gamma_i'}\}$ all contained in $W$ and with
$|\Lambda_{\gamma_i'}\cap R|\geq1$ for\vspace*{-1pt} all $i$. Any such
contour must have an edge adjacent to some $v\in W\cap R$ in the dual
lattice~$\mathbb{Z}^{2*}$. If $e$ is an edge in the dual lattice
$\mathbb {Z}^{2*}$, then there are at most $3^n$ contours $\gamma$ of
length $n$ containing $e$. Hence for large enough $\beta$,
\[
\mathcal{B}(W) \leq \Biggl(1+\sum_{n=4}^\infty3^n
e^{-(\beta/4) n
} \Biggr)^{4|W\cap R|} \leq\exp \bigl(|W\cap R| \bigr),
\]
since each contour must contain at least one edge adjacent to some
$v\in W\cap R$ in the dual lattice $\mathbb{Z}^{2*}$, there are at most
$4|W\cap R|$ such edges and the contours are edge-disjoint.

Now for $\{(\gamma_s,h_s)\}_{s\in\mathscr{S}} \in\mathcal{A}(\vec{a})$
by Proposition~\ref{p-starBound}
we have that
%
%
\begin{eqnarray}
\pi^{H'}_{\Lambda_L} \biggl( \bigcap
_{s\in\mathscr{S}}\mathscr{C}_{\gamma_s,h_s} \biggr)
&\leq&\exp \biggl(\sum_{s\in\mathscr{S}} \bigl(-\beta|\gamma_s|+C_0
|\Lambda_{\gamma_s}|e^{-4\beta h_s} \bigr) \biggr)
\nonumber\\[-9pt]\\[-9pt]
&\leq&\exp \biggl( -\frac34 \beta\sum_{s\in\mathscr{S}} |
\gamma_s| \biggr) \nonumber
\end{eqnarray}
for any $\beta\geq C_0$ since $e^{-4\beta h_s} \leq e^{-4\beta
(H+1)}\leq L^{-1}$ and $|\Lambda_{\gamma_s}|\leq(L/4)|\gamma_s|$ for
any contour $\gamma_s$ by the isoperimetric inequality in $\mathbb{Z}^2$.
Substituting this expression, we have that
\begin{eqnarray*}
&& \sum_{\{(\gamma_s,h_s)\}_{s\in\mathscr{S}} \in\mathcal{A}(\vec
{a})}\pi^{H'}_{\Lambda_L}\biggl( \bigcap_{s\in\mathscr{S}}\mathscr{C}_{\gamma_s,h_s}
\biggr)
\\
&&\qquad \leq\sum_{\{(\gamma_s,h_s)\}_{s\in\mathscr{S}}\in
\mathcal{A}(\vec{a}) } \exp \biggl( -\frac34\beta
\sum_{s\in\mathscr{S}} |\gamma_s| \biggr)
\\
&&\qquad \leq\exp \Biggl( -2\beta\sum_{i=1}^{n^+-H'}
a_i \Biggr) \sum_{\{
(\gamma_s,h_s)\}_{s\in\mathscr{S}}\in\mathcal{A}(\vec{a})} \exp \biggl( -(\beta/4) \sum_{s\in\mathscr{S}} |\gamma_s| \biggr),
\end{eqnarray*}
where the last inequality is by~(\ref{e-diagonalIso}). This in turn is
at most
\begin{eqnarray*}
\exp \Biggl( -2\beta\sum_{i=1}^{n^+-H'}
a_i \Biggr) \prod_{i=1}^{n^+-H'}
\mathcal{B} (W_{i-1} )
&\leq&\exp \Biggl( -2\beta\sum
_{i=1}^{n^+-H'} a_i + \sum_{i=1}^{n^+-H'} a_{i-1} \Biggr)
\\
&\leq&\exp \biggl( -\frac34\beta\ell L \biggr).
\end{eqnarray*}
%
The final inequality follows for large $\beta$ since $a_0=L$.
As there are at most $L^{n^+-H'} \leq L^{\log L}$ choices for
$\vec{a}=(a_1,a_2,\ldots,a_{n^+-H'})$, we have that
\begin{eqnarray*}
\pi^{H'}_{\Lambda_L} \bigl( G_\ell^+ \bigr) &\geq& 1-
\sum_{\vec
{a}} \sum_{\{(\gamma_s,h_s)\}_{s\in\mathscr{S}} \in\mathcal
{A}(\vec{a})}\pi
^{H'}_{\Lambda_L} \biggl( \bigcap_{s\in\mathscr{S}}
\mathscr{C}_{\gamma_s,h_s} \biggr)
\\
&\geq&1- L^{\log L} \exp \biggl( -\frac34\beta\ell L \biggr)
\\
&\geq& 1- \exp \biggl( -\frac\beta2 \ell L \biggr)
\end{eqnarray*}
for large $\beta$, as required.
\end{pf*}

Equation~(\ref{eq-13}) follows similarly with a simpler proof.
\end{appendix}

\section*{Acknowledgments}
We are grateful to S. Shlosman for valuable discussions. This work was
initiated while P.~Caputo, F.~Martinelli and\break  F.~L.~Toninelli were
visiting the Theory Group of Microsoft Research, Redmond. They thank
the Theory Group for its hospitality and for creating a stimulating
research environment.



%

\printaddresses

\end{document}